\documentstyle[11pt,amsmath,amssymb,graphicx,xcolor,hyperref,url,array,multirow,caption,enumitem,natbib,float,tikz,pgfplots,fontenc,amsthm,amsfonts]{article}

\renewcommand\baselinestretch{1.1}

\def\style{authordate4}

\usetikzlibrary{arrows,calc,angles,positioning,intersections,quotes,decorations.markings,backgrounds,patterns}
\usetikzlibrary{decorations.pathreplacing}
\pgfplotsset{compat=1.11}
\tikzset{
	mynode/.style={fill,circle,inner sep=1pt,outer sep=0pt}
}

\newcommand{\mbf}[1]{\mbox{\boldmath $#1$}}
\newcommand{\ba}{{\mbf \beta}}

{\catcode `\@=11 \global\let\AddToReset=\@addtoreset}
\AddToReset{equation}{section}

\AddToReset{Theorem}{section}

\newtheorem{cor}{Corollary}[section]
\newtheorem{lem}{Lemma}[section]
\newtheorem{rem}{Remark}[section]
\newtheorem{thm}{Theorem}[section]

\newtheorem{exmp}{Example}[section]

\newcommand{\cA}{{\cal A}}
\newcommand{\cB}{{\cal B}}
\newcommand{\cC}{{\cal C}}

\newcommand{\cE}{{\cal E}}

\newcommand{\cG}{{\cal G}}

\newcommand{\cK}{{\cal K}}

\newcommand{\cN}{{\cal N}}

\newcommand{\cT}{{\cal T}}
\newcommand{\cU}{{\cal U}}

\def\ba{\begin{array}}
	\def\bc{\begin{center}}
		\def\bd{\begin{description}}
			\def\be{\begin{enumerate}}
				\def\ea{\end{array}}
			\def\ec{\end{center}}
		\def\ed{\end{description}}
	\def\edt{\end{document}}
\def\ee{\end{enumerate}}
\def\ben{\begin{equation}}
\def\benn{\begin{equation*}}
\def\een{\end{equation}}
\def\eenn{\end{equation*}}
\def\benr{\begin{eqnarray}}
\def\eenr{\end{eqnarray}}
\def\benrr{\begin{eqnarray*}}
\def\eenrr{\end{eqnarray*}}

\def\al{\alpha}
\def\b{\beta}

\def\d{\dot}

\def\edt{\end{document}}

\def\g{\gamma}

\def\h{\hat}
\def\ka{\kappa}

\def\iny{\infty}
\def\ka{\kappa}

\def\la{\lambda}
\def\lel{\label}

\def\noi{\noindent}

\def\nn{\nonumber}

\def\r{\ref}

\def\si{\sigma}

\def\sti{\sum_{i=1}^n}

\def\Si{\Sigma}

\def\vep{\varepsilon}

\def\vs{\vskip}

\def\A{{\mathbb A}}
\def\R{{\mathbb R}}

\def\z{\zeta}

\DeclareMathOperator*{\argmin}{arg\,min}

\begin{document}

\bc
{\Large {\bf Detection and estimation of parameters in high dimensional multiple change point regression models via $\ell_1\big/\ell_0$ regularization and discrete optimization}}\\[.2cm]
Abhishek Kaul$^{a,}$\footnote{{\it Address for correspondence}: Abhishek Kaul, Department of Mathematics and Statistics,  Washington State University, Pullman, WA 99164, USA. Email: abhishek.kaul@wsu.edu.}, Venkata K. Jandhyala$^a$, Stergios B. Fotopoulos$^b$\\[.1cm]

$^a$Department of Mathematics and Statistics, $^b$Department of Finance and Management Science,  
Washington State University, Pullman, WA 99164, USA.

%\footnote{Research supported in part by the NSF-DMS grant 1205271.}\\[.4cm]
%Abhishek Kaul

%Michigan State University
\ec
\vs .1in
{\renewcommand{\baselinestretch}{1}
	\begin{abstract}
	Binary segmentation, which is sequential in nature is thus far the most widely used method for identifying multiple change points in statistical models. Here we propose a top down  methodology called {\it arbitrary segmentation} that proceeds in a conceptually reverse manner. We begin with an arbitrary superset of the parametric space of the change points, and locate unknown change points by suitably filtering this space down. Critically, we reframe the problem as that of variable selection in the change point parameters, this enables the filtering down process to be achieved in a single step with the aid of an $\ell_0$ regularization, thus avoiding the sequentiality of binary segmentation. We study this method under a high dimensional multiple change point linear regression model and show that rates convergence of the error in the regression and change point estimates are near optimal. We propose a simulated annealing (SA) approach to implement a key finite state space discrete optimization that arises in our method. Theoretical results are numerically supported via simulations. The proposed method is shown to possess the ability to agnostically detect the `no change' scenario. Furthermore, its computational complexity is of order $O(Np^2)+{\rm SA},$ where SA is the cost of a SA optimization on a $N$ (no. of change points) dimensional grid. Thus, the proposed methodology is significantly more computationally efficient than existing approaches. Finally, our theoretical results are obtained under weaker model conditions than those assumed in the current literature.
\end{abstract} }
\noi {\bf Keywords:   Multiple change points, Multiphase regression, High dimensional regression, $\ell_1,$ $\ell_0$ regularization, Simulated annealing.}

\section{Introduction}\label{sec:intro}

High dimensional regression models that allow vastly larger number of parameters $p$ than the sample size $n,$ have found applications in many fields of scientific inquiry such as genomics, social networking, empirical economics, finance among many others. This has led to a rapid development of statistical literature investigating methods capable of analyzing such models and data sets. One of the most successful methods for analysing high dimensional regression models has been the Lasso, which is based on the least squares loss and $\ell_1$ regularization (\cite{tibshirani1996regression}). Innumerable investigations have since been carried out to study the behavior of the Lasso estimator and its various modifications in many different settings (see e.g., \cite{zou2006adaptive}; \cite{zhao2006model}; \cite{bickel2009simultaneous}; \cite{belloni2011square} \cite{belloni2017pivotal}; \cite{kaul2014lasso}, \cite{kaul2015weighted} and the references therein). For a general overview on the developments of Lasso and its variants we refer to the monograph of \cite{buhlmann2011statistics} and the review article of \cite{tibshirani2011regression}. All aforementioned articles provide results in a regression setting where the parameters are dynamically stable. In contrast, multiphase/change point regression models provide a dynamic setting in which regression parameters are allowed to switch values based on a change inducing variable or in a time ordered sense. Such models allow for a greater versatility in modelling data, especially in a high dimensional setting. In many experiments, the estimated locations of change points may reveal additional critical information of interest.

In the past few years several articles have studied high dimensional change point models in an `only means' setup. In this setting, change points are characterized with respect to dynamic mean vectors of time ordered random vectors, where the dimension of the observation vector may be larger than the number of observations (\cite{cho2015multiple}, \cite{fryzlewicz2014wild},and \cite{wang2018high}; among others).  Another context in which high dimensional change point models have been investigated is that of a dynamic covariance structure which is related to the study of evolving networks (\cite{roy2017change},\cite{gibberd2017multiple},  \cite{atchade2017scalable}; among others). In contrast, change point methods for high dimensional linear regression models have received much less attention and only a select few articles have considered this problem in the recent literature.

In this paper, we consider a high dimensional multiphase (change point) regression model given by,
\benr\label{cp}
y_i=\sum_{j=1}^{N+1} x_i^T\b_{(j-1)}^0{\bf 1}[\tau_{j-1}^0 < w_i \le\tau_j^0] +\vep_i,\quad i=1,..,n,
\eenr
where $N\ge 0,$ $\tau_0^0=-\iny,$ $\tau_{N+1}^0=\iny,$ and ${\bf 1}[\cdot]$ represents the indicator function. The components of the change point parameter vector are assumed to be $\tau^0=(\tau^0_1,...,\tau^0_N)^T\in\bar\R^{N},$ $\bar\R=\R\cup\{-\iny\}$ such that $\tau_{j-1}^0\le \tau_j^0,$ $j=1,...,N.$ First note that when $\tau_0=\tau_1^0=...=\tau_N^0=-\iny,$ model (\ref{cp}) reduces to an ordinary linear regression model without change points. This case where all change points are at negative infinity characterizes the case of `no change', and in the following we refer to this case as $N=0.$ On the other hand we characterize the case of one or more change points, $N\ge 1,$ as when the components of $\tau^0$ are distinct and finite, i.e., $-\iny<\tau^0_1<...<\tau^0_N.$ The observed variables in model (\ref{cp}) are the response $y_i\in\R,$ the $p$-dimensional predictors $x_i\in\R^p,$ and change inducing variable $w_i\in\R.$ The parameters of interest are the number of change points $N\in\{0,1,2,...\},$ the change point parameter vector $\tau^0\in\bar\R^{N},$ and the regression parameters $\b_{(j)}^0\in\R^p,$ $j=0,...,N.$ The change points $\tau_{j}^0$ $j=1,..,N$ represents threshold values of the variable $w$ subsequent to which the regression parameter changes from its current value $\b_{(j-1)}$ to a new value $\b_{(j)}.$ Furthermore, we let $p>>n,$ so that model (\r{cp}) corresponds to a high dimensional setting.

In the classical setting with a fixed number of parameters and $n\to\iny,$ change point regression models such as (\r{cp}) have been extensively investigated, albeit a large proportion of this literature is developed in the case with only a single change point. The works of \cite{hinkley1970inference}, \cite{hinkley1972time},
\cite{jandhyala1997iterated}, \cite{bai1997estimation}, \cite{jandhyala1999capturing}, and \cite{jandhyala2013inference}, investigate the setting where parameters are assumed to
change at certain unknown time points of the sampling period. On the other hand, the works of \cite{hinkley1969inference}, \cite{koul2002asymptotics}, and
\cite{koul2003asymptotics}  study the setting where the change point is formulated based on one or more covariate thresholds. In the literature, the latter approach is typically referred to as two-phase or multiphase regression, however it is also common to broadly call both as change point regression models.

The literature on regularized estimation in change point regression models is very sparse. Models similar to (\r{cp}) with a single change point have been studied by \cite{kaul2018parameter}, \cite{lee2016lasso}, and \cite{lee2018} in the high dimensional setting. The case of multiple change points is investigated in \cite{ciuperca2014model}, \cite{zhang2015multiple}, , \cite{jin2016consistent} and \cite{leonardi2016computationally}. Amongst these articles the first three consider the fixed $p$ setting, whereas the last article considers the high dimensional setting as is also the case in this paper. The article of \cite{leonardi2016computationally} proposes a binary segmentation approach for the recovery of change points of the regression model, where change points are searched for, and then added to the set of all change points one by one. In the context of change point parameters, this binary segmentation approach can be viewed as the counterpart of step-up regression where parameters are included sequentially. It is important to\ remember that in the current high dimensional setting, in order to search for a single change point for each segment, the approach of \cite{leonardi2016computationally} requires $O\big(n{\rm  Lasso}(n)\big)$ computations, where Lasso$(n)$ represents the computational cost of running one Lasso optimization with a sample size $n.$ In fact the authors show that the overall computation cost of their approach is of the order O\big($n\log(n)$Lasso$(n)$\big).

In contrast, our approach proceeds in a conceptually reverse manner. The method that we propose can be viewed in a sense as the counterpart of step-down regression for the change point parameters. We begin with a superset of the parameteric space of the unknown change points and filter this space down to identify the unknown change points, following which we estimate the regression parameters. Critically, the `stepping down' process in our methodology can be carried out in a single step via a $\ell_0$ regularization. We achieve this by converting the problem of recovery of change points to a variable selection problem in the change point parameters. This conversion of the change point estimation problem to a variable selection problem in turn relies on initial regression estimates. The second main novelty of this manuscript is to show that, initial regression estimates that are much slower than optimal in rates of convergence can be utilized to obtain change point estimates that are themselves near optimal in rates of convergence. In other words, our setup constitutes a rare statistical scenario where relatively `poor' estimates of some parameters of a model can be utilized to obtain near optimal estimates of other parameters of the model.

The proposed method circumvents the sequential approach of binary segmentation for the recovery of change points. Consequently, the method requires only Lasso$(n)+$SA computations for the identification and recovery of change points, where SA represents the computational cost of a simulated annealing optimization which is typically very efficient. The simulated annealing algorithm is used to implement a key discrete optimization over an $O(N)$ dimensional space that arises in our methodology due to the use of an $\ell_0$ regularization. Thus our approach is far more efficient than any existing comparable methodology for high dimensional change point regression models. Being based on a $\ell_0$ regularization, our approach also provides the ability to detect the case of $N=0,$ where an ordinary linear regression without change points is more appropriate. In comparison, binary segmentation approaches typically require the existence of at least one change point. Finally, we also note that our analysis requires significantly weaker assumptions than those currently assumed in the literature. Further comparisons of our method, assumptions and results with the existing literature are made in Section \ref{sec:model}.

The remainder of this article is organized as follows. Section \ref{sec:model} provides the proposed methodology and technical assumptions required for the theoretical analysis. Section \ref{sec:Main} provides the main theoretical results regarding the performance of the proposed methodology. Section \ref{sec:implement} discusses the implementation of the proposed method and a simulated annealing approach for the implementation of a key step of our method. This section also provides numerical results on the finite sample performance of our method. The proofs of all main results are provided in Appendix A of the Supplementary materials of this article. Some additional technical results and lemma's are provided in Appendix B of the Supplementary materials.

\vspace{1mm}
\textbf{\textit{Notations}}: We conclude this section with a short note on the notations used in this paper. Throughout the paper, for any vector $\delta\in \R^p,$ $\|\delta\|_0$ represents the number of non-zero components in $\delta,$ and  $\|\delta\|_1$ and $\|\delta\|_2$ represent the usual $1$-norm and Euclidean norm, respectively. The norm $\|\delta\|_{\iny}$ represents the usual sup norm, i.e., the maximum of absolute values of all elements.  For any set of indices $S\subseteq \{1,....,p\},$ let $\delta_S=(\delta_j)_{j\in S}$ represent a sub-vector of $\delta$ containing components corresponding to the indices in $S.$ Also, let $|S|$ represent the cardinality of the set $S.$ The notation ${\bf 1}[\cdotp]$ represents the usual indicator function. We denote by $\Phi (\cdotp)$ the cdf of $w_i'$s and  let $d(\tau_a,\tau_b)=P(\tau_a< w_i\le \tau_b)=\Phi(\tau_b)-\Phi(\tau_a),$ $\tau_a\le\tau_b\in\bar\R,$ clearly, $d(\tau_a,\tau_b)=0 \Leftrightarrow  \tau_a=\tau_b.$  We represent by $\bar\R=\R\cup\{-\iny\}$ as the extended Euclidean space, with only the left closure point included. We shall also use the notation $a\vee b=\max\{a,b\},$ and $a\wedge b=\min\{a,b\},$ $a,b\in\R.$ The notation $c_u,c_m$ is used to represent generic constants that may be different from one line to the next. Here, $0<c_u<\iny$ represent universal constants, whereas $0<c_m<\iny$ are constants that depend on model parameters such as variance parameters of underlying distributions. Lastly, $0<c_1,c_2<\iny$ are also generic constants that may depend on both $c_u,$ and $c_m.$

\section{Methodology and Related Work}\label{sec:model}

\subsection{Proposed methodology}

For any $\tau_a,\tau_b\in\bar\R,$ $\tau_a\le\tau_b$ and any $\g\in\R^p,$ define the segmentwise least squares loss as,
\benr\label{def:qstar}
Q^*(\g,\tau_a,\tau_b):=\sum_{i=1}^{n} (y_i-x_i^T\g)^2{\bf 1}[\tau_a< w_i\le \tau_b],
\eenr
where the indicator function ${\bf 1}[\tau_a<w_i\le \tau_b]=0$\footnote{This is a slight misuse of notation, and is only used for simpler exposition. To be notationally precise, this term should be ${\bf 1}[w_i\le \tau_b]-{\bf 1}[w_i\le \tau_a]$, for $\tau_a\le \tau_b$} if $\tau_a=\tau_b.$ For any $\check N\ge 1,$ let $\check \tau:=(\check\tau_1,...,\check\tau_{\check N})^T\in\bar\R^{\check N},$ be any vector such that $\check \tau_{j-1}\le\check\tau_j,$ $j=1,...,\check N,$ $\check \tau_0=-\iny.$ Also, let $\check\tau_{N+1}=\iny$ such that $(\check\tau_0,\check\tau^T,\check\tau_{\check N+1})^T$ forms a partition of $\bar\R.$ Additionally, for any sequence of vectors $\al_{(j)}\in\R^p,$ $j=0,...,\check N$ denote by $\al=(\al_{(0)}^T,....,\al_{(\check N)}^T)^T\in\R^{{(\check N+1)} p}$ as the concatenation of all $\al_{(j)}s.$ Then define the total least squares loss evaluated at $(\check N,\al,\check\tau)$ as,
\benr
Q(\check N,\al,\check\tau):= \frac{1}{n}\sum_{j=1}^{\check N+1} Q^*(\al_{(j-1)},\check\tau_{j-1},\check\tau_{j}).\nn
\eenr
Next, for any $\tau\in\R^{\check N},$ define $\h\cT(\tau)\subseteq\{1,...,\check N\}$ as the set of indices of distinct and finite components of $\tau,$ i.e.,
\benr\label{def:hatct}
\h\cT(\tau)=\big\{j\in \{1,...,\check N\}\,\,;\,\, \tau_{j-1}\ne \tau_j,\,\,\tau_j\ne-\iny\big\},
\eenr
where $\tau_0=-\iny.$ Under these notations and the model (\ref{cp}), we propose estimators for the number of change points, locations of change points and the regression coefficients respectively. These estimators are stated in the following as two algorithms each consisting of two steps. The first algorithm is designed to recover the number and locations of change points of the model (\ref{cp}) and the second to recover the corresponding regression coefficients. All technical assumptions required for the theoretical validity of the proposed estimates are stated in Section \ref{subsec:assumptions}.

\begin{figure}[H]
	\noi\rule{\textwidth}{0.5pt}
	
	\vspace{-3mm}
	\begin{flushleft}{\bf Algorithm 1:} Estimation of number and locations of change point(s)\end{flushleft}
	\vspace{-4.75mm}
	\noi\rule{\textwidth}{0.5pt}
	
	\vspace{-1mm}
	\begin{itemize}[wide, labelwidth=!, labelindent=0pt]
		\item[{\bf Step 0}:] ({\bf Initializing step}) Choose any $\check N\ge 1\vee N,$ and any vector $\check\tau=(\check\tau_1,...,\check\tau_{\check N})^T\in\R^{\check N},$ satisfying {\bf Condition A}. Compute initial regression estimates $\h\al_{(j)},$ for each $j=0,...,\check N,$
		\vspace{-1mm}
		\benr
		\h\al_{(j)}=\argmin_{\al\in\R^{p}}\Big\{\frac{1}{n}Q^*(\al, \check\tau_{j-1},\check\tau_j)+\la_0\|\alpha\|_1 \Big\},\qquad \la_0>0\nn
		\eenr
		\vspace{-4mm}
		\item[{\bf Step 1}:]  Update $\check \tau\in\R^{\check N}$ to obtain estimate $\hat \tau\in\bar\R^{\check N}$, where\footnotemark,
		\vspace{-2mm}
		\benr
		\h\tau=\argmin_{\substack{\tau\in \bar\R^{\check N};\\ \tau_{j-1}\le\tau_j\, \forall\, j}}\Big\{Q(\check N, \h\al, \tau)+ \mu \sum_{j=1}^{\check N}\|d(\tau_{j-1},\tau_{j})\|_0 \Big\},\qquad \mu>0\nn
		\eenr
		\vspace{-4mm}
		
		\noi Let $\h\cT:=\h\cT(\h\tau),$ and update the estimated number of change points to $\tilde N = |\h \cT|,$ and recover the corresponding locations of change points as the subset $\tilde\tau=\h\tau_{\h\cT}\in\R^{\tilde N}.$
	\end{itemize}
	
	\vspace{-3mm}
	\noi\rule{\textwidth}{0.5pt}
\end{figure}

\footnotetext{Note that while the initializing $\check \tau$ in {\bf Step 0} is chosen in $\R^{\check N},$ however the optimization in {\bf Step 1} is performed over the extended Euclidean space $\bar\R^{\check N}.$}

Algorithm 1 begins ({\bf Step 0}) with a nearly arbitrary partition $\check\tau$ in a superset $\R^{\check N}$ of the parametric space $\R^{N}$ of the unknown change points. The simple update in {\bf Step 1} of Algorithm 1 recovers the number of change points $\tilde N,$ and the corresponding locations $\tilde\tau.$ There are two main novelties of Algorithm 1. First, instead of searching for change points sequentially, Algorithm 1 searches for them in a larger parametric space by reframing the problem as one of variable selection. Here the selection is in terms of differences between adjacent $\tau_j$'s, i.e., the $\ell_0$ regularization in {\bf Step 1} is forcing these adjacent components to collapse towards each other. This regularization can be viewed as a $\ell_0$ version of the total variation penalty on the components $\tau.$ The second main novelty is that in order to achieve this conversion to a variable selection problem, we use a nearly arbitrary partition that serves as an initial rough guess. It shall become theoretically and empirically apparent in the following that the estimates obtained in {\bf Step 1} are robust against this initial partition, i.e., nearly any arbitrarily chosen partition in {\bf Step 0} shall yield near optimal estimates from {\bf Step 1}. The underlying working mechanism of Algorithm 1 is illustrated in Figure \ref{fig:actionmechanism}.

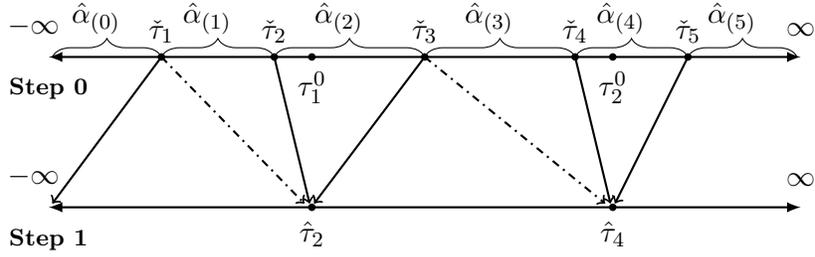
\begin{figure}[]
	\centering{
		\begin{tikzpicture}
		\draw[black,thick,latex-latex] (0,0) -- (10,0)
		node[pos=0,label=below:\textcolor{black}{\footnotesize{\bf Step 0}}]{}
		node[pos=-0.02,label=above:\textcolor{black}{$-\iny$}]{}
		node[pos=0.35,mynode,fill=black,label=below:\textcolor{black}{$\tau_1^0$}]{}
		node[pos=0.75,mynode,fill=black,label=below:\textcolor{black}{$\tau_2^0$}]{}
		node[pos=0.15,mynode,fill=black,label=above:\textcolor{black}{$\check\tau_1$}]{}
		node[pos=0.30,mynode,fill=black,label=above:\textcolor{black}{$\check\tau_2$}]{}
		node[pos=0.50,mynode,fill=black,label=above:\textcolor{black}{$\check\tau_3$}]{}
		node[pos=0.70,mynode,fill=black,label=above:\textcolor{black}{$\check\tau_4$}]{}
		node[pos=0.85,mynode,fill=black,label=above:\textcolor{black}{$\check\tau_5$}]{}
		node[pos=1,label=above:\textcolor{black}{$\iny$}]{};
		\draw[decorate,decoration={brace,amplitude=7pt}] (0.05,0) --node[above left=5pt and -9pt]{$\h\al_{(0)}$} (1.5,0);
		\draw[decorate,decoration={brace,amplitude=7pt}] (1.5,0) --node[above left=5pt and -9pt]{$\h\al_{(1)}$} (3,0);
		\draw[decorate,decoration={brace,amplitude=7pt}] (3,0) --node[above left=5pt and -9pt]{$\h\al_{(2)}$} (5,0);
		\draw[decorate,decoration={brace,amplitude=7pt}] (5,0) --node[above left=5pt and -9pt]{$\h\al_{(3)}$} (7,0);
		\draw[decorate,decoration={brace,amplitude=7pt}] (7,0) --node[above left=5pt and -9pt]{$\h\al_{(4)}$} (8.5,0);
		\draw[decorate,decoration={brace,amplitude=7pt}] (8.5,0) --node[above left=5pt and -9pt]{$\h\al_{(5)}$} (9.95,0);
		\draw[black,thick,latex-latex] (0,-2) -- (10,-2)
		node[pos=0,label=below:\textcolor{black}{\footnotesize{\bf Step 1}}]{}
		node[pos=-0.02,label=above:\textcolor{black}{$-\iny$}]{}
		node[pos=0.35,mynode,fill=black,label=below:\textcolor{black}{$\h\tau_2$}]{}
		node[pos=0.75,mynode,fill=black,label=below:\textcolor{black}{$\h\tau_4$}]{}
		node[pos=1,label=above:\textcolor{black}{$\iny$}]{};
		\draw[->,thick] (1.5,0)--(0.05,-1.95);
		\draw[->,thick] (3,0)--(3.47,-1.95);
		\draw[->,thick] (5,0)--(3.53,-1.95);
		\draw[->,thick] (7,0)--(7.47,-1.95);
		\draw[->,thick] (8.5,0)--(7.53,-1.95);
		\draw [->,thick,dash dot] (1.5,0) -- (3.40,-1.95);
		\draw [->,thick,dash dot] (5,0) -- (7.40,-1.95);
		\end{tikzpicture}
		\caption{\footnotesize{Illustration of the working mechanism of Algorithm 1 with $N=2,$ $\check N=5$: Initializing with a nearly arbitrary partition $\check \tau=(\check \tau_1,...,\check \tau_5)^T$ and corresponding regression coefficient estimates $\h\al_{(j)},$ $j=0,..,5,$ the algorithm converges to the unknown number of change points and the locations of the distinct change points lie in an optimal neighborhood of the unknown change points. Specifically, the components nearest to an unknown $\tau^0_j$ shall converge toward it, and all remaining components converge either to the previous or the next component.}}
		\label{fig:actionmechanism}}
\end{figure}

Next, we propose Algorithm 2 for the estimation of regression parameter vectors of model (\ref{cp}). This algorithm utilizes the estimated number ($\tilde N$) and locations ($\tilde\tau$) of change points from Algorithm 1, to obtain coefficient estimates on the corresponding partition yielded by $\tilde\tau.$ Note that when $\tilde N=0$ from Algorithm 1, then Algorithm 2 is equivalent to implementing the ordinary Lasso on the data $(x_i,y_i),$ $i=1,...,n.$

\begin{figure}[]
	\noi\rule{\textwidth}{0.5pt}
	
	\vspace{-3mm}
	\begin{flushleft}{\bf Algorithm 2:} Estimation of regression coefficients\end{flushleft}
	\vspace{-4.75mm}
	\noi\rule{\textwidth}{0.5pt}
	
	\vspace{-1mm}
	\begin{itemize}[wide, labelwidth=!, labelindent=0pt]
		\item[{\bf Step 0}:] Compute $\tilde N$ and $\tilde \tau$ from Algorithm 1. If $\tilde N=0,$ then fit a linear regression model without change points via Lasso.
		
		\item[{\bf Step 1}:] If  $\tilde N\ge 1,$ then for each $j=0,...,\tilde N,$ update regression parameter estimates,
		\vspace{-1mm}
		\benr
		\tilde\al_{(j)}= \argmin_{\al\in\R^{p}}\Big\{\frac{1}{n}Q^*(\al, \tilde\tau_{j-1}, \tilde\tau_j)+\la_{1j}\|\al\|_1\Big\},\qquad \la_{1j}>0.\nn
		\eenr
	\end{itemize}
	
	\vspace{-3mm}
	\noi\rule{\textwidth}{0.5pt}
\end{figure}

The main theoretical contribution of this manuscript is to show that the proposed methodology consistently recovers the unknown number of change points, and yields estimates of the locations of change points and that of regression coefficient vectors that are near optimal in their rates of convergence. Specifically, under suitable conditions, we shall derive the following relations that hold for $n$ sufficiently large with probability at least $1-c_1(1\vee N)\exp(-c_2\log p).$

\benr\label{eq:ratesintro}
&(i)&\,\, {\rm When}\,\, N\ge 0,\quad \tilde N= N,\\
&(ii)&\,\, {\rm When}\,\, N\ge 1,\quad \sum_{j=1}^{N} d(\tilde \tau_j,\tau^0_j)\le c_uc_mN\frac{s\log p}{n},\,\,{\rm and} \nn\\
&(iii)& \sum_{j=0}^N\|\tilde\al_{(j)}-\b^0_{(j)}\|_q \le c_uc_m Ns^{\frac{1}{q}}\sqrt{\frac{\log p}{n}},\quad q=1,2.\nn
\eenr

In an ordinary high dimensional linear regression model without change points, it has been shown that the optimal rate of convergence for a regression parameter vector estimate is $\sqrt{s\log p/n}$ under the $\ell_2$ norm (\cite{ye2010rate}, \cite{raskutti2011minimax}, \cite{belloni2017linear}). Also, the rate of convergence of the change point estimates in (\ref{eq:ratesintro}) matches the fastest available in the literature, see, e.g., (\cite{kaul2018parameter}, \cite{lee2016lasso}, \cite{leonardi2016computationally}, \cite{lee2018}).

The result in (\ref{eq:ratesintro}) is quite surprising since the estimates $\tilde N$ and $\tilde\tau$ are computed based on initial regression coefficient estimates $\h\al$ from {\bf Step 0} of Algorithm 1. These regression estimates may not be anywhere near optimal in their rate of convergence, since these are in turn computed based on a nearly arbitrary partition of the support of $w.$ Despite these rough regression estimates, we can prove that {\bf Step 1} of Algorithm 1 identifies the change points correctly and provides estimates that are indeed near optimal in their rate of convergence.

Next, we discuss two immediate concerns that may arise to the reader regarding Algorithm 1. First,  how stringent is Condition A on the initializers $\check N$ and $\check \tau$ in  {\bf Step 0} of this algorithm. This condition on the initializers is infact very mild. For Algorithm 1, where $\check N$ is user chosen, nearly any arbitrarily chosen partition with a large enough $\check N\ge 1\vee N,$ satisfies this condition. Other requirements of this condition are only meant to remove pathological cases, such as when all components of $\check \tau$ are closely clustered together or are concentrated at one end of the support of $w.$ From a practical perspective, an equally spaced large enough partition, as expected, works well in all empirically examined cases.

A second concern that may arise regarding Algorithm 1 is whether the optimization of {\bf Step 1} of these methods is computationally feasible. At first impression, the optimization of {\bf Step 1} does indeed appear to be computationally intensive given that it is a nonsmooth, nonconvex optimization (with no apparent convex relaxations), and with potentially multiple global optimums. However the following observations shall serve to erase this impression. Note that although this optimization of {\bf Step 1} is over the extended Euclidean space $\bar\R^{\check N},$ however, the loss function $Q(\check N, \h\al, \cdotp)$ is a step function over the finite grid $\tau\in\{-\iny,w_1,....,w_n\}^{\check N},$ with step change occurring at grid points on this $\check N$ dimensional grid (see, Figure \ref{fig:stepbehave.zeroinflate} in Section \ref{sec:implement} for an illustration of this behavior). Additionally, the $\ell_0$ norm term in the optimization of {\bf Step 1} is either $0$ or $1,$ based on whether $\tau_{j-1}$ and $\tau_{j}$ are equal or unequal respectively, in other words the distance between $\tau_{j-1}$ and $\tau_{j}$ does not influence the value of $\ell_0$ norm (note that this will not be true if in {\bf Step 1}, the $\ell_1$ norm is considered in place of the $\ell_0$ norm). These two observations together imply that any global optimum achieved in the extended Euclidean space $\bar\R^{\check N}$ is also attained at some $\check N$ dimensional point on the grid $\{-\iny,w_1,....,w_n\}^{\check N}.$ In other words, the optimization in {\bf Step 1} is reduced to a discrete optimization on a finite state space (the total number of possible states being $(n+1)^{\check N}$). In view of these observations, the optimization in {\bf Step 1} is reminiscent of the well known {\it travelling salesman problem}, and correspondingly can be solved efficiently using a simulated annealing approach. Additionally, since simulated annealing is not a gradient based approach, it is capable of easily handling a $\ell_0$ penalty. A detailed discussion of the implementation of Algorithm 1 is provided in Section \ref{sec:implement}. We also note that {\bf Step 0} of Algorithm 1, and {\bf Step 1} of Algorithm 2 are Lasso$\big(n, p\big)$ estimates. Thereby, these two steps are efficiently implementable using any one of the several available methods in the literature, for e.g. coordinate or gradient descent algorithms, see, e.g. \cite{hastie2015statistical}  or via interior point methods for linear optimization under second order conic constraints, see, e.g., \cite{koenker2014convex}.

Finally, we conclude this section by also emphasizing the computational efficiency of the proposed Algorithm 1. First note that {\bf Step 0} of Algorithm 1 are $(\check N+1)$ computations of Lasso$(n,p)$ estimates. It is also known that computational complexity of most algorithms for the Lasso optimization scales like $O(p^2).$ As briefly described above, {\bf Step 1} of Algorithm 1 shall be implemented via simulated annealing (SA) over a $\check N$ dimensional grid. SA optimizations are known to be very efficient for a large class of problems and can ordinarily be accomplished in a time scaling of order $O(\check N^4),$ see, e.g. \cite{sasaki1987optimization}. We also mention that in the worst case the complexity of SA can also be exponential, depending on the optimization under consideration. However, in our study the SA optimization of {\bf Step 1} is empirically observed to be well behaved and carried out with a cheap computational cost. Thus, assuming $\check N\le c(1\vee N),$ $c\ge 1,$ the overall complexity of Algorithm 1 is $O\big((1\vee N)p^2\big)+SA.$ Thereby, this algorithm is far more computationally efficient than any comparable existing method for the estimation of parameters of model (\r{cp}). To see this, compare the above complexity to the binary segmentation approach proposed in \cite{leonardi2016computationally}. They show that the their method is implementable with $O(n\log n)$Lasso$\big(n,(N+1)p\Big)$ computations with the aid of dynamic programming, thus the procedure effectively yields a time scaling of $O(nNp^2).$

\subsection{Assumptions}\label{subsec:assumptions}
In this subsection we state all necessary conditions and technical assumptions under which the results of this article are derived.

\vspace{2mm}
\noi{\bf Condition A (requirements of initializer):}\\~{\it
	(i) The initializing vector $\check\tau=(\check\tau_1,....,\check\tau_{\check N})^T\in \R^{\check N}$ is such that $\check N$ is larger than the true number of change points, i.e., $\check N \ge 1\vee N,$ and $\check N\le c_u(1\vee N),$ $c_u\ge 1.$\\~
	(ii) All initial change points are sufficiently separated, i.e., $d(\check\tau_{j-1},\check\tau_{j})> l_{\min}>0,$  for all $j=1,...,\check N+1,$ for some positive sequence $l_{\min},$ where $\check\tau_{0},\check\tau_{\check N+1}$ denote $-\iny$ and $\iny$ respectively.\\~
	(iii) Let ${\check u}_n=1\vee c_u(1/n)^{1/k},$ for some constants $k\in[1,\iny),$ and $c_u>0.$ Then assume that there exists a subset $\cT:=\{m_0,m_1,...,m_N,m_{N+1}\}\subseteq\{0,1,2,....\check N+1\}$ such that $m_0=0,$ $m_{N+1}=\check N+1$ and $\max_{1\le j\le N}d(\check\tau_{m_j}, \tau^0_{j})\le {\check u}_n.$ When $N=0,$ define $\cT:=\{m_0,m_{N+1}\}.$
}

\vspace{2mm}
As briefly discussed earlier, Condition A is a mild assumption on the initializers. Roughly speaking, this condition requires the initial change point vector to be a large enough partition of $\R$ where the components of this initializing vector are sufficiently separated from each other. Also, this condition requires that at least one initial change point lies in some fractional neighborhood of each unknown change point. The condition $d(\check\tau_{m_j}, \tau^0_{j})\le {\check u}_n,$ $j=1,...,N$ is very mild since the constant $k$ can be arbitrarily large\footnote{The constant $k$ can be arbitrarily large as long as the rate conditions of Condition B and C are satisfied.}. In one of our main results, we shall show that despite the initializers $\check\tau_{m_j}$ lying in an arbitrary fractional neighborhood of $\tau_j^0,$ the updated change point estimate $\tilde\tau$ satisfies $\|\tilde\tau-\tau^0\|_1\le Ns\log p/n,$ with high probability. Note that, the localization error bound of $\tilde\tau$ is free of $k.$ This condition is similar to Condition I assumed in \cite{kaul2018parameter}, we refer to that article for further insights on this condition. Here we also state that implementation of the proposed methodology does not require prior knowledge of $k.$

It is observed a large enough grid of equally separated initial change points works well in nearly all empirically examined cases. The term $l_{\min}$ in Condition A(i) is allowed to potentially decrease to zero with $n,$ however this dependence is suppressed for clarity of exposition. The rate at which such a convergence of $l_{\min}$ is allowed also depends on other model parameters, and is explicitly stated in Condition B(iii).

We can now define the $\check N$-dimensional parameter that {\bf Step 1} of Algorithm 1 is designed to recover in place of the $N$-dimensional $\tau^0.$ For this purpose first define
a set of indices $\cT^*=\{h_0,h_1,...,h_N,h_{N+1}\},$ where $\{h_1,...,h_{N}\}\subseteq\{1,...,\check N\},$ $h_0=0,$ and $h_{N+1}=\check N+1.$ Consider any $\tau\in\R^{\check N},$ satisfying $\tau_1\le\tau_2\le...\le\tau_{\check N},$ then the components $h_j$ of $\cT^*$ are defined as,
\benr\label{def:settstar}
\hspace{1cm} h_{j}=\min\Big\{k\in\{1,...,\check N\};\,\, k>m_{j-1};,\,\, \tau_{k}=\tau_{m_j}\Big\},\,\, {\rm for}\,\, j=1,...,N.
\eenr
where $\cT=\{m_0,m_1...,m_{\check N+1}\}$ is given in Condition A. Clearly, the construction of these indices depend on the choice of the vector $\tau\in\bar\R^{\check N},$ and the set $\cT.$ In the following this dependence is notationally suppressed for clarity of exposition, and is to be understood implicitly. The indices $h_j$'s are meant to capture the first index $j$ after $m_{j-1}$ for which $\tau_j=\tau_{m_j}.$ In the case where the chosen $\tau\in\bar\R^{\check N}$ is such that $\tau_{m_j-1}<\tau_{m_j},$ for all $j=1,..,N,$ then the set $\cT^*=\cT.$

Now define the vector $\tau^*=(\tau^*_1,...,\tau^*_{\check N})^T\in\bar\R^{\check N}$ such that,
\benr\label{def:taustar}
\hspace{1cm}\tau^*_{h_j}=\tau^0_j,\,\,j=1,...,N\,\,{\rm and}\,\, \tau_k^*=\tau^0_j,\,\, h_j \le k \le m_j,\,\,j=1,...N,
\eenr
and finally, $\tau^*_j=\tau^*_{j-1}$ for all remaining indices in the set $\cT^{*c},$ where, as before, $\tau^*_0=-\iny.$ Note that under the above definition of $\tau^*,$ the subset of finite and distinct components of this vector is exactly the unknown parameter vector $\tau^0,$ however the orientation or order in which they appear in this $\check N$-dimensional vector  may be different depending on the set $\cT^*$ and the chosen $\tau,$ as well as the set $\cT.$

To see the need for this non-traditional construction of the target parameter $\tau^*,$ first recall that the objective function in {\bf Step 1} of Algorithm 1 is non-convex, and consequently may have multiple global optimums. Now consider any such global optimum $\h\tau=(\h\tau_1,...,\h\tau_{\check N})^T\in\bar\R^{\check N}$ and let the orientation index set $\cT^*$ be defined in accordance with this optimum $\h\tau,$ together with the index set $\cT.$ Then, the $\tau^*$ constructed with corresponding $\cT^*,$ forms the target vector that $\h\tau$ is infact approximating. The non-traditional aspect of this construction is that the subset of finite and distinct components of the target vector $\tau^*$ is exactly the parameter vector $\tau^0$ and thus fixed and non-random. However the orientation in which the components may appear depend on the optimizer itself, i.e. this orientation may be random and depends on the orientation in which the global optimum is achieved. In the following we illustrate the construction of $\tau^*$ using a concrete example.

\begin{exmp} Consider the model (\ref{cp}) with $N=3,$ and $\tau^0=(-1,0,1).$ Let Algorithm 1 be initialized with $\check N=7$ and $\check\tau$ such that $\cT=\{0,2,4,6,8\},$ i.e., the second, fourth and sixth components of $\check\tau$ are in a fractional neighborhood of the $\tau^0_1,$ $\tau^0_2$ and $\tau^0_3$ respectively. Now suppose following two cases.
	\begin{itemize}[leftmargin=*]
		\item[a] In the first case, suppose that the global optimum $\h\tau=(\h\tau_1,...,\h\tau_7)^T$ obtained from {\bf Step 1} of Algorithm 1 is such that $\h\tau_1<\h\tau_2,$ $\h\tau_3<\h\tau_4,$ and $\h\tau_5=\h\tau_6.$ Thus, in this case, by the definition of the set $\cT^*$ we have that $\cT^*=\{0,2,4,5,8\}.$ Consequently, by the definition of $\tau^*,$ we have that $\tau^*=(-\iny,-1,-1,0,1,1,1).$
		\item[b] In the second case, suppose that the global optimum $\h\tau=(\h\tau_1,...,\h\tau_7)^T$ obtained from {\bf Step 1} of Algorithm 1 is such that $\h\tau_1<\h\tau_2,$ $\h\tau_3<\h\tau_4,$ and $\h\tau_5<\h\tau_6.$ In this case, we have that $\cT^*=\{0,2,4,6,8\},$ and $\tau^*=(-\iny,-1,-1,0,0,1,1).$
	\end{itemize}	
	Our results to follow shall show that any global optimum $\h\tau,$ must lie in a near optimal neighborhood of the corresponding $\tau^*,$ with high probability. Note that irrespective of the orientation of the components of the vector $\tau^*,$ the subset of finite and distinct components is exactly $\tau^0.$ Correspondingly, we shall obtain the estimates $\tilde\tau$ which are obtained as the subset of distinct and finite components of $\h\tau.$
\end{exmp}

\noi{\bf Condition B (assumptions on model dimensions):}\\~
{\it (i) For $j=1,...,N,$ let $S_j=\big\{k\in\{1,...,p\};\,\,\b_{(j)k}^0\ne 0\big\}$ and $S=\cup_j S_j.$ Then for some $s=s_n\ge 1,$ we assume $|S|\le s.$\\~
	(ii) The model dimensions $s,p,n,$ satisfy $s\log p\big/nl_{\min}^2\to 0.$\\~
	(iii) The choice of $k\in[1,\iny),$ and $l_{\min}$ of Condition A, $\rho^2$ of Condition C, together with $s,n,$ satisfies $(s\rho^2)\big/(l_{\min}^2n^{1/k})\to 0.$}

\vspace{2mm}
Condition B(i) is the usual sparsity assumption on high dimensional models. Conditions B(ii) and B(iii) are restrictions on model dimensions, Condition B(iii) restricts the dimensionality of the model in accordance with the initializing $\check u_n$-neighborhood and the minimum separation $l_{\min}.$ The largest model allowed by Condition B occurs when the initializers in Condition A allows for $k=1,$ $l_{\min}>c_u,$ and Condition C allows for $\rho^2=O(1).$ In this case, we require $s\log p/n\to 0,$ i.e., Condition B(iii) becomes redundant given Condition B(ii). 

\vspace{2mm}
\noi{\bf Condition C (assumptions on change parameters):} {\it If $N\ge 1,$\\~
	(i) Define the minimum jump size $\xi_{\min}:=\min_{j}\|\b_{(j)}-\b_{(j-1)}\|_2,$ $j=1,...,N,$ and assume that it is bounded below, i.e.,  $\xi_{\min}>c_u,$ $c_u>0.$ Also define the maximum jump size $\xi_{\max}:=\max_j\|\b_{(j)}-\b_{(j-1)}\|_2,$ $j=1,...,N,$ and let $\rho=\xi_{\max}/\xi_{\min}$  be the ratio of these jump sizes. Assume that $\rho\sqrt{s\log p/n}\to 0.$ \\~
	(ii) Assume that all unknown change points are sufficiently separated, i.e., $d(\tau_{j-1}^0,\tau_j^0)\ge c_ul_{\min},$ $c_u>0$ for all $j=1,...,N,$ such that, $N\rho\check u_n/l_{\min}\to 0.$}

\vspace{2mm}
This condition is only applicable when at least one change point exists in the model (\ref{cp}). When no change point exists ($N=0$), we can instead define the ratio $\rho=1,$ and all remaining conditions can be ignored. Condition C(ii) is satisfied trivially if only a finite number of change points are assumed in model (\ref{cp}) and the jump ratio $\rho\le c_u,$ i.e., the maximum and minimum jumps are of the same order. Note that Condition C(ii) and Condition A(ii) are controlled by the same sequence $l_{\min},$ essentially assuming the least separation between the initializing change points and that between the true change points are of the same order. This is again not asking for much, since by assumption we have also assumed that $\check N\le c(1\vee N)$ in Condition A(i). In the case of an increasing number of change points, its rate is controlled by C(ii). Note that we do not make any assumptions on the maximum jump size $\xi_{\max},$ instead we control the jump ratio $\rho.$

\vspace{2mm}
\noi{\bf Condition D (assumptions on model distributions):} \\~ {\it
	(i) The vectors $x_i=(x_{i1},...,x_{ip})^T,$ $i=1,..,n,$ are i.i.d subgaussian\footnote{Recall that for $\si>0,$ the random variable $\eta$ is said to be $\si$-subgaussian if, for all $t\in\R,$ $E[\exp(t\eta)] \le \exp(\si^2t^2/2).$ Similarly, a random vector $\xi\in\R^p$ is said to be $\si$-subgaussian if the inner products $\langle\xi, v\rangle$ are $\si$-subgaussian for any $v\in\R^p$ with $\|v\|_2 = 1.$} with mean vector zero, and variance parameter $\si_x^2\le C.$ Furthermore, the covariance matrix $\Sigma:=Ex_ix_i^T$ has bounded eigenvalues, i.e., $0<\ka\le\rm{min eigen}(\Si)<\rm{max eigen}(\Si)\le\phi<\iny.$\\~
	(ii) The model errors $\vep_i$ are i.i.d. subgaussian with mean zero and variance parameter $\si_{\vep}^2\le C.$\\~
	(iii) The change inducing random variables $w_i,$ $i=1,...,n$ are i.i.d, with cdf represented by $\Phi(\tau_a)=P(w_i\le \tau_a),$ $\tau_a\in\bar\R,$ and the distance between any two $\tau_a<\tau_b\in\bar\R$ in the cdf scale represented as $d(\tau_a,\tau_b)=P(\tau_a< w_i\le\tau_b).$\\~
	(iv) The r.v.'s $x_i,w_i,\vep_i$ are independent of each other.}

The subgaussian assumptions in Condition D(i) and D(ii) are now standard in high dimensional linear regression models and are known to accommodate a large class of random designs. In ordinary high dimensional linear regression, these assumptions are used to establish well behaved restricted eigenvalues of the Gram matrix $\sum x_ix_i^T/n$ (\cite{raskutti2010restricted}; \cite{rudelson2012reconstruction}), which are in turn used to derive convergence rates of $\ell_1$ regularized estimators (\cite{bickel2009simultaneous}; and several others). These assumptions shall play a similar role in our high dimensional multiple change point setting. Condition D(iii) on the change inducing variable, allows for both discrete or continuous r.v.'s. Finally, we also note that assumption D(iii) on the change inducing variable $w,$ allows for both continuous or discrete r.v.'s.

From a general perspective of regularized estimation for high dimensional change point linear regression models, the works that are closely related to this article are \cite{kaul2018parameter}, \cite{lee2016lasso}, \cite{lee2018},  and \cite{leonardi2016computationally}. The idea of converting a multiple change point detection problem to a variable selection problem using an $\ell_0$ regularization and an arbitrary segmentation is novel and is completely different from all articles listed above. The articles \cite{lee2016lasso}, \cite{kaul2018parameter}, and \cite{lee2018}, consider a setting with only a single change point. From a technical perspective, the assumptions made on model distributions in this article are similar to those made in \cite{kaul2018parameter} and are comparable to those assumed in \cite{leonardi2016computationally}. A major advantage of the proposed methodology is its ability to detect the `no change' case, i.e., where there are no change points in the model, to the best of our knowledge, the only other article that posses this capability is \cite{kaul2018parameter}, although it is limited to atmost a single change point.  Finally, we also emphasize that for the detection and estimation of multiple change points in regression models, the methodology proposed in this article is much more efficient with a computational complexity of $O(Np^2)+$SA,  in comparison to the existing binary segmentation approach proposed in \cite{leonardi2016computationally}, which scales like $O(nNp^2).$

\section{Main Results}\label{sec:Main}

To present the results of this section we require the following definitions. For any $\tau_a,\tau_b\in\bar\R,$ let
\benr\label{def:zeta}
\z_i(\tau_a,\tau_b)=\begin{cases}
	{\bf 1}[\tau_{a}<w_i\le \tau_b],\quad  {\rm if}\,\, \tau_a<\tau_{b}, \\
	{\bf 1}[\tau_b< w_i\le \tau_a],\quad {\rm if}\,\,\tau_b<\tau_a.
\end{cases}
\eenr
Here it is implicitly understood that $\z_i(\tau_a,\tau_b)=0,$ if $\tau_a=\tau_b.$ Also, define for any $\tau_a,\tau_b\in\bar\R,$ the following set of random indices,
\benr\lel{def:nw}
n^w(\tau_a,\tau_b)=\begin{cases}
	i\in\{1,...,n\};\quad \tau_a< w_i\le \tau_b,\quad {\rm if}\,\, \tau_a<\tau_b,\\
	i\in\{1,...,n\};\quad \tau_b< w_i\le \tau_a,\quad {\rm if}\,\,\tau_b<\tau_a.
\end{cases}
\eenr
Here $n^w(\tau_a,\tau_b)=\emptyset,$ if $\tau_a=\tau_b.$ To develop our results we require control on the cardinality  $|n^w(\tau_a,\tau_b)|$ of the random set $n^w(\tau_a,\tau_b).$ Note that this cardinality is determined by the r.v.'s defined in (\ref{def:zeta}), i.e., $|n^w(\tau_a,\tau_b)|= \sti \z_i(\tau_a,\tau_b).$ In view of this observation, the following lemma provides uniform control (over $\tau_a,\tau_b$ ) on the stochastic quantity $\sti \z_i(\tau_a,\tau_b).$

%%%%%%%%%%%%%%%%%%%%%%%%%%%%%%%%%%%%%%%%%%%%%%%%%%%%%%%%%%%%%%%%%%%%%%%%%%%%%%%%%%%%%%%%%%%%%%%%%%%%%%%%%%%%%%%%%%%%%%%%%%%%%%%%%%%%%%%%%%%%%%%%%%%%%%%%%%%%%%%%%%%%%%%%%%%%%%%%%%%%
\begin{lem}\label{lem:indicatorbound} Let $u_n,v_n$ be any non-negative sequences such that $\log (u_n^{-1})=O(\log p)$ and $v_n\ge c_u\log p /n,$ $c_u>0.$ Then under Condition D(iii), we have,
	\benr
	&(i)& \sup_{\substack{\tau_a,\tau_b\in\bar\R;\\d(\tau_a,\tau_b)\le u_n}}\frac{1}{n}\sti \z_i(\tau_a,\tau_b)\le c_u\max\Big\{\frac{\log p}{n}, u_n\Big\},\nn\\
	&(ii)& \inf_{\substack{\tau_a,\tau_b\in\bar\R;\\d(\tau_a,\tau_b)\ge v_n}}\frac{1}{n}\sti \z_i(\tau_a,\tau_b)\ge c_u v_n,\nn
	\eenr
	with probability at least $1- c_1\exp(-c_2\log p),$ for $n$ sufficiently large.
\end{lem}
%%%%%%%%%%%%%%%%%%%%%%%%%%%%%%%%%%%%%%%%%%%%%%%%%%%%%%%%%%%%%%%%%%%%%%%%%%%%%%%%%%%%%%%%%%%%%%%%%%%%%%%%%%%%%%%%%%%%%%%%%%%%%%%%%%%%%%%%%%%%%%%%%%%%%%%%%%%%%%%%%%%%%%%%%%%%%%%%%%%%

An application of Lemma \ref{lem:indicatorbound} leads to uniform control (over $\tau_a,$ $\tau_b$) of other stochastic quantities such as $\big\|\sum_{i\in n^w} \vep_ix_i^T\big\|_{\iny},$ among others, which are necessary for the arguments to follow. These bounds are provided in Lemma \ref{lem:crossbounds} in supplementary materials of this article. More simplistic versions of Lemma \ref{lem:indicatorbound} have also been used by \cite{kaul2017structural} in the context of graphical models with missing data, and in \cite{kaul2018parameter} in the context of high dimensional change point regression with a single change point.

To proceed further, recall that {\bf Step 1} of Algorithm 1, utilizes estimates of regression coefficients from {\bf Step 0}, which are based on misspecified initial change points. Thus, in order to obtain variable selection and estimation results regarding the change point estimates of {\bf Step 1}, we first need to analyze the rates of convergence of regression estimates of {\bf Step 0}. This analysis in turn requires restricted eigenvalue conditions on the Gram matrix $\sum_{i} x_ix_i^T/n,$ which is described in the following.

For any deterministic set $S\subseteq \{1,2,...,p\},$ define the collection $\A$ as,
\benr\label{def:seta}
\A=\Big\{\delta\in\R^p;\, \|\delta_{S^c}\|_1\le 3\|\delta_S\|_1\Big\}.
\eenr
Then, \cite{bickel2009simultaneous} define the lower restricted eigenvalue condition as,
\benr\label{def:reb}
\inf_{\delta\in \A} \frac{1}{n}\sti \delta^T x_ix_i^T\delta \ge c_u\ka\|\delta\|_2^2,\quad\rm{for\,\,some\,\,constant}\,\, \ka>0.
\eenr
Our analysis shall require uniform versions of the condition (\ref{def:reb}), these are developed in Lemma \ref{lem:restrictedeigen}. Additionally, we shall also require the set $\A_2$ defined below, which is a slightly different version of the set $\cA$ defined in (\ref{def:seta}).
\benr\label{def:seta2}
\A_2=\Big\{\delta\in\R^p; \|\delta_{S^c}\|_1\le 3\|\delta_S\|_1+ c_u\xi_{\max}\sqrt{s}\Big\}.
\eenr
Finally, we also mention that other weaker versions of Condition (\ref{def:reb}) are also available in the literature, such as the compatibility condition of \cite{buhlmann2011statistics}, and the $\ell_q$ sensitivity of \cite{gautier2011high}. In the setup of common random designs, it is also well established that condition (\r{def:reb}) holds with probability converging to $1,$ see for e.g. \cite{raskutti2010restricted},  \cite{rudelson2012reconstruction} for Gaussian designs and \cite{loh2012} for sub-Gaussian designs. The following lemma provides the plausibility of the uniform restricted eigenvalue conditions required in our analysis.

%%%%%%%%%%%%%%%%%%%%%%%%%%%%%%%%%%%%%%%%%%%%%%%%%%%%%%%%%%%%%%%%%%%%%%%%%%%%%%%%%%%%%%%%%%%%%%%%%%%%%%%%%%%%%%%%%%%%%%%%%%%%%%%%%%%%%%%%%%%%%%%%%%%%%%%%%%%%%%%%%%%%%%%%%%%%%%%%%%%%
\begin{lem}\label{lem:restrictedeigen}
	Let $\A$ and $\A_2$ be as given in (\ref{def:seta}) and (\ref{def:seta2}) respectively, for $S$ as defined in Condition B. Let $u_n, v_n$ be non-negative sequences such that $\log (u_n^{-1})=O(\log p)$ and $v_n\ge c_us\log p\big/n,$ for a suitably chosen constant $c_u>0.$ Then under Conditions B(i), B(ii) and D, and for $n$ sufficiently large,  the following restricted eigenvalue conditions hold with probability at least $1-c_1\exp(-c_2\log p),$
	\benr
	&(i)& \inf_{\substack{\tau_a,\tau_b\in\bar\R;\\d(\tau_a,\tau_b)\ge v_n}}\inf_{\delta\in \A} \frac{1}{n}\sum_{i\in n^w(\tau_a,\tau_b)}\delta^T x_ix_i^T \delta \ge c_uc_m v_n\|\delta\|_2^2,\nn\\
	&(ii)&\inf_{\substack{\tau_a,\tau_b\in\bar\R;\\ d(\tau_a,\tau_b)\ge v_n}}\inf_{\delta\in\A_2}\frac{1}{n}\sum_{i\in n^w(\tau_a,\tau_b)}\delta^Tx_ix_i^T\delta \ge c_uc_m v_n\|\delta\|_2^2 - c_u c_m \frac{\xi_{\max}^2s\log p}{n},\nn\\
	&(iii)&\sup_{\substack{\tau_a,\tau_b\in\bar\R;\\d(\tau_a,\tau_b)\le u_n}}\sup_{\delta\in \A}\frac{1}{n}\sum_{i\in n_w}\delta^T x_ix_i^T \delta\le c_uc_m\|\delta\|_2^2 \max\Big\{\frac{s\log p}{n}, u_n\Big\}.\nn
	\eenr
\end{lem}
%%%%%%%%%%%%%%%%%%%%%%%%%%%%%%%%%%%%%%%%%%%%%%%%%%%%%%%%%%%%%%%%%%%%%%%%%%%%%%%%%%%%%%%%%%%%%%%%%%%%%%%%%%%%%%%%%%%%%%%%%%%%%%%%%%%%%%%%%%%%%%%%%%%%%%%%%%%%%%%%%%%%%%%%%%%%%%%%%%%%

In the following, for any positive number $r>0$ and any $\tau_a\in\R,$ define the interval $\cB(\tau_a,r)=\big\{\tau\in\R;\, d(\tau_a,\tau)\le r\big\}.$ The rates of the initial regression coefficient estimates of {\bf Step 0} of Algorithm 1 shall be a consequence of the following general result.

%%%%%%%%%%%%%%%%%%%%%%%%%%%%%%%%%%%%%%%%%%%%%%%%%%%%%%%%%%%%%%%%%%%%%%%%%%%%%%%%%%%%%%%%%%%%%%%%%%%%%%%%%%%%%%%%%%%%%%%%%%%%%%%%%%%%%%%%%%%%%%%%%%%%%%%%%%%%%%%%%%%%%%%%%%%%%%%%%%%%
\begin{thm}\label{thm:initialrate}
	Suppose Condition B(i), B(ii), C(ii) and D. Let $u_n$ be any non-negative sequence satisfying $\log (u_n^{-1})=O(\log p)$ and let $Q^*$ be as given in (\ref{def:qstar}). For any $\tau_a,\tau_b\in\bar\R,$ let $\h\al\in\R^{p}$ be the solution to the Lasso optimization
	\benr
	\h\al=\argmin_{\al\in\R^p}\Big\{\frac{1}{n}Q^*(\al,\tau_a,\tau_b)+\la_0\|\al\|_1 \Big\}\nn.
	\eenr
	Additionally, for any fixed $j=1,...,N+1,$ let $\cC_j=\cC_j^1\cup\cC_j^2\cup\cC_j^3\cup\cC^4_j,$ where
	\benr
	\cC_j^1&=&\Big\{\tau_a,\tau_b\in\bar\R;\,\, d(\tau_a,\tau_b)>l_{\min}; \tau_a\in\cB(\tau_{j-1}^0,u_n), \tau_b\in\cB(\tau_j^0,u_n)\Big\},\nn\\
	\cC_j^2&=&\Big\{\tau_a,\tau_b\in\bar\R;\,\, d(\tau_a,\tau_b)>l_{\min}; \tau_a\ge\tau_{j-1}^0, \tau_b\in\cB(\tau_j^0,u_n)\Big\},\nn\\
	\cC_j^3&=&\Big\{\tau_a,\tau_b\in\bar\R;\,\, d(\tau_a,\tau_b)>l_{\min}; \tau_a\in\cB(\tau_{j-1}^0,u_n), \tau_b\le \tau_j^0\Big\},\nn\\
	\cC_j^4&=&\Big\{\tau_a,\tau_b\in\bar\R;\,\, d(\tau_a,\tau_b)>l_{\min}; \tau_a\ge\tau_{j-1}^0, \tau_b\le\tau_{j}^0,\Big\}.\nn
	\eenr
	Then choosing $\la_0=c_uc_m\max\big\{\sqrt{\log p/n}, \xi_{\max}u_n\big\},$ for $n$ sufficiently large we have for $j=1,...,N+1,$
	\benr
	\,\,\sup_{\tau_a,\tau_b\in\cC_j}\|\h\al-\b^0_{(j-1)}\|_q\le c_uc_m s^{\frac{1}{q}} \max\Big\{\sqrt{\frac{\log p}{n}},\,\, \xi_{\max}u_n\Big\}\Big/l_{\min},\quad q=1,2,\nn
	\eenr
	with probability at least $1-c_1\exp(-c_2\log p).$
\end{thm}
%%%%%%%%%%%%%%%%%%%%%%%%%%%%%%%%%%%%%%%%%%%%%%%%%%%%%%%%%%%%%%%%%%%%%%%%%%%%%%%%%%%%%%%%%%%%%%%%%%%%%%%%%%%%%%%%%%%%%%%%%%%%%%%%%%%%%%%%%%%%%%%%%%%%%%%%%%%%%%%%%%%%%%%%%%%%%%%%%%%%

Theorem \ref{thm:initialrate} can be used to obtain the rates of convergence of $\h\al_{(j)},$ $j=0,...,\check N,$ obtained from {\bf Step 0} of Algorithm 1 and Algorithm 2. To state these rates explicitly we require the following notation. Let $\cT$ be as defined in Condition A, and define for each $j\in\cT^c,$
\benr\label{def:kj}
k_j=\min\Big\{k;\,\, 1\le k\le N+1;\,\, \tau^0_{k-1}<\check \tau_j\le \tau^0_{k}\Big\}.
\eenr
Simply stated, the index $k_j$ is the first index $k$ between $1,..,N+1,$ such that $\tau_j$ lies between $\tau^0_{k-1}$ and $\tau^0_k.$ This index $k_j$ identifies the regression coefficient vector $\b^0_{(k_j-1)}$ with its approximation $\h\al_{(j-1)}$ for each $j\in\cT^c,$ this notation is illustrated in Example \ref{eg:initialrate}. Under this notation, the following corollary provides the rates of convergence of the initial regression estimates.

%%%%%%%%%%%%%%%%%%%%%%%%%%%%%%%%%%%%%%%%%%%%%%%%%%%%%%%%%%%%%%%%%%%%%%%%%%%%%%%%%%%%%%%%%%%%%%%%%%%%%%%%%%%%%%%%%%%%%%%%%%%%%%%%%%%%%%%%%%%%%%%%%%%%%%%%%%%%%%%%%%%%%%%%%%%%%%%%%%%%
\begin{cor}\label{cor:intialrate}
	Let $\check N$ and $\check\tau\in\R^{\check N}$ be any initializers satisfying Condition A and assume the conditions of Theorem \ref{thm:initialrate}. Also, let $\h\al_{(j)},$ $j=0,...,N$ be the estimates obtained from {\bf Step 0} of Algorithm 1 and let ${k_j}$ be as defined in (\ref{def:kj}). Then, upon choosing $\la_0=c_uc_m\max\big\{\sqrt{\log p/n}, \xi_{\max}{\check u}_n\big\},$ $q=1,2,$ and $n$ sufficiently large, we have the following.\\~
	(i) For each fixed $j=1,...,N+1,$
	\benr
	\|\h\al_{(m_j-1)}-\b^0_{(j-1)}\|_q\le c_uc_m s^{\frac{1}{q}} \max\Big\{\sqrt{\frac{\log p}{n}},\,\, \xi_{\max}{\check u}_n\Big\}\Big/l_{\min},\nn
	\eenr	
	with probability at least $1-c_1\exp(-c_2\log p).$ \\
	(ii) For each fixed $j\in\cT^c,$
	\benr
	\|\h\al_{(j-1)}-\b^0_{(k_j-1)}\|_q\le c_uc_m s^{\frac{1}{q}} \max\Big\{\sqrt{\frac{\log p}{n}},\,\,\xi_{\max}{\check u}_n\Big\}\Big/l_{\min},\nn
	\eenr
	with probability at least $1-c_1\exp(-c_2\log p).$
\end{cor}
%%%%%%%%%%%%%%%%%%%%%%%%%%%%%%%%%%%%%%%%%%%%%%%%%%%%%%%%%%%%%%%%%%%%%%%%%%%%%%%%%%%%%%%%%%%%%%%%%%%%%%%%%%%%%%%%%%%%%%%%%%%%%%%%%%%%%%%%%%%%%%%%%%%%%%%%%%%%%%%%%%%%%%%%%%%%%%%%%%%%

%%%%%%%%%%%%%%%%%%%%%%%%%%%%%%%%%%%%%%%%%%%%%%%%%%%%%%%%%%%%%%%%%%%%%%%%%%%%%%%%%%%%%%%%%%%%%%%%%%%%%%%%%%%%%%%%%%%%%%%%%%%%%%%%%%%%%%%%%%%%%%%%%%%%%%%%%%%%%%%%%%%%%%%%%%%%%%%%%%%%
\begin{exmp}\label{eg:initialrate} {\rm
		Suppose $N=2,$ $\check N=4$ and the chosen initial $\check\tau\in\R^{4}$ is in the orientation illustrated in Figure \ref{fig:initialorientation}.
		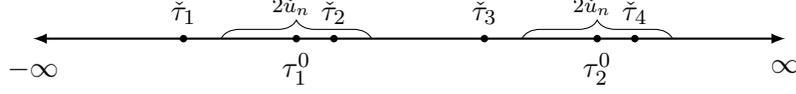
\begin{figure}[H]
			\centering{
				\begin{tikzpicture}
				\draw[black,thick,latex-latex] (0,0) -- (10,0)
				node[pos=0,label=below:\textcolor{black}{$-\iny$}]{}
				node[pos=0.35,mynode,fill=black,label=below:\textcolor{black}{$\tau_1^0$}]{}
				node[pos=0.75,mynode,fill=black,label=below:\textcolor{black}{$\tau_2^0$}]{}
				node[pos=0.20,mynode,fill=black,label=above:\textcolor{black}{$\check\tau_1$}]{}
				node[pos=0.40,mynode,fill=black,label=above:\textcolor{black}{$\check\tau_2$}]{}
				node[pos=0.60,mynode,fill=black,label=above:\textcolor{black}{$\check\tau_3$}]{}
				node[pos=0.80,mynode,fill=black,label=above:\textcolor{black}{$\check\tau_4$}]{}
				node[pos=1,label=below:\textcolor{black}{$\iny$}]{};
				\draw[decorate,decoration={brace,amplitude=7pt}] (2.5,0) --node[above left=5pt and -9pt]{$\scriptstyle{2{\check u}_n}$} (4.5,0);
				\draw[decorate,decoration={brace,amplitude=7pt}] (6.5,0) --node[above left=5pt and -9pt]{$\scriptstyle{2{\check u}_n}$} (8.5,0);
				\end{tikzpicture}
				\caption{\footnotesize{A possible orientation of initializers $\check\tau\in\R^{4},$ where $N=2.$}}
				\label{fig:initialorientation}}
		\end{figure}
		
		\vspace{-4mm}
		In this orientation of $\check\tau$, we have $\cT=\{m_0,m_1,m_2,m_3\}=\{0,2,4,5\},$ and $\{k_1,k_3\}=\{1,2\}.$ Consequently, by Corollary \ref{cor:intialrate}, the initial regression estimates $\h\al_{(0)},\h\al_{(1)},....,\h\al_{(4)}$ will be such that $\h\al_{(0)},$ $\h\al_{(2)}$ will approximate $\b^0_{(0)},$ $\b^0_{(1)}$ respectively and $\h\al_{(1)},$ $\h\al_{(3)},\h\al_{(4)}$ will approximate $\b^0_{(0)},\b^0_{(1)}$ and $\b^0_{(2)}$ respectively.}
\end{exmp}
%%%%%%%%%%%%%%%%%%%%%%%%%%%%%%%%%%%%%%%%%%%%%%%%%%%%%%%%%%%%%%%%%%%%%%%%%%%%%%%%%%%%%%%%%%%%%%%%%%%%%%%%%%%%%%%%%%%%%%%%%%%%%%%%%%%%%%%%%%%%%%%%%%%%%%%%%%%%%%%%%%%%%%%%%%%%%%%%%%%%

%%%%%%%%%%%%%%%%%%%%%%%%%%%%%%%%%%%%%%%%%%%%%%%%%%%%%%%%%%%%%%%%%%%%%%%%%%%%%%%%%%%%%%%%%%%%%%%%%%%%%%%%%%%%%%%%%%%%%%%%%%%%%%%%%%%%%%%%%%%%%%%%%%%%%%%%%%%%%%%%%%%%%%%%%%%%%%%%%%%%
We now turn our attention to the main goal of this article, i.e., establishing variable selection and estimation results of change point estimates obtained from {\bf Step 1} of Algorithm 1. To achieve this, we require the following series of definitions. Let $\tau^*$ be as defined in (\ref{def:taustar}), and for any $\check N\ge 1,$ $\tau\in \R^{\check N}$ and $\al\in\R^{p(\check N+1)},$ define,
\benr\label{def:calliustar}
\hspace{1cm}\cU^*(\check N,\al,\tau)&=&Q(\check N,\h\al,\tau)-Q(\check N,\h\al,\tau^*),\\
\cU(\check N,\h\al,\tau)&=&\cU^*(\check N,\h\al,\tau)+\mu\sum_{j=1}^{\check N}\Big(\|d(\tau_j,\tau_{j-1})\|_0-\|d(\tau_j^*,\tau_{j-1}^*)\|_0\Big).\nn
\eenr
From the definition (\ref{def:taustar}), note that when $\check N=N,$ we have that $\tau^*=\tau^0.$ Recall the sets of indices $\cT$ from Condition A, and the set $\cT^*$ from (\ref{def:settstar}) for any $\tau\in\bar\R^{\check N}.$ Note that the intersection $\cT^{*c}\cap\cT^{c}$ comprises of all possible indices that may potentially lead to distinct interruptions between $\tau_{h_0},\tau_{h_1},...,\tau_{h_{\check N+1}}.$ Keeping this observation in mind, consider any non-negative sequences $u_n,v_n,$ any subset $\cK\subseteq \cT^{*c}\cap\cT^{c},$ define the collection,
\benr\label{def:rtset}
\hspace{2mm}\cG(u_n,v_n,\cK)=\Big\{\tau\in\bar\R^{\check N};\,\, \tau_1\le\tau_2\le...\le\tau_{\check N},\,\,\hspace{2cm}\\
v_{n}\le \sum_{j=1}^N d(\tau_{h_j},\tau^0_j)\le u_{n},\,\, {\rm and\,\,for\,\,each}\,\,l\in\cK,\,\,\tau_{l}\ne\tau_{l-1}\Big\}.\hspace{-1.5cm}\nn
\eenr
The arguments $u_n,v_n$ capture information regarding the closeness of an arbitrary vector to the unknown change point vector in the components corresponding to the set $\cT^*.$ The set $\cK\subseteq\cT^{*c}$ captures all distinct interruptions between any two components with indices in the set $\cT^*.$ The following example provides more insight to the construction of the set $\cG(u_n,v_n,\cK),$ and its defining arguments.
\begin{exmp} Consider the model (\ref{cp}) with $N=2.$ Let the initializer $\check \tau$ be chosen such that $\check N=5,$ such that $\cT=\{0,2,4,6\},$ and $\cT^c=\{1,3,5\}.$ Then for any $\tau=(\tau_1,\tau_2,\tau_3,\tau_4,\tau_5)^T\in\bar\R^{\check N},$ satisfying $\tau_1\le\tau_2...\le\tau_5,$ consider the following three scenarios.
	\begin{itemize}[leftmargin=*]
		\item[a] If $\tau_1<\tau_2<\tau_3=\tau_4<\tau_5,$ then $\cT^*=\{0,2,3,6\},$ and $\cT^{*c}=\{1,4,5\}.$ Clearly, the set $\cT^{*c}\cap\cT^c=\{1,5\}$ form the distinct interruptions. Thus, assuming that $v_n\le d(\tau_2,\tau_1^0)+d(\tau_3,\tau_2^0)\le u_n,$ then $\tau\in \cG(u_n,v_n,\cK),$ with $\cK=\{1,5\}.$
		\item[b] If $\tau_1=\tau_2=\tau_3=\tau_4=\tau_5,$ then $\cT^*=\{0,1,3,6\},$ and $\cT^{*c}=\{5\}.$ The potential interruptions can be due to induces in the set $\cT^{*c}\cap\cT^c=\{5\},$ however since in this case $\tau_5=\tau_4,$ hence $\cK=\emptyset.$
		\item[c]  If $\tau_1=\tau_2<\tau_3<\tau_4=\tau_5,$ then $\cT^*=\{0,1,4,6\},$ $\cT^{*c}=\{2,3,5\}.$ Potential interruptions can be due to induces in the set $\cT^{*c}\cap\cT^c=\{3,5\}.$ Since $\tau_5=\tau_4,$ thus in this case $\cK=\{3\}$ captures the sole distinct interruption.
	\end{itemize}
\end{exmp}

A partial motivation for defining the collection $\cG(u_n,v_n,\cK)$ is as follows. Recall from the results stated in (\ref{eq:ratesintro}), we intend to show that the number of finite and distinct components $\tilde N$ of $\h\tau$ obtained from {\bf Step 1} of Algorithm 1 matches exactly with the true number of change points $N,$ with high probability. The argument we develop to prove this result proceeds by showing that $\h\tau$ must lie in $\cG(u_n,v_n,\cK),$ where $\cK=\emptyset,$ with high probability. Note that the latter statement shall infact imply the desired result.

Finally, for any non-negative sequence $u_n,$ we also define the function,
\benr\label{def:fun}
F(u_n)=\begin{cases} 0 &\,\,{\rm if}\,\, u_n/l_{\min}\to 0\\ N &\,\, {\rm otherwise}\end{cases}.
\eenr
The following lemma provides a uniform lower bound of the expression $\cU(\check N,\h\al,\tau),$ over the collection $\cG:=\cG(u_n,v_n,\cK),$ that holds with high probability. This result shall lie at the heart of the argument used to obtain the main results of this article regarding variable selection and estimation of change points from Algorithm 1. For the result to follow, let $r_n$ be the $\ell_2$ rate obtained from the initial regression coefficients provided in Corollary \ref{cor:intialrate}, i.e., $r_n= c_uc_m \sqrt{s}\max\big\{\sqrt{\log p/n},\,\,\xi_{\max}{\check u}_n\big\}\big/l_{\min}.$
%%%%%%%%%%%%%%%%%%%%%%%%%%%%%%%%%%%%%%%%%%%%%%%%%%%%%%%%%%%%%%%%%%%%%%%%%%%%%%%%%%%%%%%%%%%%%%%%%%%%%%%%%%%%%%%%%%%%%%%%%%%%%%%%%%%%%%%%%%%%%%%%%%%%%%%%%%%%%%%%%%%%%%%%%%%%%%%%%%%%

%%%%%%%%%%%%%%%%%%%%%%%%%%%%%%%%%%%%%%%%%%%%%%%%%%%%%%%%%%%%%%%%%%%%%%%%%%%%%%%%%%%%%%%%%%%%%%%%%%%%%%%%%%%%%%%%%%%%%%%%%%%%%%%%%%%%%%%%%%%%%%%%%%%%%%%%%%%%%%%%%%%%%%%%%%%%%%%%%%%%
\begin{lem}\label{lem:ulowerbound} Suppose Conditions A, B(i), B(ii) C, and D hold. Let $\check u_n$ be as given in Condition A and choose $\la_0$ as prescribed in Corollary \ref {cor:intialrate}. Let $u_n,v_n$ be any non-negative sequences such that $\log (u_n^{-1})=O(\log p).$ Let $\cG:=\cG(u_n,v_n,\cK),$ and $F(u_n),$ be as defined in (\ref{def:rtset}), and (\ref{def:fun}). Additionally, let $\h\al$ be the estimates obtained from {\bf Step 0} of Algorithm 1. Then for $n$ sufficiently large, we have the following lower bounds.\vspace{1mm} \\
	(i) When $N=0,$ we have,
	\benr
	\inf_{\tau\in\cG}\cU(\check N,\h\al,\tau)\ge \mu|\cK|-c_uc_m|\cK|r_n^2- c_uc_m|\cK|\sqrt{\frac{s\log p}{n}}r_n,\nn
	\eenr
	with probability at least $1-c_1 (1\vee N)\exp(-c_2\log p).$\vspace{1mm}\\
	(ii) When $N\ge 1,$ and $v_n\ge c_uNs\log p/n,$\footnote{This result is also valid when $v_n=0.$} we have,
	\benr
	\inf_{\tau\in\cG}\cU(\check N,\h\al,\tau)&\ge& c_uc_mv_{n}+\mu|\cK|-c_uc_mN\frac{\rho^2s\log p}{n}-\frac{c_uc_m}{(1\vee\xi_{\min}^2)}|\cK|r_n^2\nn\\
	&&-\frac{c_uc_m}{(1\vee\xi_{\min}^2)}r_n^2u_{n}- \frac{c_uc_m\rho}{(1\vee\xi_{\min})}\sqrt{\frac{s\log p}{n}}\sqrt{Nu_n}\nn\\
	&&- \frac{c_uc_m}{(1\vee\xi_{\min}^2)}|\cK|\sqrt{\frac{s\log p}{n}}r_n-  \frac{\mu}{(1\vee\xi_{\min}^2)}F(u_n),\nn
	\eenr
	with probability at least $1-c_1 (1\vee N)\exp(-c_2\log p).$
\end{lem}
%%%%%%%%%%%%%%%%%%%%%%%%%%%%%%%%%%%%%%%%%%%%%%%%%%%%%%%%%%%%%%%%%%%%%%%%%%%%%%%%%%%%%%%%%%%%%%%%%%%%%%%%%%%%%%%%%%%%%%%%%%%%%%%%%%%%%%%%%%%%%%%%%%%%%%%%%%%%%%%%%%%%%%%%%%%%%%%%%%%%

%%%%%%%%%%%%%%%%%%%%%%%%%%%%%%%%%%%%%%%%%%%%%%%%%%%%%%%%%%%%%%%%%%%%%%%%%%%%%%%%%%%%%%%%%%%%%%%%%%%%%%%%%%%%%%%%%%%%%%%%%%%%%%%%%%%%%%%%%%%%%%%%%%%%%%%%%%%%%%%%%%%%%%%%%%%%%%%%%%%%
The preceding results developed in this article provide the necessary machinery required to obtain the main results of this article regarding estimation of the number and locations of change points, and the regression coefficients obtained from Algorithm 1. The results to follow shall essentially say that, with high probability, Algorithm 1 exactly recovers the unknown number of change points and yields estimates of locations of change points that are in a near optimal neighborhood of the unknown change points. Additionally, Algorithm 2 yields regression coefficient estimates that are in an optimal neighborhood of the unknown regression coefficients. The following theorem provides the validity of the estimates $\tilde N$ and $\tilde\tau$ obtained from Algorithm 1.
%%%%%%%%%%%%%%%%%%%%%%%%%%%%%%%%%%%%%%%%%%%%%%%%%%%%%%%%%%%%%%%%%%%%%%%%%%%%%%%%%%%%%%%%%%%%%%%%%%%%%%%%%%%%%%%%%%%%%%%%%%%%%%%%%%%%%%%%%%%%%%%%%%%%%%%%%%%%%%%%%%%%%%%%%%%%%%%%%%%%

%%%%%%%%%%%%%%%%%%%%%%%%%%%%%%%%%MAIN THEOREM%%%%%%%%%%%%%%%%%%%%%%%%%%%%%%%%%%%%%%%%%%%%%%%%%%%%%%%%%%%%%%%%%%%%%%%%%%%%%%%%%%%%%%%%%%%%%%%%%%%%%%%%%%%
\begin{thm}\label{thm:mainresult} Assume Conditions A, B, C and D and choose $\la_0$ as prescribed in Corollary \ref{cor:intialrate}. Let $\cT,$ $\cT^*,$ and $\h\cT$ be as defined in Condition A, (\ref{def:settstar}) and (\ref{def:hatct}) respectively. Then upon choosing $\mu=c_uc_m\rho(s\log p/n)^{1/k^*},$ with $k^*=2\vee k,$ the estimates $\tilde N,$ and $\tilde \tau$ obtained from Algorithm 1 satisfy the following relations,
	\benr
	&(i)& {\rm When}\,\, N\ge 0,\quad{\rm we\,\,have}\quad \tilde N=N,\nn\\
	&(ii)&  {\rm When}\,\, N\ge 1,\quad{\rm we\,\,have}\quad \sum_{j=1}^{N} d(\tilde\tau_{j},\tau_{j}^0)\le c_uc_m N\rho^2\frac{s\log p}{n},\nn
	\eenr
	with probability at least $1-c_1(1\vee N)\exp(-c_2\log p),$ and for $n$ sufficiently large.
\end{thm}
%%%%%%%%%%%%%%%%%%%%%%%%%%%%%%%%%%%%%%%%%%%%%%%%%%%%%%%%%%%%%%%%%%%%%%%%%%%%%%%%%%%%%%%%%%%%%%%%%%%%%%%%%%%%%%%%%%%%%%%%%%%%%%%%%%%%%%%%%%%%%%%%%%%%%%%%%

The usefulness of Theorem \ref{thm:mainresult} is apparent. Despite initializing Algorithm 1 with an arbitrarily large $\check N,$ the estimates $\h\tau$ obtained from {\bf Step 1} of Algorithm 1 will have exactly $N$ finite and distinct components with high probability (all other components will collapse to any of the remaining $N$ distinct components or negative infinity). Additionally, the components of $\h\tau$ that are identified as finite and distinct will lie in a near optimal neighborhood of the true change point vector $\tau^0.$ Recall that estimate $\h\tau$ from {\bf Step 1} of Algorithm 1 are computed based on regression estimates from {\bf Step 0} that may be much slower than optimal in their rate of convergence. Yet, $\h\tau$ of is near optimal in its rate of convergence. It is also important to remember that this process is carried out in a single step and not by an iterative procedure, thereby also providing the algorithm its computational advantage. The following theorem provides the rate of convergence of regression coefficient estimates obtained from Algorithm 2.
%%%%%%%%%%%%%%%%%%%%%%%%%%%%%%%%%%%%%%%%%%%%%%%%%%%%%%%%%%%%%%%%%%%%%%%%%%%%%%%%%%%%%%%%%%%%%%%%%%%%%%%%%%%%%%%%%%%%%%%%%%%%%%%%%%%%%%%%%%%%%%%%%%%%%%%%%%%%%%%%%%%%%%%%%%%%%%%%%%%%

%Recall that {\bf Step 0} of Algorithm 1 yields estimates of regression coefficients $\h\al_{(j)}\in\R^p,$ $j=0,...,\check N.$ Note that, these estimates can equivalently be expressed as,
%\benr
%\h\al_{(j-1)}=\argmin_{\al\in\R^p}\Big\{\frac{1}{n}Q^*(\al,\check\tau_{j-1},\check \tau_{j})+\la_0\|\al\|_1 \Big\},\quad j=1,...,N+1.\nn
%\eenr
%This equivalence directly allows the applicability of

%%%%%%%%%%%%%%%%%%%%%%%%%%%%%%%%%%%%%%%%%%%%%%%%%%%%%%%%%%%%%%%%%%%%%%%%%%%%%%%%%%%%%%%%%%%%%%%%%%%%%%%%%%%%%%%%%%%%%%%%%%%%%%%%%%%%%%%%%%%%%%%%%%%%%%%%%%%%%%%%%%%%%%%%%%%%%%%%%%%%

\begin{cor}\label{cor:tildealpha} Suppose the conditions of Theorem \ref{thm:mainresult} and for each $j=0,...,N,$ choose $\la_{1j}=c_uc_m\max\big\{\sqrt{\log p/n},\,\, \xi_{\max}(|\tilde\tau_j-\tau^0_{j}|\vee|\tilde\tau_{j+1}-\tau^0_{j+1}|)\big\}.$ Let $N\ge 1,$ and $\tilde\al_{(j)},$ $j=0,...,N$ be estimates of the regression coefficients obtained from Algorithm 3. Then, for $n$ sufficiently large and $q=1,2,$ we have the following bound,
	\benr
	\sum_{j=0}^{N}\|\tilde\al_{(j)}-\b^0_{(j)}\|_q\le\nn c_uc_mN\frac{s^{\frac{1}{q}}}{l_{\min}}\max\Big\{\sqrt{\frac{\log p}{n}},\,\, \xi_{\max}\rho^2\frac{s\log p}{n}\Big\},\nn
	\eenr
	that holds with probability at least $1-c_1(1\vee N)\exp(-c_2\log p). $
\end{cor}
%%%%%%%%%%%%%%%%%%%%%%%%%%%%%%%%%%%%%%%%%%%%%%%%%%%%%%%%%%%%%%%%%%%%%%%%%%%%%%%%%%%%%%%%%%%%%%%%%%%%%%%%%%%%%%%%%%%%%%%%%%%%%%%%%%%%%%%%%%%%%%%%%%%%%%%%%%%%%%%%%%%%%%%%%%%%%%%%%%%%

%%%%%%%%%%%%%%%%%%%%%%%%%%%%%%%%%%%%%%%%%%%%%%%%%%%%%%%%%%%%%%%%%%%%%%%%%%%%%%%%%%%%%%%%%%%%%%%%%%%%%%%%%%%%%%%%%%%%%%%%%%%%%%%%%%%%%%%%%%%%%%%%%%%%%%%%%%%%%%%%%%%%%%%%%%%%%%%%%%%%
To conclude this section, we present the following corollary that specifies conditions under which near optimality of these rates is observed, as described in (\ref{eq:ratesintro}).
%%%%%%%%%%%%%%%%%%%%%%%%%%%%%%%%%%%%%%%%%%%%%%%%%%%%%%%%%%%%%%%%%%%%%%%%%%%%%%%%%%%%%%%%%%%%%%%%%%%%%%%%%%%%%%%%%%%%%%%%%%%%%%%%%%%%%%%%%%%%%%%%%%%%%%%%%%%%%%%%%%%%%%%%%%%%%%%%%%%%

%%%%%%%%%%%%%%%%%%%%%%%%%%%%%%%%%%%%%%%%%%%%%%%%%%%%%%%%%%%%%%%%%%%%%%%%%%%%%%%%%%%%%%%%%%%%%%%%%%%%%%%%%%%%%%%%%%%%%%%%%%%%%%%%%%%%%%%%%%%%%%%%%%%%%%%%%%%%%%%%%%%%%%%%%%%%%%%%%%%%
\begin{cor}\label{cor:rateoptimal} Suppose conditions of Theorem \ref{thm:mainresult} and Corollary \ref{cor:tildealpha}. Then assuming that $\rho^2=O(1),$ we have the following relations with probability at least $1-c_1(1\vee N)\exp(-c_2\log p).$ \\
	(i) For $N\ge 0,$ $\tilde N=N.$\\~
	(ii) For $N\ge 1,$ and $n$ sufficiently large, $\sum_{j=1}^{N} d({\tilde\tau}_{j},\tau_{j}^0) \le c_uc_m Ns\log p\big/n.$
	\\~
	(iii) Additionally assuming that $N\le 1\vee c_{u1},$ $l_{\min}\ge c_{u2},$ and that $\xi_{\max}s\sqrt{\log p/n}=O(1).$ We have,
	$\sum_{j=0}^{N}\|\tilde\al_{(j)}-\b^0_{(j)}\|_q\le\nn c_uc_mNs^{1/q}\sqrt{\log p/n},$ for $q=1,2$ and $n$ sufficiently large.
\end{cor}
%%%%%%%%%%%%%%%%%%%%%%%%%%%%%%%%%%%%%%%%%%%%%%%%%%%%%%%%%%%%%%%%%%%%%%%%%%%%%%%%%%%%%%%%%%%%%%%%%%%%%%%%%%%%%%%%%%%%%%%%%%%%%%%%%%%%%%%%%%%%%%%%%%%%%%%%%%%%%%%%%%%%%%%%%%%%%%%%%%%%

\section{Implementation and numerical results}\label{sec:implement}

In this section we discuss the implementation of the proposed methodology and provide monte carlo simulation results of the same. First, as briefly stated in Section \ref{sec:intro}, for any fixed $\al\in\R^{p(\check N+1)},$ the loss function $Q(\check N,\al,\tau)$ is step function of $\tau,$ with step changes occurring at any point on the $\check N$ dimensional finite grid $\{-\iny,w_1,....,w_n,\iny\}^{\check N}.$ We illustrate this fact in Figure \ref{fig:stepbehave.zeroinflate}, for the special case where $N=\check N=1.$ To proceed with the implementation of Algorithm 1, first note that {\bf Step 1} of Algorithm 1 requires $\Phi(\cdot)$ to be known \big(via the distance function $d(\cdot)$\big), which is typically not the case in practice. However, also note that the function $d(\cdot)$ appears in the optimization of {\bf Step 1} only via the $\ell_0$ norm,  $\|d(\tau_{j-1},\tau_j)\|_0.$ Observing that $\|d(\tau_{j-1},\tau_j)\|_0=\|\tau_{j-1}-\tau_j\|_0,$ provided we implicitly define the additional conventions $\|\iny-\iny\|_0:=0,$ and $\|\iny-a\|_0:=1,$ for any $a<\iny,$ in the implementation. Thus, the term $\|d(\tau_{j-1},\tau_j)\|_0$ can be replaced by $\|\tau_{j-1}-\tau_j\|_0$ without altering the estimator. Alternatively, to avoid this notational complexity in coding the estimator, a new surrogate variable $w_i^*$ can be created which follows a pseudo uniform distribution\footnote{Here we refer to a pseudo uniform distribution in the sense typically used in MCMC methods, where the realizations $w_i^*,....w_n^*$ reproduce the behavior of $n$ realizations of a $\cU(0,1]$ distribution, see, Definition 2.1 of \cite{robert2013monte}.}, $w_i^*\sim\cU(0,1],$ while preserving the data structure. This can be done as follows, let $w_{(1)},..w_{(n)}$ represent the order statistics of $w_i's,$ and construct $w_{(i)}^*=i/n,$ $i=1,...,n.$ Since $w_i$'s are independent realizations, the surrogate $w_i^*\sim\cU(0,1]$ in the sense described above. In this case, we can reparameterize the model (\ref{cp}) to an ordinary change point regression model as follows. First, re-order all observations with respect to the ordered surrogate change inducing variable $w_{(1)}^*,...,w_{(n)}^*.$ Then we can express model (\r{cp}) as,
\benr\label{model:reparcp}
\hspace{0.5cm}y_{i}= x_{i}^T\b_{(j-1)}^0 +\vep_{i},\quad \tau_{j-1}^{\dagger}<i/n\le \tau_j^{\dagger},\,\,j=1,...,N+1.
\eenr
Here, $\tau_j^{\dagger},$ $j=1,...,N$ are reparameterized change point parameters in the ${\rm Supp}(w^*)=(0,1],$ and $\tau_0^{\dagger}=0,\tau_{N+1}^{\dagger}=1.$ In view of this reparameterization, together with the step behavior of the function $Q(\check N, \h\al,\cdot),$ we can now equivalently implement Algorithm 1a, in place of Algorithm 1.

\begin{figure}[H]
	\noi\rule{\textwidth}{0.5pt}
	
	\vspace{-3mm}
	\begin{flushleft}{\bf Algorithm 1a:} Detection and estimation of number of change point(s) with reparameterization and data $(x,y,w^*)$\end{flushleft}
	\vspace{-4.75mm}
	\noi\rule{\textwidth}{0.5pt}
	
	\vspace{-1mm}
	\begin{itemize}[wide, labelwidth=!, labelindent=0pt]
		\item[{\bf Step 0}:] ({\bf Initializing step}) Choose any $\check N\ge 1\vee N,$ and any vector $\check\tau=(\check\tau_1,...,\check\tau_{\check N})^T\in\R^{\check N},$ satisfying {\bf Condition A}. Compute initial regression estimates $\h\al_{(j)},$ for each $j=0,...,\check N,$
		\vspace{-1mm}
		\benr
		\h\al_{(j)}=\argmin_{\al\in\R^{p}}\Big\{\frac{1}{n}Q^*(\al, \check\tau_{j-1},\check\tau_j)+\la_0\|\alpha\|_1 \Big\},\qquad \la_0>0\nn
		\eenr
		\vspace{-4mm}
		\item[{\bf Step 1}:]  Update $\check \tau\in\R^{\check N}$ to obtain estimate $\hat \tau\in\bar\R^{\check N}$, where,
		\vspace{-2mm}
		\benr
		\h\tau=\argmin_{\substack{\tau\in \{0,\frac{1}{n},\frac{2}{n}...,1\}^{\check N};\\ \tau_{j-1}\le\tau_j\, \forall\, j}}\Big\{Q(\check N, \h\al, \tau)+ \mu \sum_{j=1}^{\check N}\|d(\tau_{j-1},\tau_{j})\|_0 \Big\},\qquad \mu>0.\nn
		\eenr
		\vspace{-4mm}
		
		\noi Let $\h\cT:=\h\cT(\h\tau),$ and update the estimated number of change points to $\tilde N = |\h \cT|,$ and recover the corresponding locations of change points as the subset $\tilde\tau=\h\tau_{\h\cT}\in\R^{\tilde N}.$
	\end{itemize}
	
	\vspace{-3mm}
	\noi\rule{\textwidth}{0.5pt}
\end{figure}

The change made in Algorithm 1a (in comparison to Algorithm 1) is in {\bf Step 1} of the procedure. First instead of searching over the extended Euclidean space, we are instead searching over a finite multi-dimensional grid. Second, owing to the creation of the surrogate change inducing variable, $w_i^*\sim\cU(0,1],$ we have $d(\tau_{j-1},\tau_j)=|\tau_{j-1}-\tau_j|.$ The only difference is that, Algorithm 1a estimates the parameters of the reparameterized model (\ref{model:reparcp}) instead of (\ref{cp}). The change point parameters of model (\r{cp}) can be easily obtained from those of (\ref{model:reparcp}) by reverting back to the corresponding quantiles.

Observe that {\bf Step 0} of Algorithm 1a and {\bf Step 1} of Algorithm 2 are ordinary Lasso optimizations, these can be accomplished by several different methods available in the literature, for e.g. coordinate or gradient descent algorithms, see, e.g. \cite{hastie2015statistical}  or via interior point methods for linear optimization under second order conic constraints, see, e.g., \cite{koenker2014convex}. On the other hand, the implementation of {\bf Step 1} of Algorithm 1a is a non trivial task. Keeping in mind that this step is a discrete optimization over a finite state space, we propose a simulated annealing approach for this purpose and the method is discussed in the following subsection.

\subsection{Implementation of {\textit{\textbf {Step 1}}} of Algorithm 1a via simulated annealing} \label{sec:simanneal}

Simulated annealing is a well known variant of the Metropolis Hastings algorithm, see, for e.g. Chapter 5 and Chapter 7 of the monograph \cite{robert2013monte}. This algorithm is especially useful for finite state space optimizations, and its stochastic nature endows it with its most desirable feature, which is its ability to escape local optimums while only visiting very few states of the state space under consideration.

First, we require another reparameterization of {\bf Step 1} of Algorithm 1a. Let $d^{\dagger}=(d_1^{\dagger},...,d_{N}^{\dagger})^T\in\R^N,$ be parameters of the model (\ref{model:reparcp}), such that $n\tau_1^{\dagger}=d_1,n\tau_2^{\dagger}=d_1+d_2,...,n\tau_{N}^{\dagger}=\sum_{j=1}^{N}d_j^{\dagger}.$ Then {\bf Step 1} of Algorithm 1a can equivalently be performed by searching for an optimizer $\h d=(\h d_1,...,\h d_{\check N})^T$ in the state space $\{0,1...n\}^{\check N},$ as follows,
\benr\label{eq:reparopt}
\h d=\argmin_{\substack{d\in \{0,1,2...,n\}^{\check N};\\ \sum_{j=1}^{\check N}d_j\le n}}\Big\{Q(\check N, \h\al, \tau)+ \mu \sum_{j=1}^{\check N}\|\frac{d_j}{n}\|_0 \Big\},\qquad \mu>0,
\eenr
where $\tau=(\tau_1,...,\tau_{\check N})^T,$ with $n\tau_j=\sum_{k=1}^j d_k,$ $j=1,...,\check N.$ Finally, the change point estimates of {\bf Step 1} of Algorithm 1a can be recovered by computing $\h\tau=(\h d_1,\h d_{1}+\h d_2,....,\sum_{j=1}^{\check N} \h d_j)^T\big/n.$ We adopt simulated annealing in the context of optimization (\ref{eq:reparopt}).

For efficient implementation of this procedure, one requires a carefully constructed proposal density taking into account special features of the problem under consideration. Specifically, in our setup we construct a proposal density which encourages the algorithm to visit sparse states of the components of the vector $d,$ over which the optimization (\ref{eq:reparopt}) is to be performed.

%Let $M$ represent the number iterations to be performed of the simulated annealing algorithm, and let $T_i,$ $g_i$ $i=1,..,M$ represent the temperature and proposal density for the $i^{th}$ iteration, respectively. The only change in our optimization approach from the classical simulated annealing algorithm, as given on Page 163 of \cite{robert2013monte} is that we require the proposal density $g_i$ to change with every iteration. This change is made in order to encourage sparse solutions of change point difference estimates $\h d_i$'s, as described in the following.

\vspace{1mm}
{\textit{Construction of proposal density}}: In the optimization step of (\ref{eq:reparopt}), the finite state space under consideration is $\{0,1,...,n\}^{\check N}.$ Additionally we intend to construct a proposal density that encourages the algorithm to visit $\check N$ dimensional states with sparse solutions. For this purpose, let $M\ge 1$ be the total number of iterations of the simulated annealing algorithm to be performed, and for any $x=(x_1,...,x_{\check N})^T\in\{0,1,...,n\}^{\check N},$ let $g(x)=\big(g_1(x_1),...,g_{\check N}(x_{\check N})\big)^T$ be the $\check N$ dimensional componentwise density functions, where each component is a discrete uniform density with an inflated probability at zero, i.e., for each $i=1,...,M,$ $j=1,...,\check N,$ define,
\benr\label{eq:proposal}
g_j(x):=g_{j}(x_j;d;b;\pi_{ij})=
\begin{cases}
	\pi_{ij}, & x=0 \\
	{\rm discrete Uniform}, & x_j\in\{l,u\},
\end{cases}
\eenr
where $d=(d_1,...,d_{\check N})^T\in\{0,...,n\}^{\check N},$ $b\in\{0,...,n\}$ and $\pi_{ij}\in[0,1]$ are parameters of this proposal distribution. The lower and upper limits are $l=\max\{0,d_j-b\},$ and $u=\min\{n-\sum_{j=1}^{j-1}d_k,d_j+b\}.$ Here $b$ and $\pi_{ij}$'s are user chosen parameters, where higher values of $b$ allow for larger jumps between states and $\pi_{ij}$'s are zero inflation parameters that encourage sparsity in the $j^{th}$ component. Lastly, the parameter $d$ is the $\check N$ dimensional centering parameter, i.e., realizations from this proposal are roughly centered around the components of $d.$ Note that the limits of the discrete uniform part of the proposal enforce the restriction that any candidate state $d'$ generated by the proposal satisfies $\sum_{j=1}^{\check N}d_j'\le n,$ which is required for the optimization (\ref{eq:reparopt}).

Next we discuss the choice of the zero inflation parameters $\pi_{ij}$'s in the proposal densities. The objective of introducing this zero inflation in the proposal is meant in order to allow the algorithm to visit all combinations of sparse states of the components of $d.$ For this purpose we design a zero inflation mechanism changing with iteration $i,$ as illustrated in Figure \ref{fig:stepbehave.zeroinflate} (for the case $\check N=3$). The zero inflation parameter $\pi_{ij}$ for each component $j$ of proposal is constructed to follow a sine curve oscillating in the interval (0,1), over the iterations $i$'s. Critically, the sine curve corresponding to each component $g_{ij}$ is chosen such that it has a different period of oscillation in comparison to all other components. These varying periods of oscillation create all possible sparsity patterns among the components of the candidate $d,$ i.e., given a large number of periods of the sine curves, any sparse combination of $d$'s will be generated at some iterations between $1,...,M.$ More specifically, for each iteration $i=1,...,M,$ we set
\benr\label{eq:proposal}
\pi_{ij}=0.475\sin\Big(\frac{i2\pi}{Ma_j}\Big)+0.475,
\eenr
here $a_j$ is the number of periods of the sine curve between $1,...,M,$ chosen for the $j^{th}$ component. This completes the necessary requirements to implement simulated annealing. For completeness, we state in Algorithm 3, the simulated annealing algorithm in context of the optimization (\ref{eq:reparopt}).

\begin{figure}[]
	\centering
	\begin{minipage}[b]{0.45\textwidth}
		\includegraphics[width=\textwidth]{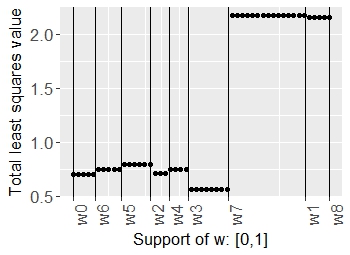}
	\end{minipage}
	\hspace{0.15in}
	\begin{minipage}[b]{0.45\textwidth}
		\includegraphics[width=\textwidth]{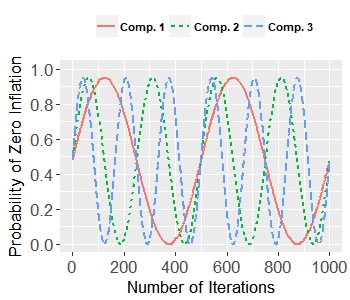}
	\end{minipage}
	\caption{\footnotesize{ {\it Left Panel}: Step behavior of $Q({\check N},\h\al,\tau)$ over $\tau\in {\rm Supp}(w),$ $Q({\check N},\h\al,\tau)$ evaluated over grid of points $\tau\in\{0,0.02,...,1\}.$ Here $w_i\sim\cU(0,1),$ $n=7$ $N=1,$ $\check N=1,$ $p=3,$ $\b^0_{(0)}=(1,0,0)^T,$ $\b^0_{(0)}=(1,1,0)^T,$ $\h\al_{(0)}=(0.41,0,0)^T,$ $\h\al_{(1)}=(0.13,0.92,0)^T,$ $w_0=0,$ $w_8=1.$ Observe that step changes occur at $w_i$'s.} {\it Right Panel}: Construction of zero inflation for proposal density (\ref{eq:proposal}) with $\check N=3.$ Zero inflation probability $\pi_{ij}$ is controlled via a sine curve, where the period of each sine curve is different, thereby producing candidate states $d,$ with all possible sparsity patterns.}
	\label{fig:stepbehave.zeroinflate}
\end{figure}

\begin{figure}[H]
	\noi\rule{\textwidth}{0.5pt}
	
	\vspace{-3mm}
	\begin{flushleft}{\bf Algorithm 4:} Simulated annealing for implementation of optimization (\ref{eq:reparopt})\end{flushleft}
	\vspace{-4.75mm}
	\noi\rule{\textwidth}{0.5pt}
	
	\vspace{-2mm}
	\begin{flushleft} Let $d^i=(d_1^i,...d_{\check N}^i)$ represent the state at the $i^{th}$ iteration, then\end{flushleft}
	
	\vspace{-2.5mm}
	\begin{itemize}[wide, labelwidth=!, labelindent=0pt]
		\item[(i)] Simulate a  candidate $d=(d_1,...,d_{\check N})^T$ from the $\check N$-dimensional proposal density $g(x; d^{i};b;\pi_{i})$ constructed in (\ref{eq:proposal}), where $\pi_i=(\pi_{i1},...,\pi_{i\check N})^T.$
		\item[(ii)] Accept $d^{i+1}=d$ with probability $\rho_i=\exp\big(\Delta h_i/T_i\big)\wedge 1;$ take $d^{i+1}=d^{i}$ otherwise.
		\item[(iii)] Update $T_i$ to $T_{i+1},$ and $\pi_{ij}$ to $\pi_{(i+1)j},$ for each $j=1,...,\check N.$	
	\end{itemize}
	
	\vspace{-3mm}
	\noi\rule{\textwidth}{0.5pt}
\end{figure}

Here $\Delta h_i=h(d^i)-h(d),$ where $h(d)=Q(\check N, \h\al, \tau)+ \mu \sum_{j=1}^{\check N}\|\frac{d_j}{n}\|_0.$ Also, $T_i,$ $i=1,..,n$ represents a user chosen decreasing sequence of positive numbers, which is also commonly referred to as the `temperature function' of simulated annealing. An illustration of the evolution of the simulated annealing algorithm with the above described proposal density for the optimization (\ref{eq:reparopt}) is provided in Figure \ref{fig:simannealevolution}. The following subsection provides numerical results obtained via monte carlo simulations of the methodology described here.

\begin{figure}
	\centering
	\includegraphics[width=0.5\linewidth]{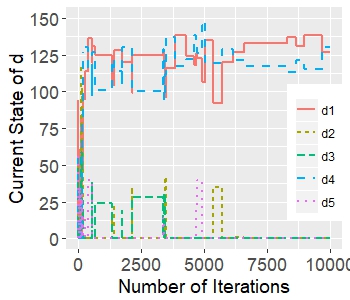}
	\caption{\footnotesize{Illustration of the evolution of simulated annealing for the optimization (\ref{eq:reparopt}) to obtain $\h d=(\h d_1,...,\h d_{\check N})^T.$ Here $n=375,$ $p=50,$ $N=2,$ $\check N=5,$ $\mu=0.25.$ The true change points are located at $\tau_1^0=125$ and $\tau_2^0=250.$ The proposal density is that in (\ref{eq:proposal}), with $a_j=1\big/\big(250+25*(j-1)\big),$ $j=1,...,\check N.$ The temperature function is set to $T_i=1/\log(1+i),$ $i=1,...,10000.$ Observe that all but all but the components $\h d_1,$ $\h d_4$ converge to zero and $\h d_1$ converges to $127$ and $\h d_1+\h d_2+\h d_3$ converges to $257,$ which are near the locations of the true change point parameters. Starting value in the algorithm: $d=(94,56,56,56,56),$ i.e., $\tau=(94,150,206,262,318).$}}
	\label{fig:simannealevolution}
\end{figure}

\begin{rem}  The construction of the surrogate change inducing variable $w_i^*,$ the reparameterizaton of (\ref{model:reparcp}) and (\ref{eq:reparopt}) is required only to avoid coding complexity of the algorithm. In general, a similar simulated annealing approach can be easily developed for directly implementing {\bf Step 1} of Algorithm 1 on the state space $\{-\iny,w_1,...,w_n,\iny\}^{\check N},$ however to avoid redundancy, these details are omitted.
\end{rem}

\subsection{Numerical Results}

The main objective of the monte carlo simulations of this section are to assess the empirical performance of Algorithm 1 of the proposed method, which performs the detection and estimation of change points in the assumed model. We do not perform simulations for Algorithm 3 of the process since this step is an ordinary lasso optimization, whose empirical validity has been established in the literature via an innumerable number of simulations.

In view of the reparameterization described earlier in this section, we consider the data generating process given in (\ref{model:reparcp}). The r.v.'s $\vep_i,$ $w_i$ and $x_i$ are drawn independently satisfying $\vep_i\sim\cN(0,\si_{\vep}^2),$ and $x_i\sim\cN(0,\Si).$ Here, $\Sigma$ is a $p \times p$ matrix with elements $\Sigma_{ij}=\rho^{|i-j|},$ $i,j=1,...,p.$ We set, $\si_{\vep}=1$ and $\rho = 0.5.$ The number of change points $N$ is set to one of $\{0,1,...,4\},$ i.e., we consider one to five segment models. The case of $N=0,$ where no change points are assumed is only a detection problem, as opposed to the remaining cases where the objective is both detection and estimation of change points. The change point parameters are assumed to be equally spaced in $(0,1),$ specifically, we set $\tau_1^{\dagger}=\frac{1}{N+1},\tau_2^{\dagger}= \frac{2}{N+1},...,\tau_N^{\dagger}=\frac{N}{N+1}.$ Simulations are performed for all combinations of the parameters $p\in\{50,175,300\},$ and $n\in\{250,375,500,625\}.$ Note that the total number of parameters to be estimated for each combination of $(p,N)$ is $p(N+1)+N.$ The regression coefficients are set in the following manner. The even numbered regression coefficient vectors $\b^0_{2j}=(1_{1\times 5}, 0_{1\times(p-5)})^T,$ for all $j\ge 0,$ such that $2j\le N,$ and the odd numbered coefficient vectors are chosen as $\b^0_{2j+1}=(0_{1\times 5},1_{1\times 5}, 0_{1\times(p-5)})^T,$ for all $j\ge 0,$ such that $2j+1\le N.$  Here $0_{1\times 5}=(0,...,0)_{1\times 5},$ $1_{1\times 5}=(1,...,1)_{1\times 5},$ and  $0_{1\times(p-5)}=(0,...,0)_{1\times(p-5)}.$ The initial number of change points assumed in {\bf Step 0} of Algorithm 1a are set to $\check N=4,5,6,7$ for $n=250,375,500,625,$ respectively. Finally, the parameters of the simulated annealing optimization are chosen as follows. The total number of iterations performed under simulated annealing is set to $M=10,000,$ the temperature function (over iterations) is set to $T_i=1\big/(temp*\log(1+i)),$ $temp=1.25,$ $i=1,...,M.$ The period of sine curves constructed for zero inflation of the proposal density described in Section \ref{sec:simanneal} is chosen as $a_j=1\big/\big(250+25*(j-1)\big),$ $j=1,...,\check N,$ i.e., the first component of proposal completed $250$ oscillations within the $M$ iterations and each following component has $25$ more oscillations than the previous.

All results are based on 100 monte carlo repetitions. Computations are performed in the software R, \cite{rcite}. All lasso optimizations are performed with the R package `glmnet', developed by \cite{friedmanglmnet}. For reporting our results, we compute monte carlo approximations of the following metrics. On the detection of change points:  Probability of match ${\rm (PrM)}= E{\bf 1}[\tilde N=N],$ Probability of exceeding ${\rm (PrE)}= E{\bf 1}[\tilde N>N],$ and Probability of lower number ${\rm (PrL)}= E{\bf 1}[\tilde N<N],$ ${\rm Bias}(\tilde N)= E(\tilde N-N),$ ${\rm RMSE}(\tilde N)=\big(E(\tilde N- N)^2\big)^{\frac{1}{2}}.$ On the estimation of location of change points conditioned on correct recovery of the number of change points:  ${\rm Bias(L)}=\|E\big(\tilde\tau-\tau\big|\tilde N=N\big)\|_2,$ and ${\rm RMSE(L)}=\|\big(E\big((\tilde\tau-\tau)^2\big|\tilde N=N\big)\big)^{\frac{1}{2}}\|_1.$

\vspace{1mm}
{\textit{Choice of tuning parameters $\la_0,\la_1,\mu$}}: For lasso optimization of {\bf Step 0}, the regularization parameter $\la_0$ is chosen via a 5-fold cross validation (performed internally by the R package `glmnet'). Next, we use a BIC-type criteria to choose the regularizer $\mu$ of {\bf Step 1} of Algorithm 1a. Specifically, let $\h d(\mu)$ represent the solution of (\ref{eq:reparopt}) and $\h\tau(\mu)$ be the corresponding change point solution, then we choose $\mu$ as argument that minimizes the criteria,
\benr
{\rm BIC}(\mu)= \log\big(Q\big(\check N,\h\al,\h\tau(\mu)\big)\big) + c\frac{\|\h d(\mu)\|_0\log n}{n}
\eenr
Here we set $c=10,$ which performs well in all empirically examined cases.

The simulation results for $p=50,175,300$ are reported in Table \ref{tab:p50}, Table \ref{tab:p175} and Table \ref{tab:p300}, respectively. The results are encouraging and supportive of our theoretical findings. For nearly all examined cases, in $\approx 80\%$ of all simulations, the estimated number of change points match exactly with the unknown number of change points. In cases where there is mismatch between $\tilde N$ and $N,$ it can be approximated from ${\rm Bias}(\tilde N)$ and ${\rm RMSE}(\tilde N),$ that the proposed procedure misses the unknown number of change points by $\approx 1$ change point. In these cases of mismatch, it is also observed that under the given settings, $\tilde N$ exceeds $N,$ indicating the BIC selection criteria can be further tightened by increasing the value of the constant chosen in its definition. Additionally, it is also observed from  ${\rm Bias(L)},$ and ${\rm RMSE(L)}$ that the components of $\tilde\tau$ precisely converge toward the locations of the unknown change points, however, as expected some deterioration in accuracy is observed as $p$ increases.

\begin{table}[]
	\caption{Numerical results on the performance of Algorithm 1 in estimating the number of change points $N$ and their locations $\tau^0,$ when $p=50.$}
	\label{tab:p50}
	\begin{tabular}{cccccccccc}
		\hline
		$n$ & $N$ & $\check N$ & PrM  & PrE  & PrL  & Bias(N) & RMSE(N)  & Bias(L)  & RMSE(L)  \\ \hline
		250 & 0   & 4          & 0.82 & 0.18 & 0    & 0.19    & 0.458258 & NA       & NA       \\
		250 & 1   & 4          & 0.94 & 0.06 & 0    & 0.06    & 0.244949 & 8.51E-05 & 0.007449 \\ \hline
		375 & 0   & 5          & 0.93 & 0.07 & 0    & 0.09    & 0.360555 & NA       & NA       \\
		375 & 1   & 5          & 0.97 & 0.03 & 0    & 0.03    & 0.173205 & 0.000137 & 0.005215 \\
		375 & 2   & 5          & 0.95 & 0.05 & 0    & 0.05    & 0.223607 & 0.000465 & 0.013372 \\ \hline
		500 & 0   & 6          & 0.94 & 0.06 & 0    & 0.08    & 0.34641  & NA       & NA       \\
		500 & 1   & 6          & 0.98 & 0.02 & 0    & 0.03    & 0.223607 & 0.000857 & 0.004755 \\
		500 & 2   & 6          & 0.92 & 0.08 & 0    & 0.08    & 0.282843 & 0.000724 & 0.01201  \\
		500 & 3   & 6          & 0.93 & 0.07 & 0    & 0.07    & 0.264575 & 0.001743 & 0.030009 \\ \hline
		625 & 0   & 7          & 0.97 & 0.03 & 0    & 0.04    & 0.244949 & NA       & NA       \\
		625 & 1   & 7          & 0.94 & 0.06 & 0    & 0.06    & 0.244949 & 0.000749 & 0.004155 \\
		625 & 2   & 7          & 0.76 & 0.24 & 0    & 0.26    & 0.547723 & 0.000426 & 0.012903 \\
		625 & 3   & 7          & 0.55 & 0.45 & 0    & 0.51    & 0.806226 & 0.001581 & 0.033614 \\
		625 & 4   & 7          & 0.72 & 0.28 & 0    & 0.3     & 0.583095 & 0.003469 & 0.070702 \\ \hline
	\end{tabular}
\end{table}

\begin{table}[]
	\caption{Numerical results on the performance of Algorithm 1 in estimating the number of change points $N$ and their locations $\tau^0,$ when $p=175.$}
	\label{tab:p175}
	\begin{tabular}{cccccccccc}
		\hline
		$n$ & $N$ & $\check N$ & PrM  & PrE  & PrL  & Bias(N) & RMSE(N)  & Bias(L)  & RMSE(L)  \\ \hline
		250 & 0   & 4          & 0.74 & 0.26 & 0    & 0.28    & 0.565685 & NA       & NA       \\
		250 & 1   & 4          & 0.93 & 0.07 & 0    & 0.07    & 0.264575 & 0.001462 & 0.006941 \\ \hline
		375 & 0   & 5          & 0.87 & 0.13 & 0    & 0.13    & 0.360555 & NA       & NA       \\
		375 & 1   & 5          & 0.95 & 0.05 & 0    & 0.05    & 0.223607 & 0.00073  & 0.006646 \\
		375 & 2   & 5          & 0.97 & 0.03 & 0    & 0.03    & 0.173205 & 0.000505 & 0.013853 \\ \hline
		500 & 0   & 6          & 0.9  & 0.1  & 0    & 0.14    & 0.489898 & NA       & NA       \\
		500 & 1   & 6          & 0.92 & 0.08 & 0    & 0.1     & 0.374166 & 0.000283 & 0.005345 \\
		500 & 2   & 6          & 0.92 & 0.08 & 0    & 0.09    & 0.331662 & 0.000997 & 0.014112 \\
		500 & 3   & 6          & 0.98 & 0.02 & 0    & 0.02    & 0.141421 & 0.002776 & 0.031699 \\ \hline
		625 & 0   & 7          & 0.94 & 0.06 & 0    & 0.08    & 0.34641  & NA       & NA       \\
		625 & 1   & 7          & 0.96 & 0.04 & 0    & 0.04    & 0.2      & 1.58E-20 & 0.004888 \\
		625 & 2   & 7          & 0.73 & 0.27 & 0    & 0.28    & 0.547723 & 0.001232 & 0.015207 \\
		625 & 3   & 7          & 0.56 & 0.44 & 0    & 0.51    & 0.818535 & 0.001971 & 0.031628 \\
		625 & 4   & 7          & 0.77 & 0.23 & 0    & 0.23    & 0.479583 & 0.001937 & 0.076667 \\ \hline
	\end{tabular}
\end{table}

\begin{table}[]
	\caption{Numerical results on the performance of Algorithm 1 in estimating the number of change points $N$ and their locations $\tau^0,$ when $p=300.$}
	\label{tab:p300}
	\begin{tabular}{cccccccccc}
		\hline
		$n$ & $N$ & $\check N$ & PrM  & PrE  & PrL  & Bias(N) & RMSE(N)  & Bias(L)  & RMSE(L)  \\ \hline
		250 & 0   & 4          & 0.65 & 0.35 & 0    & 0.38    & 0.663325 & NA       & NA       \\
		250 & 1   & 4          & 0.93 & 0.07 & 0    & 0.07    & 0.264575 & 0.000344 & 0.006479 \\ \hline
		375 & 0   & 5          & 0.81 & 0.19 & 0    & 0.23    & 0.556776 & NA       & NA       \\
		375 & 1   & 5          & 0.97 & 0.03 & 0    & 0.03    & 0.173205 & 0.000687 & 0.00773  \\
		375 & 2   & 5          & 0.94 & 0.06 & 0    & 0.06    & 0.244949 & 0.00205  & 0.015561 \\ \hline
		500 & 0   & 6          & 0.87 & 0.13 & 0    & 0.16    & 0.469042 & NA       & NA       \\
		500 & 1   & 6          & 0.92 & 0.08 & 0    & 0.09    & 0.331662 & 0.000739 & 0.005421 \\
		500 & 2   & 6          & 0.93 & 0.07 & 0    & 0.07    & 0.264575 & 0.000882 & 0.015368 \\
		500 & 3   & 6          & 0.95 & 0.05 & 0    & 0.05    & 0.223607 & 0.001901 & 0.033745 \\ \hline
		625 & 0   & 7          & 0.86 & 0.14 & 0    & 0.19    & 0.538516 & NA       & NA       \\
		625 & 1   & 7          & 0.92 & 0.08 & 0    & 0.1     & 0.374166 & 0.000313 & 0.005492 \\
		625 & 2   & 7          & 0.82 & 0.18 & 0    & 0.18    & 0.424264 & 0.001554 & 0.015243 \\
		625 & 3   & 7          & 0.53 & 0.47 & 0    & 0.5     & 0.748331 & 0.0013   & 0.030562 \\
		625 & 4   & 7          & 0.77 & 0.23 & 0    & 0.25    & 0.538516 & 0.002367 & 0.069581 \\ \hline
	\end{tabular}
\end{table}

\section{Discussion} Dynamic high dimensional regression models which are characterized via change points, provide an intuitive modelling approach that allows for dynamic behavior of parameters. These models allow for much greater versatility of the assumed model, and consequently a greater fidelity to the data structure. These models have been sparsely used in applications due to gaps in theoretical understanding and a lack of availability of efficient methods for estimation of parameters for such models. This article serves to fill this void. We develop a novel methodology for the detection and estimation of multiple change points in high dimensional linear regression models. The proposed method is theoretically sound and empirically more efficient than methods currently available in the literature. The idea of {\it arbitrary segmentation} is not restricted to regression models and the proposed methodology could potentially be developed for other relevant models such as dynamic networks. Two technical questions remained unanswered. First, what is optimal rate of regularized change point estimates in a high dimensional setting such as the one considered in this article. Second, is there  theoretical validity of a BIC type criteria for the selection of the regularization parameter in the $\ell_0$ regularization considered in this article. However these questions are left open for further investigations.

\bibliography{multiplecp}

\begin{thebibliography}{}

\bibitem[\protect\citename{Atchade \& Bybee, }2017]{atchade2017scalable}
{\sc Atchade, Yves, \& Bybee, Leland}. 2017.
\newblock A scalable algorithm for gaussian graphical models with
  change-points.
\newblock {\em arxiv preprint arxiv:1707.04306}.

\bibitem[\protect\citename{Bai, }1997]{bai1997estimation}
{\sc Bai, Jushan}. 1997.
\newblock Estimation of a change point in multiple regression models.
\newblock {\em Review of economics and statistics}, {\bf 79}(4), 551--563.

\bibitem[\protect\citename{Belloni {\em et~al.\ }\relax,
  }2011]{belloni2011square}
{\sc Belloni, Alexandre, Chernozhukov, Victor, \& Wang, Lie}. 2011.
\newblock Square-root lasso: pivotal recovery of sparse signals via conic
  programming.
\newblock {\em Biometrika}, {\bf 98}(4), 791--806.

\bibitem[\protect\citename{Belloni {\em et~al.\ }\relax,
  }2017a]{belloni2017linear}
{\sc Belloni, Alexandre, Rosenbaum, Mathieu, \& Tsybakov, Alexandre~B.} 2017a.
\newblock Linear and conic programming estimators in high dimensional
  errors-in-variables models.
\newblock {\em Journal of the royal statistical society: Series b (statistical
  methodology)}, {\bf 79}(3), 939--956.

\bibitem[\protect\citename{Belloni {\em et~al.\ }\relax,
  }2017b]{belloni2017pivotal}
{\sc Belloni, Alexandre, Chernozhukov, Victor, Kaul, Abhishek, Rosenbaum,
  Mathieu, \& Tsybakov, Alexandre~B.} 2017b.
\newblock Pivotal estimation via self-normalization for high-dimensional linear
  models with error in variables.
\newblock {\em arxiv preprint arxiv:1708.08353}.

\bibitem[\protect\citename{Bickel {\em et~al.\ }\relax,
  }2009]{bickel2009simultaneous}
{\sc Bickel, Peter~J., Ritov, Ya’acov, Tsybakov, Alexandre~B., {\em et~al.\
  }\relax}. 2009.
\newblock Simultaneous analysis of lasso and dantzig selector.
\newblock {\em The annals of statistics}, {\bf 37}(4), 1705--1732.

\bibitem[\protect\citename{B{\"u}hlmann \& Van De~Geer,
  }2011]{buhlmann2011statistics}
{\sc B{\"u}hlmann, Peter, \& Van De~Geer, Sara}. 2011.
\newblock {\em Statistics for high-dimensional data: methods, theory and
  applications}.
\newblock Springer Science \& Business Media.

\bibitem[\protect\citename{Cho \& Fryzlewicz, }2015]{cho2015multiple}
{\sc Cho, Haeran, \& Fryzlewicz, Piotr}. 2015.
\newblock Multiple-change-point detection for high dimensional time series via
  sparsified binary segmentation.
\newblock {\em Journal of the royal statistical society: Series b (statistical
  methodology)}, {\bf 77}(2), 475--507.

\bibitem[\protect\citename{Ciuperca, }2014]{ciuperca2014model}
{\sc Ciuperca, Gabriela}. 2014.
\newblock Model selection by lasso methods in a change-point model.
\newblock {\em Statistical papers}, {\bf 55}(2), 349--374.

\bibitem[\protect\citename{Durrett, }2010]{durrett2010probability}
{\sc Durrett, Rick}. 2010.
\newblock {\em Probability: theory and examples}.
\newblock Cambridge university press.

\bibitem[\protect\citename{Friedman {\em et~al.\ }\relax,
  }2010]{friedmanglmnet}
{\sc Friedman, Jerome, Hastie, Trevor, \& Tibshirani, Robert}. 2010.
\newblock glmnet: lasso and elastic-net regularized generalized linear models,
  2010b.
\newblock {\em Url http://cran. r-project. org/package= glmnet. r package
  version},  1--1.

\bibitem[\protect\citename{Fryzlewicz, }2014]{fryzlewicz2014wild}
{\sc Fryzlewicz, Piotr}. 2014.
\newblock Wild binary segmentation for multiple change-point detection.
\newblock {\em The annals of statistics}, {\bf 42}(6), 2243--2281.

\bibitem[\protect\citename{Gautier \& Tsybakov, }2011]{gautier2011high}
{\sc Gautier, Eric, \& Tsybakov, Alexandre}. 2011.
\newblock High-dimensional instrumental variables regression and confidence
  sets.
\newblock {\em arxiv preprint arxiv:1105.2454}.

\bibitem[\protect\citename{Gibberd \& Roy, }2017]{gibberd2017multiple}
{\sc Gibberd, Alex~J., \& Roy, Sandipan}. 2017.
\newblock Multiple changepoint estimation in high-dimensional gaussian
  graphical models.
\newblock {\em arxiv preprint arxiv:1712.05786}.

\bibitem[\protect\citename{Hastie {\em et~al.\ }\relax,
  }2015]{hastie2015statistical}
{\sc Hastie, Trevor, Tibshirani, Robert, \& Wainwright, Martin}. 2015.
\newblock {\em Statistical learning with sparsity: the lasso and
  generalizations}.
\newblock CRC press.

\bibitem[\protect\citename{Hinkley, }1969]{hinkley1969inference}
{\sc Hinkley, David~V.} 1969.
\newblock Inference about the intersection in two-phase regression.
\newblock {\em Biometrika}, {\bf 56}(3), 495--504.

\bibitem[\protect\citename{Hinkley, }1970]{hinkley1970inference}
{\sc Hinkley, David~V.} 1970.
\newblock Inference about the change-point in a sequence of random variables.
\newblock {\em Biometrika}.

\bibitem[\protect\citename{Hinkley, }1972]{hinkley1972time}
{\sc Hinkley, David~V.} 1972.
\newblock Time-ordered classification.
\newblock {\em Biometrika}, {\bf 59}(3), 509--523.

\bibitem[\protect\citename{Jandhyala \& Fotopoulos,
  }1999]{jandhyala1999capturing}
{\sc Jandhyala, Venkata~K., \& Fotopoulos, Stergios~B.} 1999.
\newblock Capturing the distributional behaviour of the maximum likelihood
  estimator of a changepoint.
\newblock {\em Biometrika}, {\bf 86}(1), 129--140.

\bibitem[\protect\citename{Jandhyala \& MacNeill, }1997]{jandhyala1997iterated}
{\sc Jandhyala, Venkata~K., \& MacNeill, Ian~B.} 1997.
\newblock Iterated partial sum sequences of regression residuals and tests for
  changepoints with continuity constraints.
\newblock {\em Journal of the royal statistical society: Series b (statistical
  methodology)}, {\bf 59}(1), 147--156.

\bibitem[\protect\citename{Jandhyala {\em et~al.\ }\relax,
  }2013]{jandhyala2013inference}
{\sc Jandhyala, Venkata~K., Fotopoulos, Stergios~B., MacNeill, Ian~B., \& Liu,
  Pengyu}. 2013.
\newblock Inference for single and multiple change-points in time series.
\newblock {\em Journal of time series analysis}, {\bf 34}(4), 423--446.

\bibitem[\protect\citename{Jin {\em et~al.\ }\relax, }2016]{jin2016consistent}
{\sc Jin, Baisuo, Wu, Yuehua, \& Shi, Xiaoping}. 2016.
\newblock Consistent two-stage multiple change-point detection in linear
  models.
\newblock {\em Canadian journal of statistics}, {\bf 44}(2), 161--179.

\bibitem[\protect\citename{Kaul, }2014]{kaul2014lasso}
{\sc Kaul, Abhishek}. 2014.
\newblock Lasso with long memory regression errors.
\newblock {\em Journal of statistical planning and inference}, {\bf 153},
  11--26.

\bibitem[\protect\citename{Kaul \& Koul, }2015]{kaul2015weighted}
{\sc Kaul, Abhishek, \& Koul, Hira~L.} 2015.
\newblock Weighted ℓ1-penalized corrected quantile regression for high
  dimensional measurement error models.
\newblock {\em Journal of multivariate analysis}, {\bf 140}, 72--91.

\bibitem[\protect\citename{Kaul {\em et~al.\ }\relax,
  }2017]{kaul2017structural}
{\sc Kaul, Abhishek, Davidov, Ori, \& Peddada, Shyamal~D.} 2017.
\newblock Structural zeros in high-dimensional data with applications to
  microbiome studies.
\newblock {\em Biostatistics}, {\bf 18}(3), 422--433.

\bibitem[\protect\citename{Kaul {\em et~al.\ }\relax, }2019]{kaul2018parameter}
{\sc Kaul, Abhishek, Jandhyala, Venkata~K., \& Fotopoulos, Stergios~B.} 2019.
\newblock An efficient two step algorithm for high dimensional change point
  regression models without grid search.
\newblock {\em Journal of machine learning research (to appear), arxiv preprint
  arxiv:1805.03719}.

\bibitem[\protect\citename{Koenker \& Mizera, }2014]{koenker2014convex}
{\sc Koenker, Roger, \& Mizera, Ivan}. 2014.
\newblock Convex optimization in r.
\newblock {\em Journal of statistical software}, {\bf 60}(5), 1--23.

\bibitem[\protect\citename{Koul \& Qian, }2002]{koul2002asymptotics}
{\sc Koul, Hira~L., \& Qian, Lianfen}. 2002.
\newblock Asymptotics of maximum likelihood estimator in a two-phase linear
  regression model.
\newblock {\em Journal of statistical planning and inference}, {\bf 108}(1-2),
  99--119.

\bibitem[\protect\citename{Koul {\em et~al.\ }\relax,
  }2003]{koul2003asymptotics}
{\sc Koul, Hira~L., Qian, Lianfen, \& Surgailis, Donatas}. 2003.
\newblock Asymptotics of m-estimators in two-phase linear regression models.
\newblock {\em Stochastic processes and their applications}, {\bf 103}(1),
  123--154.

\bibitem[\protect\citename{Lee {\em et~al.\ }\relax, }2016]{lee2016lasso}
{\sc Lee, Sokbae, Seo, Myung~Hwan, \& Shin, Youngki}. 2016.
\newblock The lasso for high dimensional regression with a possible change
  point.
\newblock {\em Journal of the royal statistical society: Series b (statistical
  methodology)}, {\bf 78}(1), 193--210.

\bibitem[\protect\citename{Lee {\em et~al.\ }\relax, }2018]{lee2018}
{\sc Lee, Sokbae, Liao, Yuan, Seo, Myung~Hwan, \& Shin, Youngki}. 2018.
\newblock Oracle estimation of a change point in high-dimensional quantile
  regression.
\newblock {\em Journal of the american statistical association}, {\bf 0}(0),
  1--11.

\bibitem[\protect\citename{Leonardi \& B{\"u}hlmann,
  }2016]{leonardi2016computationally}
{\sc Leonardi, Florencia, \& B{\"u}hlmann, Peter}. 2016.
\newblock Computationally efficient change point detection for high-dimensional
  regression.
\newblock {\em arxiv preprint arxiv:1601.03704}.

\bibitem[\protect\citename{Loh \& Wainwright, }2012]{loh2012}
{\sc Loh, Po-Ling, \& Wainwright, Martin~J.} 2012.
\newblock High-dimensional regression with noisy and missing data: Provable
  guarantees with nonconvexity.
\newblock {\em Ann. statist.}, {\bf 40}(3), 1637--1664.

\bibitem[\protect\citename{Maurer, }2003]{maurer2003bound}
{\sc Maurer, Andreas}. 2003.
\newblock A bound on the deviation probability for sums of non-negative random
  variables.
\newblock {\em J. inequalities in pure and applied mathematics}, {\bf 4}(1),
  15.

\bibitem[\protect\citename{{R Core Team}, }2017]{rcite}
{\sc {R Core Team}}. 2017.
\newblock {\em R: A language and environment for statistical computing}.
\newblock R Foundation for Statistical Computing, Vienna, Austria.

\bibitem[\protect\citename{Raskutti {\em et~al.\ }\relax,
  }2010]{raskutti2010restricted}
{\sc Raskutti, Garvesh, Wainwright, Martin~J., \& Yu, Bin}. 2010.
\newblock Restricted eigenvalue properties for correlated gaussian designs.
\newblock {\em Journal of machine learning research}, {\bf 11}(Aug),
  2241--2259.

\bibitem[\protect\citename{Raskutti {\em et~al.\ }\relax,
  }2011]{raskutti2011minimax}
{\sc Raskutti, Garvesh, Wainwright, Martin~J., \& Yu, Bin}. 2011.
\newblock Minimax rates of estimation for high-dimensional linear regression
  over $l_q$-balls.
\newblock {\em Ieee transactions on information theory}, {\bf 57}(10),
  6976--6994.

\bibitem[\protect\citename{Robert \& Casella, }2013]{robert2013monte}
{\sc Robert, Christian, \& Casella, George}. 2013.
\newblock {\em Monte carlo statistical methods}.
\newblock Springer Science \& Business Media.

\bibitem[\protect\citename{Roy {\em et~al.\ }\relax, }2017]{roy2017change}
{\sc Roy, Sandipan, Atchad{\'e}, Yves, \& Michailidis, George}. 2017.
\newblock Change point estimation in high dimensional markov random-field
  models.
\newblock {\em Journal of the royal statistical society: Series b (statistical
  methodology)}, {\bf 79}(4), 1187--1206.

\bibitem[\protect\citename{Rudelson \& Zhou, }2012]{rudelson2012reconstruction}
{\sc Rudelson, Mark, \& Zhou, Shuheng}. 2012.
\newblock Reconstruction from anisotropic random measurements.
\newblock {\em Pages  10--1 of:} {\em Conference on learning theory}.

\bibitem[\protect\citename{Sasaki, }1987]{sasaki1987optimization}
{\sc Sasaki, Galen~H.} 1987.
\newblock {\em Optimization by simulated annealing: A time-complexity
  analysis.}
\newblock Tech. rept. ILLINOIS UNIV AT URBANA DEPT OF ELECTRICAL ENGINEERING.

\bibitem[\protect\citename{Tibshirani, }1996]{tibshirani1996regression}
{\sc Tibshirani, Robert}. 1996.
\newblock Regression shrinkage and selection via the lasso.
\newblock {\em Journal of the royal statistical society. series b
  (methodological)},  267--288.

\bibitem[\protect\citename{Tibshirani, }2011]{tibshirani2011regression}
{\sc Tibshirani, Robert}. 2011.
\newblock Regression shrinkage and selection via the lasso: a retrospective.
\newblock {\em Journal of the royal statistical society: Series b (statistical
  methodology)}, {\bf 73}(3), 273--282.

\bibitem[\protect\citename{Vershynin, }2010]{vershynin2010introduction}
{\sc Vershynin, Roman}. 2010.
\newblock Introduction to the non-asymptotic analysis of random matrices.
\newblock {\em arxiv preprint arxiv:1011.3027}.

\bibitem[\protect\citename{Wang \& Samworth, }2018]{wang2018high}
{\sc Wang, Tengyao, \& Samworth, Richard~J.} 2018.
\newblock High dimensional change point estimation via sparse projection.
\newblock {\em Journal of the royal statistical society: Series b (statistical
  methodology)}, {\bf 80}(1), 57--83.

\bibitem[\protect\citename{Ye \& Zhang, }2010]{ye2010rate}
{\sc Ye, Fei, \& Zhang, Cun-Hui}. 2010.
\newblock Rate minimaxity of the lasso and dantzig selector for the lq loss in
  lr balls.
\newblock {\em Journal of machine learning research}, {\bf 11}(Dec),
  3519--3540.

\bibitem[\protect\citename{Zhang {\em et~al.\ }\relax,
  }2015]{zhang2015multiple}
{\sc Zhang, Bingwen, Geng, Jun, \& Lai, Lifeng}. 2015.
\newblock Multiple change-points estimation in linear regression models via
  sparse group lasso.
\newblock {\em Ieee trans. signal processing}, {\bf 63}(9), 2209--2224.

\bibitem[\protect\citename{Zhao \& Yu, }2006]{zhao2006model}
{\sc Zhao, Peng, \& Yu, Bin}. 2006.
\newblock On model selection consistency of lasso.
\newblock {\em Journal of machine learning research}, {\bf 7}(Nov), 2541--2563.

\bibitem[\protect\citename{Zou, }2006]{zou2006adaptive}
{\sc Zou, Hui}. 2006.
\newblock The adaptive lasso and its oracle properties.
\newblock {\em Journal of the american statistical association}, {\bf
  101}(476), 1418--1429.

\end{thebibliography}
\bibliographystyle{\style}

\pagebreak

\setcounter{page}{1}

\begin{center}
	{\sc Supplementary Materials for ``Detection and estimation of parameters in high dimensional multiple change point regression models via $\ell_1\big/\ell_0$ regularization and discrete optimization"}
\end{center}	

\appendix

\section{Proofs of Section 3}

%%%%%%%%%%%%%%%%%%%%%%%%%%%%%%%%%%%%%%%%%%%%%%%%%%%%%%%%%%%%%%%%%%%%%%%%%%%%%%%%%%%%%%%%%%%%%%%%%%%%%%%%%%%%%%%%%%%%%%%%%%%%%%%%%%%%%%%%%%%%%%%%%%%%%%%%%%%%%%%%%%%%%%%%%%%%	
\vspace{2mm}
\begin{proof}[Proof of Lemma \ref{lem:indicatorbound}]
	
	We begin by proving Part (i) of this lemma. Since $\bar\R$ is compact under the metric $\Phi(\cdot),$ divide the space $\bar\R$ into $l=1/2u_n$ closed intervals (disjoint except at the boundaries), each of length $2u_n.$ Let $\tau_1,...\tau_{l}$ be fixed points which represent the centres of these intervals. We shall show that the following bound holds,
	\benr\label{eq:ballbound}
	\max_{j=1,...,l}\sup_{\substack{\tau\in\bar\R;\\\tau\in\cB(\tau_j,u_n)}}\frac{1}{n}\sti \z_i(\tau)\le c_u\max\Big\{\frac{\log p}{n}, u_n\Big\},
	\eenr
	with probability at least $1-c_1\exp(-c_2\log p),$ for $n$ sufficiently large. Assuming (\ref{eq:ballbound}), observe that any $\tau_a,\tau_b\in\bar\R$ satisfying $d(\tau_a,\tau_b)\le u_n,$ must lie in atmost two adjacent intervals $\cB(\tau_j,u_n)\cup\cB(\tau_{j+1},u_n),$ for some $j=1,...,l-1.$ This implies that
	\benr
	\sup_{\substack{\tau_a,\tau_b\in\bar\R;\\d(\tau_a,\tau_b)\le u_n}}\frac{1}{n}\sti \z_i(\tau_a,\tau_b)\le 2\max_{j=1,...,l}\sup_{\substack{\tau\in\bar\R;\\\tau\in\cB(\tau_j,u_n)}}\frac{1}{n}\sti \z_i(\tau)\le c_u\max\Big\{\frac{\log p}{n}, u_n\Big\},\nn
	\eenr
	with probability at least $1-c_1\exp(-c_2\log p).$ Thus to prove part (i), it only remains to prove (\ref{eq:ballbound}), this is done in the following. Consider a fixed $j\in\{1,...,l\}$ and
	let $\tau_a>\tau_j$ be a boundary point on the right of $\tau_j,$ such that $d(\tau_j,\tau_a)=u_n.$ Then note that $p_n:=E\z_i(\tau_a,\tau_j)=d(\tau_a,\tau_j).$ Since $\z_i(\tau_a,\tau_b),$ $i=1,...,n$ are Bernoulli r.v.'s, hence for any $s>0,$ the moment generating function is given by $E\exp\big(s\z_i(\tau_a,\tau_j)\big)=q_n+p_n\exp(s),$ where $q_n=1-p_n.$ Applying the Chernoff Inequality, we obtain,
	\benr
	P\big(\sti \z_i(\tau_a,\tau_j) > t+np_n\big)&=&P\big(e^{\sti s\z_i(\tau_a,\tau_j)}> e^{(st+snp_n)}\big)\nn\\
	&\le& e^{-s(t+np_n)}[q_n+p_n e^s]^n.\nn
	\eenr
	Now, in order to show,
	\benr\lel{l1enn1}
	\frac{1}{n}\sti \z_i(\tau_a,\tau_j)\le c_u\max\Big\{\frac{\log p}{n}, u_n\Big\}
	\eenr
	with probability at least $1-c_1\exp(-c_2\log p),$ we divide the argument into two cases. First, when $d(\tau_a,\tau_j)\ge c\log p/n,$ for some constant $c>0.$ In this case, upon choosing $t=nd(\tau_a,\tau_j)$ we obtain,
	\benr
	P\big(\sti \z_i(\tau_a,\tau_j)> 2nd(\tau_j,\tau_a)\big) \le e^{[- 2snd(\tau_a,\tau_j)]}[1+(d(\tau_a,\tau_j))(e^s-1)]^n.\nn
	\eenr
	Using the deterministic inequality $(1+x)^k\le \exp(kx),$ for any $k,x>0,$ we obtain that
	\benr
	P\big(\sti\z_i(\tau_a,\tau_j)> 2nd(\tau_a,\tau_j)\big) \le  e^{- 2snd(\tau_a,\tau_j)}e^{(e^{s}-1)nd(\tau_a,\tau_j)}  \le  e^{-c_2\log p}.\nn
	\eenr
	The inequality to the right follows by choosing $s=\log 2,$ which maximizes the function $f(s)=2s-e^s+1$ and provides a positive value at the maximum, and by using the restriction $d(\tau_a,\tau_j)\ge c\log p/n.$ Next, when $d(\tau_a,\tau_)< c\log p/n.$ Here choose $t=c\log p$ to obtain,
	\benr\lel{l1e1}
	P\big(\sti \z_i(\tau_a,\tau_j)> c\log p+ nd(\tau_a,\tau_j)\big)\hspace{1.5in}\nn\\
	\le e^{[-sc\log p - snd(\tau_a,\tau_j)]}[1+(d(\tau_a,\tau_j))(e^s-1)]^n.
	\eenr
	Calling upon the inequality $(1+x)^k\le \exp(kx),$ for any $k,x>0,$ we can bound the RHS of (\r{l1e1}) from above by $\exp\big[-s c\log p+(e^s-s-1) \log p\big].$ Now  $s=\log(1+c)$ provides a positive value at the maximum, since it maximizes $f(s)=(1+c)s-e^s+1.$ Then for any $c>0,$ we obtain,
	\benr
	P\big(\sti\z_i(\tau_a,\tau_j)> c\log p+ nd(\tau_a,\tau_j)\big) &\le & e^{-c_2\log p}. \nn
	\eenr
	Upon combining both cases, (\r{l1enn1}) follows by noting $d(\tau_a,\tau_j)=u_n. $
	
	Now repeating the same argument for a fixed boundary point $\tau_b$ on the left of $\tau_j,$ such that $d(\tau_b,\tau_j)=u_n,$ and applying a union bound we obtain,
	\benr\lel{l1e2}
	\max_{\tau\in\{\tau_a,\tau_b\}}\frac{1}{n}\sti \z_i(\tau,\tau_j)\le c_u\max\Big\{\frac{\log p}{n}, u_n\Big\}
	\eenr
	with probability at least $1- c_1\exp(-c_2\log p).$	In order to show that (\r{l1enn1}) holds uniformly over $\cB(\tau_j,u_n).$ For this, we begin by noting that for any $\tau\in\cB(\tau_j,u_n),$ where $\tau>\tau_j$ we have
	$\z_i(\tau,\tau_j)={\bf 1}[w_i\in(\tau_j,\tau)]\le {\bf 1}\big[w_i\in(\tau_j,\tau_a)\big].$ Similarly for any $\tau\in \cT(\tau_j,u_n)$ where $\tau<\tau_j$ we have $\z_i(\tau)\le {\bf 1}\big[w_i\in (\tau_{b},\tau_j)\big].$ Thus
	\benr\lel{l1e3}
	\hspace{1cm}\sup_{\tau\in \cB(\tau_j,u_n)} \frac{1}{n}\sti \z_i(\tau,\tau_j)\le \max_{\tau\in\{\tau_a,\tau_b\}}\frac{1}{n}\sti \z_i(\tau,\tau_j)\le c_u\Big\{\frac{\log p}{n},u_n\Big\}.
	\eenr
	with probability at least $1-c_1\exp(-c_2\log p).$ Combining the bound (\ref{l1e3}) over all $j=1,...,l$ using a union bound, we obtain (\ref{eq:ballbound}) with probability at least $1-c_1(2u_n)^{-1}\exp(-c_2\log p).$ Finally since by assumption $\log (u_n^{-1})= O(\log p)$ therefore, (\ref{eq:ballbound}) holds with probability at least  $1-c_1\exp(-c_2\log p),$ for $n$ sufficiently large. This completes the proof of Part (i).
	
	The proof of Part (ii) proceeds with a similar idea as Part (i). Divide the space $\bar\R$ into $l=2/v_n$ closed intervals (disjoint except at the boundaries), each of length $v_n/2.$ Let $\tau_1,...\tau_{l}$ be fixed points which represent the centres of these intervals. We shall show that,
	\benr\label{eq:ballboundlower}
	\min_{j=1,...,l}\inf_{\substack{\tau\in\bar\R;\\\tau\in\cB(\tau_j,v_n/4)}}\frac{1}{n}\sti \z_i(\tau,\tau_j)\ge c_uv_n,
	\eenr
	with probability at least $1-c_1\exp(-c_2\log p),$ for $n$ sufficiently large. Assuming (\ref{eq:ballboundlower}), observe that, at least one interval $\cB(\tau_j,v_n/2),$ $j=1,...,l$ will be contained in the interval between any two $\tau_a,\tau_b\in\bar\R$ satisfying $d(\tau_a,\tau_b)\ge v_n.$ This implies that
	\benr
	\inf_{\substack{\tau_a,\tau_b\in\bar\R;\\d(\tau_a,\tau_b)\ge v_n}}\frac{1}{n}\sti \z_i(\tau_a,\tau_b)\ge \min_{j=1,...,l}\inf_{\substack{\tau\in\bar\R;\\\tau\in\cB(\tau_j,v_n/4)}}\frac{1}{n}\sti \z_i(\tau,\tau_j)\ge c_uv_n\nn
	\eenr
	with probability at least $1-c_1\exp(-c_2\log p).$ Thus to prove part (i), it only remains to prove (\ref{eq:ballboundlower}). For this purpose, we use a lower bound for sums of non-negative r.v.s' stated in Lemma \r{lem:maurer}. This result was originally proved by \cite{maurer2003bound}. For a fixed right boundary point $\tau_a>\tau_j$ such that $d(\tau_a,\tau_j)=v_n/4,$ set $t=v_n$ in Lemma \r{lem:maurer}. Then we have
	\benr
	P\Big(\frac{1}{n}\sti \z_i(\tau_a,\tau_j)\le v_n\Big)\le \exp\Big(-4\frac{n^2v_n^2}{nv_n}\Big)\le c_1\exp(-c_2\log p),\nn
	\eenr
	where the last inequality follows from $v_n\ge c\log p/n.$  We obtain the same bound applying a similar argument for the left boundary point $\tau_b<\tau_{j}$ such that $d(\tau_{b},\tau_j)=v_n/4.$ Now applying an elementary union bound we obtain
	\benr\lel{l1e4}
	\hspace{0.5cm}P\Big(\min_{\tau\in\{\tau_a,\tau_b\}}\frac{1}{n}\sti \z_i(\tau,\tau_j)\ge c_u v_n\Big)\ge 1-c_1\exp(-c_2\log p).
	\eenr
	In order to obtain uniformity over $\tau\in \big\{\tau;\, d(\tau,\tau_j)\ge v_n/4\big\}$ note that for $\tau>\tau_j,$ we have $\z_i(\tau,\tau_j)={\bf 1}\big[w_i\in(\tau_j,\tau)\big]\ge {\bf 1}[w_i\in(\tau_j,\tau_a]]$ and for any $\tau<\tau_j,$ we have $\z_i(\tau,\tau_j)={\bf 1}[w_i\in(\tau_b,\tau_j)]\ge {\bf 1}[w_i\in[\tau_b,\tau_j)].$ This implies that
	\benr\lel{l1e5}
	\inf_{\substack{\tau\in\bar\R;\\ d(\tau,\tau_j)\ge v_n}}\frac{1}{n}\sti \z_i(\tau,\tau_j)\ge \min_{\tau\in\{\tau_a,\tau_b\}}\frac{1}{n}\sti \z_i(\tau,\tau_j)\ge c_uv_v.
	\eenr
	with probability at least $1-c_1\exp(-c_2\log p).$ Finally, (\ref{eq:ballboundlower}) follows by using a union bound over all $j=1,...,N$ and recalling that by assumption $v_n\ge c\log p/ n$ and therefore $\log (v_n^{-1})=O(\log p).$ This complete the proof of Lemma \ref{lem:indicatorbound}.
\end{proof}
%%%%%%%%%%%%%%%%%%%%%%%%%%%%%%%%%%%%%%%%%%%%%%%%%%%%%%%%%%%%%%%%%%%%%%%%%%%%%%%%%%%%%%%%%%%%%%%%%%%%%%%%%%%%%%%%%%%%%%%%%%%%%%%%%%%%%%%%%%%%%%%%%%%%%%%%%%%%%%%%%%%%%%%%%%%%

%%%%%%%%%%%%%%%%%%%%%%%%%%%%%%%%%%%%%%%%%%%%%%%%%%%%%%%%%%%%%%%%%%%%%%%%%%%%%%%%%%%%%%%%%%%%%%%%%%%%%%%%%%%%%%%%%%%%%%%%%%%%%%%%%%%%%%%%%%%%%%%%%%%%%%%%%%%%%%%%%%%%%%%%%%%%
\vspace{2mm}
\begin{proof}[Proof of Lemma \ref{lem:restrictedeigen}]
	In the following let $n^w:=n^w(\tau_a,\tau_b).$ To prove Part (i) note that,
	\benr\label{eq:equality}
	\hspace{1cm}\inf_{\substack{\tau_a,\tau_b\in\bar\R;\\d(\tau_a,\tau_b)\ge v_n}}\inf_{\delta\in \A} \frac{1}{n}\sum_{i\in n^w}\delta^T x_ix_i^T \delta = \inf_{\substack{\tau_a,\tau_b\in\bar\R;\\d(\tau_a,\tau_b)\ge v_n}}\frac{|n^w|}{n}\inf_{\delta\in\A}\frac{1}{|n^w|}\sum_{i\in n^w}\delta^T x_ix_i^T \delta
	\eenr
	Let $P_w(\cdot)$ represent the conditional probability $P(\cdot|w),$ where $w=(w_1,...,w_n)^T.$ Recalling that $w$ is independent of $x,\vep,$ by assumption D(iv) and applying Lemma \ref{lem:durrett} and Lemma \ref{lem:lw} we obtain,
	\benr\label{eq:conditionalprob1}
	P_w\Big(\inf_{\delta\in\A}\frac{1}{|n^w|}\sum_{i\in n^w} \delta^Tx_ix_i^T\delta\ge \ka\|\delta\|_2^2-c_uc_m\frac{\log p}{|n^w|}\|\delta\|_1^2\Big)\ge\hspace{1cm}\nn\\
	1-c_1\exp(-c_2\log p)
	\eenr
	Since the probability in the RHS of (\ref{eq:conditionalprob1}) is free of $w,$ taking expectations on both sides yields,
	\benr\label{eq:unconditionalprob1}
	P\Big(\inf_{\delta\in\A}\frac{1}{|n^w|}\sum_{i\in n^w} \delta^Tx_ix_i^T\delta\ge \ka\|\delta\|_2^2-c_uc_m\frac{\log p}{|n^w|}\|\delta\|_1^2\Big)\ge\hspace{1cm}\nn\\
	1-c_1\exp(-c_2\log p)
	\eenr
	Recall from Part (ii) of Lemma \ref{lem:indicatorbound} that $\inf_{d(\tau_a,\tau_b)\ge v_n}|n^w|/n\ge c_uv_n,$ with probability at least $1-c_1\exp(-c_2\log p).$  Also, since $\delta\in\A,$ hence $\|\delta\|_1^2\le c_us\|\delta\|_2^2.$ Combining these results with (\ref{eq:unconditionalprob1}) and substituting in (\ref{eq:equality}) we obtain,
	\benr
	\inf_{\substack{\tau_a,\tau_b\in\bar\R;\\d(\tau_a,\tau_b)\ge v_n}}\inf_{\delta\in \A} \frac{1}{n}\sum_{i\in n^w}\delta^T x_ix_i^T \delta &\ge& c_uc_mv_n\|\delta\|_2^2-c_uc_m\frac{s\log p}{n}\|\delta\|_2^2\nn\\
	&\ge& c_uc_mv_n\|\delta\|_2^2,\nn
	\eenr
	with probability at least $1-c_1\exp(-c_2\log p).$ Here the final inequality follows since by assumption $v_n\ge cs\log p\big/n.$ This completes the proof of Part (i). The proof of Part (ii) and Part (iii) are very similar to Part (i) and thus only key steps are provided. To prove part (ii), proceed as in Part (i) to obtain,
	\benr\label{eq:unconditionalprob2}
	P\Big(\inf_{\delta\in\A_2}\frac{1}{|n^w|}\sum_{i\in n^w} \delta^Tx_ix_i^T\delta\ge \ka\|\delta\|_2^2-c_uc_m\frac{\log p}{|n^w|}\|\delta\|_1^2\Big)\ge\hspace{1cm}\nn\\
	1-c_1\exp(-c_2\log p)
	\eenr
	In this case since $\delta\in\A_2,$ hence $\|\delta\|_1^2\le c_us(\|\delta\|_2^2+\xi_{\max}^2).$ Substituting this result in (\ref{eq:unconditionalprob2}) and proceeding as in Part (i) yields,
	\benr
	\inf_{\substack{\tau_a,\tau_b\in\bar\R;\\d(\tau_a,\tau_b)\ge v_n}}\inf_{\delta\in \A_2} \frac{1}{n}\sum_{i\in n^w}\delta^T x_ix_i^T \delta &\ge& c_uc_mv_n\|\delta\|_2^2-c_uc_m\frac{s\log p}{n}\|\delta\|_2^2-\frac{\xi_{\max}^2s\log p}{n}\nn\\
	&\ge& c_uc_mv_n\|\delta\|_2^2-\frac{\xi_{\max}^2s\log p}{n},\nn
	\eenr
	with probability at least $1-c_1\exp(-c_2\log p).$ This completes the proof of Part (ii). To prove Part (iii), note that
	\benr\lel{eq:equalitysup}
	\sup_{\substack{\tau_a,\tau_b\in\bar\R;\\ d(\tau_a,\tau_b)\le u_n}}\sup_{\delta\in\A}\frac{1}{n}\sum_{i\in n_w} \delta^Tx_ix_i^T\delta=\sup_{\substack{\tau_a,\tau_b\in\bar\R;\\ d(\tau_a,\tau_b)\le u_n}}\frac{|n_w|}{n}\sup_{\delta\in\A}\frac{1}{|n_w|}\sum_{i\in n_w} \delta^Tx_ix_i^T\delta\nn
	\eenr
	Now, from Part (i) of Lemma \ref{lem:indicatorbound} we have that $\sup_{d(\tau_a,\tau_b)\le u_n}|n^w|/n\le c_u\max\{\log p/n ,u_n\}.$ Proceeding via the conditional probability argument described for Part (i) leads to,
	\benr
	\sup_{\substack{\tau_a,\tau_b\in\bar\R;\\ d(\tau_a,\tau_b)\le u_n}}\sup_{\delta\in\A} \frac{1}{n}\sum_{i\in n^w}\delta^T x_ix_i^T \delta &\le& c_uc_m\|\delta\|_2^2 \max\Big\{\frac{\log p}{n}, u_n\Big\} + c_uc_m\frac{s\log p}{n}\|\delta\|_2^2\nn\\
	&\le& c_uc_m\|\delta\|_2^2 \max\Big\{\frac{s\log p}{n},u_n\Big\}\nn
	\eenr
	with probability at least $1-c_1\exp(-c_2\log p).$ This completes the proof of Lemma \ref{lem:restrictedeigen}.
\end{proof}
%%%%%%%%%%%%%%%%%%%%%%%%%%%%%%%%%%%%%%%%%%%%%%%%%%%%%%%%%%%%%%%%%%%%%%%%%%%%%%%%%%%%%%%%%%%%%%%%%%%%%%%%%%%%%%%%%%%%%%%%%%%%%%%%%%%%%%%%%%%%%%%%%%%%%%%%%%%%%%%%%%%%%%%%%%%%

%%%%%%%%%%%%%%%%%%%%%%%%%%%%%%%%%%%%%%%%%%%%%%%%%%%%%%%%%%%%%%%%%%%%%%%%%%%%%%%%%%%%%%%%%%%%%%%%%%%%%%%%%%%%%%%%%%%%%%%%%%%%%%%%%%%%%%%%%%%%%%%%%%%%%%%%%%%%%%%%%%%%%%%%%%%%
\vspace{2mm}
\begin{proof}[Proof of Theorem \ref{thm:initialrate}]
	First consider the case where, $\tau_a,\tau_b\in \cC_j^1.$ Then a simple algebraic manipulation of the basic inequality $\frac{1}{n}Q^*(\h\al,\tau_a,\tau_b)+\la\|\h\al\|_1\le \frac{1}{n}Q^*(\b^0_{(j-1)},\tau_a,\tau_b)+\la\|\b^0_{(j-1)}\|_1$ yields,
	\benr\label{eq:basicinq}
	\frac{1}{n}\sum_{i\in n^w(\tau_a,\tau_b)} \|x_i^T(\h\al-\b^{0}_{(j-1)})\|_2^2 + \la_0\|\al\|_1\le\hspace{0.5in}\nn\\
	\Big|\frac{2}{n}\sum_{i\in n^w(\tau_a,\tau_b)} \tilde\vep_ix_i^T(\h\al-\b^{0}_{(j-1)})\Big|+ \la_0\|\b^0_{(j)}\|_1.
	\eenr
	Here $\tilde\vep_i=y_i-x_i^T\b^0_{(j-1)}.$ Note that $\tilde\vep_i$ may or may not be the same as $\vep_i$ depending on the index $i.$  Also, we have the following bound,
	\benr
	\frac{1}{n}\Big\|\sum_{i\in n^w(\tau_a,\tau_b)} \tilde\vep_ix_i^T\Big\|_{\iny}\hspace{4in}\nn\\
	\le \frac{1}{n}\Big\|\sum_{i\in n^w(\tau_{j-1}^0,\tau_j^0)} \vep_ix_i^T\Big\|_{\iny} +\frac{1}{n}\Big\|\sum_{i\in n^w(\tau_a,\tau_{j-1}^0)} \tilde\vep_ix_i^T\Big\|_{\iny}+\frac{1}{n}\Big\|\sum_{i\in n^w(\tau_j^0,\tau_b)} \tilde\vep_ix_i^T\Big\|_{\iny}\hspace{1cm}\nn\\
	\le \frac{1}{n}\Big\|\sum_{i\in n^w(\tau_{j-1}^0,\tau_j^0)} \vep_ix_i^T\Big\|_{\iny}+\frac{1}{n}\Big\|\sum_{i\in n^w(\tau_{j-1}^0,\tau_j^0)} \vep_ix_i^T\Big\|_{\iny}
	+\frac{1}{n}\Big\|\sum_{i\in n^w(\tau_{j-1}^0,\tau_j^0)} \vep_ix_i^T\Big\|_{\iny}\hspace{0.75cm}\nn\\
	+\frac{1}{n}\Big\|\sum_{i\in n^w(\tau_a,\tau_{j-1}^0)} (\b_{(j-1)}^0-\b_{(j-2)}^0)^Tx_ix_i^T \Big\|_{\iny}+\frac{1}{n}\Big\|\sum_{i\in n^w(\tau_j^0,\tau_b)} (\b_{(j)}^0-\b_{(j-1)}^0)^Tx_ix_i^T \Big\|_{\iny}\nn\\
	\le c_uc_m\sqrt{\frac{\log p}{n}} + c_uc_m\sqrt{\frac{\log p}{n}}\max\Big\{\sqrt{\frac{\log p}{n}},\sqrt{u_n}\Big\}+c_uc_m\max\Big\{\frac{\xi_{\max}\log p}{n},\xi_{\max}u_n\Big\}\nn\\
	\le c_uc_m\max\Big\{\sqrt{\frac{\log p}{n}},\xi_{\max}u_n\Big\}=\la.\hspace{3.1in}\nn
	\eenr
	The second to last inequality here follows by applying the bounds provided in Lemma \ref{lem:crossbounds}. Substituting the bound of the final inequality in (\ref{eq:basicinq}), and choosing $\la_0=2\la,$ yields the relation $\|\h\al_{S^c}\|_1\le 3\|(\h\al-\b^0_{(j-1)})_S\|_1,$ consequently the vector $\h\al-\b^0_{(j-1)}\in\A.$ Thus the first two inequalities of Lemma \ref{lem:restrictedeigen} are now applicable. From (\ref{eq:basicinq}) and an application of Part (i) Lemma \ref{lem:restrictedeigen} with $v_n=l_{\min}$ we can obtain,
	\benr
	c_uc_ml_{\min}\|\h\al-\b^0_{j-1}\|_2^2\le \sqrt{s} \la \|\h\al-\b^0_{j-1}\|_2,\nn
	\eenr
	which directly implies that $\|\h\al-\b^0_{j-1}\|_2\le \sqrt{s}\la\big/l_{\min}.$ To obtain the $\ell_1$ bound, recall that since $\h\al-\b^0_{(j-1)}\in\A,$ hence $\|\h\al-\b^0_{(j-1)}\|_1\le \sqrt{s}\|\h\al-\b^0_{(j-1)}\|_2.$ To complete the proof of this case, note that all bounds in the above arguments hold uniformly over any $\tau_a,\tau_b\in\cC^1_j,$ with probability at least $1-c_1\exp(-c_2\log p).$ The cases of $\tau_a,\tau_b\in\cC_j^2,$ $\cC^3_j$ and $\cC^4_j$ can be proved similarly. The final statement of the lemma follows by applying a union bound.
\end{proof}
%%%%%%%%%%%%%%%%%%%%%%%%%%%%%%%%%%%%%%%%%%%%%%%%%%%%%%%%%%%%%%%%%%%%%%%%%%%%%%%%%%%%%%%%%%%%%%%%%%%%%%%%%%%%%%%%%%%%%%%%%%%%%%%%%%%%%%%%%%%%%%%%%%%%%%%%%%%%%%%%%%%%%%%%%%%%

%%%%%%%%%%%%%%%%%%%%%%%%%%%%%%%%%%%%%%%%%%%%%%%%%%%%%%%%%%%%%%%%%%%%%%%%%%%%%%%%%%%%%%%%%%%%%%%%%%%%%%%%%%%%%%%%%%%%%%%%%%%%%%%%%%%%%%%%%%%%%%%%%%%%%%%%%%%%%%%%%%%%%%%%%%%%
\vspace{2mm}
\begin{proof}[Proof of Corollary \ref{cor:intialrate}]
	The proof of this result is a direct consequence of Theorem \ref{thm:initialrate}. Observe that by Condition A(i) and A(ii), the initial change point vector $\check\tau=(\check\tau_1,\check\tau_2,...,\check\tau_{\check N})^T$ satisfies the following. For any $j=1,...,N$ the pair $\tau_{m_{j}-1},\tau_{m_j}$ lies in either $\cC^1_{j},$ or $\cC^2_{j}$ as defined in Theorem \ref{thm:initialrate}. Part (i) of this corollary follows by applying Theorem \ref{thm:initialrate}. Similarly, Part (ii) follows by noting that for any $j\in \cT^c,$ the pair $\check\tau_{j-1},\check \tau_j$ belongs to either $\cC_{k_j}^3$ or $\cC_{k_j}^4.$ This completes the proof of this corollary.
\end{proof}
%%%%%%%%%%%%%%%%%%%%%%%%%%%%%%%%%%%%%%%%%%%%%%%%%%%%%%%%%%%%%%%%%%%%%%%%%%%%%%%%%%%%%%%%%%%%%%%%%%%%%%%%%%%%%%%%%%%%%%%%%%%%%%%%%%%%%%%%%%%%%%%%%%%%%%%%%%%%%%%%%%%%%%%%%%%%

%%%%%%%%%%%%%%%%%%%%%%%%%%%%%%%%%%%%%%%%%%%%%%%%%%%%%%%%%%%%%%%%%%%%%%%%%%%%%%%%%%%%%%%%%%%%%%%%%%%%%%%%%%%%%%%%%%%%%%%%%%%%%%%%%%%%%%%%%%%%%%%%%%%%%%%%%%%%%%%%%%%%%%%%%%%%
\begin{rem}
	(Additional notation used in the Proof of Lemma \ref{lem:ulowerbound}):  Recall that the set $\cT=\{m_0,m_1,...,m_{N},m_{N+1}\}$ (defined in Condition A) is the subset of indices of $\{0,1,2....,\check N+1\},$ such that the initial change point $\check\tau_{m_j}$ lies in a ${\check u}$-neighborhood of $\tau^0_j.$ In the proof to follow, we use the notation $\sum_{m_{j-1}<l<m_j},$ to represent the sum over all possible indices $l$ which lie between $m_{j-1}$ and $m_j.$ For example, let $N=2,$ $\check N=5$ and consider any $\tau\in\bar\R^{\check N}$ in the orientation described in the following Figure \ref{fig:appendixorientation}.
	
	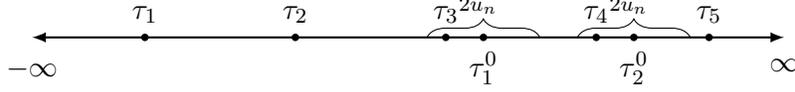
\begin{figure}[H]
		\centering{
			\begin{tikzpicture}
			\draw[black,thick,latex-latex] (0,0) -- (10,0)
			node[pos=0,label=below:\textcolor{black}{$-\iny$}]{}
			node[pos=0.6,mynode,fill=black,label=below:\textcolor{black}{$\tau_1^0$}]{}
			node[pos=0.8,mynode,fill=black,label=below:\textcolor{black}{$\tau_2^0$}]{}
			node[pos=0.15,mynode,fill=black,label=above:\textcolor{black}{$\tau_1$}]{}
			node[pos=0.35,mynode,fill=black,label=above:\textcolor{black}{$\tau_2$}]{}
			node[pos=0.55,mynode,fill=black,label=above:\textcolor{black}{$\tau_3$}]{}
			node[pos=0.75,mynode,fill=black,label=above:\textcolor{black}{$\tau_4$}]{}
			node[pos=0.9,mynode,fill=black,label=above:\textcolor{black}{$\tau_5$}]{}
			node[pos=1,label=below:\textcolor{black}{$\iny$}]{};
			\draw[decorate,decoration={brace,amplitude=7pt}] (5.25,0) --node[above left=5pt and -9pt]{$\scriptstyle{2u_n}$} (6.75,0);
			\draw[decorate,decoration={brace,amplitude=7pt}] (7.25,0) --node[above left=5pt and -9pt]{$\scriptstyle{2u_n}$} (8.75,0);
			\end{tikzpicture}
			\caption{\footnotesize{A possible orientation of initializers $\check\tau\in\R^{5},$ where $N=2.$}}
			\label{fig:appendixorientation}}
	\end{figure}
	Then, $\cT=\{m_0,m_1,m_2,m_3\}=\{0,3,4,6\},$ and for $j=1,$ we denote by,
	\benr
	\sum_{m_{j-1}<l<m_j}Q^*(\al_{(l-1)},\tau_{l-1},\tau_l)= Q^*(\al_{0},\tau_{0},\tau_1)+Q^*(\al_{1},\tau_{1},\tau_2)\nn
	\eenr
\end{rem}
%%%%%%%%%%%%%%%%%%%%%%%%%%%%%%%%%%%%%%%%%%%%%%%%%%%%%%%%%%%%%%%%%%%%%%%%%%%%%%%%%%%%%%%%%%%%%%%%%%%%%%%%%%%%%%%%%%%%%%%%%%%%%%%%%%%%%%%%%%%%%%%%%%%%%%%%%%%%%%%%%%%%%%%%%%%%

%%%%%%%%%%%%%%%%%%%%%%%%%%%%%%%%%%%%%%%%%%%%%%%%%%%%%%%%%%%%%%%%%%%%%%%%%%%%%%%%%%%%%%%%%%%%%%%%%%%%%%%%%%%%%%%%%%%%%%%%%%%%%%%%%%%%%%%%%%%%%%%%%%%%%%%%%%%%%%%%%%%%%%%%%%%%
\begin{rem} (Useful observation utilized in the Proof of Lemma \ref{lem:ulowerbound}): Consider the following decomposition of the $\ell_0$ regularizing term in {\bf Step 1} of Algorithm 1. For any $\tau\in\R^{\check N}$ such that, $\tau\in\cG(u_n,v_n,\cK,\cK_2),$ and $\tau^*$ as defined in (\ref{def:taustar}), we have that,
	\benr\label{eq:l0decomp}
	\sum_{j=1}^{\check N}\Big(\|d(\tau_{j-1},\tau_j)\|_0-\|d(\tau_{j-1}^*,\tau_j^*)\|_0\Big)=\sum_{j\in\cT^c}\|d(\tau_{j-1},\tau_j)\|_0\nn\\
	+\sum_{j=1}^{N}\Big(\|d(\tau_{h_j-1},\tau_{h_j})\|_0-\|d(\tau_{h_j-1}^*,\tau_{h_j}^*)\|_0\Big)\nn
	\eenr
\end{rem}
%%%%%%%%%%%%%%%%%%%%%%%%%%%%%%%%%%%%%%%%%%%%%%%%%%%%%%%%%%%%%%%%%%%%%%%%%%%%%%%%%%%%%%%%%%%%%%%%%%%%%%%%%%%%%%%%%%%%%%%%%%%%%%%%%%%%%%%%%%%%%%%%%%%%%%%%%%%%%%%%%%%%%%%%%%%%

%%%%%%%%%%%%%%%%%%%%%%%%%%%%%%%%%%%%%%%%%%%%%%%%%%%%%%%%%%%%%%%%%%%%%%%%%%%%%%%%%%%%%%%%%%%%%%%%%%%%%%%%%%%%%%%%%%%%%%%%%%%%%%%%%%%%%%%%%%%%%%%%%%%%%%%%%%%%%%%%%%%%%%%%%%%%
\vspace{2mm}
\begin{proof}[Proof of Lemma \ref{lem:ulowerbound}] Consider any $\tau\in\cG(u_n,v_n,\cK).$ The proof to follow relies in part on an algebraic manipulation of $\cU^*(\check N,\h\al,\tau)$ defined in (\ref{def:calliustar}), which in turn requires a decomposition of the least squares loss $Q(\check N, \al,\tau).$ This decomposition of the least squares loss depends on the orientation of $\tau,$ and in the following we assume a specific orientation of $\tau,$ such that $\tau_{h_j-1}\le\tau^0_{j}\le \tau_{h_j}\le\tau_{j+1},$ $j=1,...,N.$ While assuming this orientation does lead to a loss in generality in the sense that it does not include all possible $\tau,$ however it can be observed that any orientation of $\tau$ shall lead to the same lower bound, this can be verified by following the same argument as below, however with a correspondingly different decomposition of the least squares loss.  We also refer to Lemma 4.1 of \cite{kaul2018parameter}, which provides a similar result in the special case with $\check N=1,$ for further intuition as to how the same bound persists under any other orientation. In the following, for any $\al_{(j)}\in\R^{p},$ $j=0,...,\check N,$ let $\al$ represent the concatenation of $\al_{(j)}'$s. Then consider,
	\benr\label{eq:Qtau}
	Q(\check N,\al,\tau)&=&\frac{1}{n}\sum_{j=1}^{\check N+1}Q^*(\al_{(j-1)},\tau_{j-1},\tau_j)=\frac{1}{n}\sum_{j=1}^{N+1}\sum_{h_{j-1}<l\le h_j}Q^*(\al_{(l-1)},\tau_{l-1},\tau_l)\nn\\
	&=&\frac{1}{n}\sum_{j=1}^{N+1}\sum_{h_{j-1}<l< h_j}Q^*(\al_{(l-1)},\tau_{l-1},\tau_l)+\frac{1}{n}\sum_{j=1}^{N+1}Q^*(\al_{(h_j-1)},\tau_{h_j-1},\tau_{h_j})\nn\\
	&=&\frac{1}{n}\sum_{j=1}^{N+1}\sum_{m_{j-1}<l< h_j}Q^*(\al_{(l-1)},\tau_{l-1},\tau_l)+\frac{1}{n}\sum_{j=1}^{N+1}Q^*(\al_{(h_j-1)},\tau_{h_j-1},\tau_{j}^0)\nn\\
	&&+\frac{1}{n}\sum_{j=1}^{N}Q^*(\al_{(h_j-1)},\tau_{j}^0,\tau_{h_j}).\nn
	\eenr
	Now, recall the definition of $\tau^*$ from (\ref{def:taustar}) and note that,
	\benr\label{eq:Qtaustar}
	Q(\check N,\al,\tau^*)&=&\frac{1}{n}\sum_{j=1}^{\check N+1}Q^*(\al_{(j-1)},\tau_{j-1}^*,\tau_j^*)=\frac{1}{n}\sum_{j=1}^{N+1}Q^*(\al_{(h_j-1)},\tau_{h_j-1}^*,\tau_{h_j}^*)\nn\\
	&=&\frac{1}{n}\sum_{j=1}^{N+1}\sum_{m_{j-1}<l< h_j}Q^*(\al_{(h_j-1)},\tau_{l-1},\tau_{l})+\frac{1}{n}\sum_{j=1}^NQ(\al_{(h_{j+1}-1)},\tau_{j}^0,\tau_{h_j})\nn\\
	&&+\frac{1}{n}\sum_{j=1}^{N+1}Q^*(\al_{(h_j-1)},\tau_{h_j-1},\tau_{j}^0).\nn
	\eenr
	Substituting the above expressions for $Q(\check N,\al,\tau)$ and $Q(\check N,\al,\tau^*)$ in the definition of $\cU^*(\check N, \al,\tau)$ given in (\ref{def:calliustar}), we obtain,
	\benr
	\cU^*(\check N, \al,\tau)&=&Q(\check N, \al,\tau)-Q(\check N, \al,\tau^*)\nn\\
	&=& \frac{1}{n}\sum_{j=1}^{N+1}\sum_{m_{j-1}<l< h_j}Q^*(\al_{(l-1)},\tau_{l-1},\tau_l)+\frac{1}{n}\sum_{j=1}^{N}Q^*(\al_{(h_j-1)},\tau_{j}^0,\tau_{h_j})\nn\\
	&&-\frac{1}{n}\sum_{j=1}^NQ(\al_{(h_{j+1}-1)},\tau_{j}^0,\tau_{h_j})-\frac{1}{n}\sum_{j=1}^{N+1}\sum_{m_{j-1}<l< h_j}Q^*(\al_{(h_j-1)},\tau_{l-1},\tau_{l})\nn\\
	&:=&(T1)+(T2)-(T3)-(T4)\nn
	\eenr
	Further simplifying terms (T1)-(T4) we obtain,
	\benr
	T1&=&\frac{1}{n}\sum_{j=1}^{N+1}\sum_{m_{j-1}<l< h_j}Q^*(\al_{(l-1)},\tau_{l-1},\tau_l)=\frac{1}{n}\sum_{j=1}^{N+1}\sum_{m_{j-1}<l< h_j}\sum_{\tau_{l-1}<w_i<\tau_{l}}\big(y_i-x_i^T\al_{(l-1)}\big)^2\nn\\
	&=& \frac{1}{n}\sum_{j=1}^{N+1}\sum_{m_{j-1}<l< h_j}\sum_{\tau_{l-1}<w_i<\tau_{l}} \vep_i^2-\frac{2}{n}\sum_{j=1}^{N+1}\sum_{m_{j-1}<l< h_j}\sum_{\tau_{l-1}<w_i<\tau_{l}}\vep_ix_i^T(\al_{(l-1)}-\b^0_{(j-1)}) \nn\\
	&&+\frac{1}{n}\sum_{j=1}^{N+1}\sum_{m_{j-1}<l< h_j}\sum_{\tau_{l-1}<w_i<\tau_{l}}(\al_{(l-1)}-\b^0_{(j-1)})^Tx_ix_i^T(\al_{(l-1)}-\b^0_{(j-1)})\nn
	\eenr
	
	\benr
	T2&=&\frac{1}{n}\sum_{j=1}^{N}Q^*(\al_{(h_j-1)},\tau_{j}^0,\tau_{h_j})=\frac{1}{n}\sum_{j=1}^{N}\sum_{i\in n^w(\tau^0_j,\tau_{h_j})}\big(y_i-x_i^T\al_{(h_j-1)}\big)^2\nn\\
	&=&\frac{1}{n}\sum_{j=1}^{N}\sum_{i\in n^w(\tau^0_j,\tau_{h_j})}\big(\vep_i-x_i^T(\al_{(h_j-1)}-\b^0_{(j)})\big)^2\nn\\
	&=& \frac{1}{n}\sum_{j=1}^{N}\sum_{i\in n^w(\tau^0_j,\tau_{h_j})} \vep_i^2 - \frac{2}{n}\sum_{j=1}^{N}\sum_{i\in n^w(\tau^0_j,\tau_{h_j})}\vep_ix_i^T(\al_{(h_j-1)}-\b^0_{(j)})\nn\\
	&&+\frac{1}{n}\sum_{j=1}^{N}\sum_{i\in n^w(\tau^0_j,\tau_{h_j})}(\al_{(h_j-1)}-\b^0_{(j)})^Tx_ix_i^T(\al_{(h_j-1)}-\b^0_{(j)})\nn
	\eenr
	
	\benr
	T3&=&\frac{1}{n}\sum_{j=1}^NQ(\al_{(h_{j+1}-1)},\tau_{j}^0,\tau_{h_j})\nn\\
	&=&\frac{1}{n}\sum_{j=1}^{N}\sum_{i\in n^w(\tau^0_j,\tau_{h_j})} \vep_i^2 - \frac{2}{n}\sum_{j=1}^{N}\sum_{i\in n^w(\tau^0_j,\tau_{h_j})}\vep_ix_i^T(\al_{(h_{j+1}-1)}-\b^0_{(j)})\nn\\
	&&+\frac{1}{n}\sum_{j=1}^{N}\sum_{i\in n^w(\tau^0_j,\tau_{h_j})}(\al_{(h_{j+1}-1)}-\b^0_{(j)})^Tx_ix_i^T(\al_{(h_{j+1}-1)}-\b^0_{(j)})\nn
	\eenr
	
	\benr
	T4&=&\frac{1}{n}\sum_{j=1}^{N+1}\sum_{m_{j-1}<l< h_j}Q^*(\al_{(h_j-1)},\tau_{l-1},\tau_{l})\nn\\
	&=&\frac{1}{n}\sum_{j=1}^{N+1}\sum_{m_{j-1}<l< h_j}\sum_{\tau_{l-1}<w_i<\tau_{l}}\big(y_i-x_i^T\al_{(h_j-1)}\big)^2\nn\\
	&=&\frac{1}{n}\sum_{j=1}^{N+1}\sum_{m_{j-1}<l< h_j}\sum_{\tau_{l-1}<w_i<\tau_{l}} \vep_i^2\nn\\
	&&-\frac{2}{n}\sum_{j=1}^{N+1}\sum_{m_{j-1}<l< h_j}\sum_{\tau_{l-1}<w_i<\tau_{l}}\vep_ix_i^T(\al_{(h_j-1)}-\b^0_{(j-1)}) \nn\\
	&&+\frac{1}{n}\sum_{j=1}^{N+1}\sum_{m_{j-1}<l< h_j}\sum_{\tau_{l-1}<w_i<\tau_{l}}(\al_{(h_j-1)}-\b^0_{(j-1)})^Tx_ix_i^T(\al_{(h_j-1)}-\b^0_{(j-1)})\nn
	\eenr
	Substituting the above expressions for terms $(T1)-(T4)$ back in the expression for $\cU^*(\check N,\al,\tau),$ while also noting that all terms involving $\vep_i^2$ cancel each other, we obtain,
	
	\benr
	\cU^*(\check N,\al,\tau)&=&\frac{1}{n}\sum_{j=1}^{N}\sum_{i\in n^w(\tau^0_j,\tau_{h_j})}(\al_{(h_j-1)}-\b^0_{(j)})^Tx_ix_i^T(\al_{(h_j-1)}-\b^0_{(j)})\nn\\
	&&+\frac{1}{n}\sum_{j=1}^{N+1}\sum_{m_{j-1}<l< h_j}\sum_{\tau_{l-1}<w_i<\tau_{l}}(\al_{(l-1)}-\b^0_{j-1})^Tx_ix_i^T(\al_{(l-1)}-\b^0_{(j-1)})\nn\\
	&&-\frac{1}{n}\sum_{j=1}^{N}\sum_{i\in n^w(\tau^0_j,\tau_{h_j})}(\al_{(h_{j+1}-1)}-\b^0_{(j)})^Tx_ix_i^T(\al_{(h_{j+1}-1)}-\b^0_{(j)})\nn\\
	&&-\frac{1}{n}\sum_{j=1}^{N+1}\sum_{m_{j-1}<l< h_j}\sum_{\tau_{l-1}<w_i<\tau_{l}}(\al_{(h_j-1)}-\b^0_{(j-1)})^Tx_ix_i^T(\al_{(h_j-1)}-\b^0_{(j-1)})\nn\\
	&&-\frac{2}{n}\sum_{j=1}^{N+1}\sum_{m_{j-1}<l< h_j}\sum_{\tau_{l-1}<w_i<\tau_{l}}\vep_ix_i^T(\al_{(l-1)}-\b^0_{(j-1)})\nn\\
	&&-\frac{2}{n}\sum_{j=1}^{N}\sum_{i\in n^w(\tau^0_j,\tau_{h_j})}\vep_ix_i^T(\al_{(h_j-1)}-\b^0_{(j)})\nn\\
	&&+\frac{2}{n}\sum_{j=1}^{N}\sum_{i\in n^w(\tau^0_j,\tau_{h_j})}\vep_ix_i^T(\al_{(h_{j+1}-1)}-\b^0_{(j)})\nn\\
	&&+\frac{2}{n}\sum_{j=1}^{N+1}\sum_{m_{j-1}<l< h_j}\sum_{\tau_{l-1}<w_i<\tau_{l}}\vep_ix_i^T(\al_{(h_j-1)}-\b^0_{(j-1)}) \nn \\
	&:=& (R1)+(R2)-(R3)-(R4)-(R5)-(R6)+(R7)+(R8)\nn
	\eenr
	Here, the terms $(R1),(R3),(R6),(R7)$ are non-zero only when $N\ge 1.$ In the case where $N=0,$ these four terms will be identically zero. Also note that $R2\ge 0,$ since it is a quadratic form. Observe that when $\cU^*(\check N,\al,\tau)$ is evaluated at $\h\al,$  and at any $\tau\in\cG(u_n,v_n,\cK),$ the following uniform bounds for the terms $(R1)-(R8)$ hold,  each with probability at least $1-c_1N\exp(-c_2\log p).$ These bounds for terms $(R1)-(R8)$ follow from applications of Lemma \ref{lem:restrictedeigen}, Lemma \ref{lem:crossbounds} and Corollary \ref{cor:intialrate}. Details pertaining to the derivations of these bounds are discussed in detail in Lemma \ref{lem:ulowersupportargument} in Appendix B of the supplementary materials.	
	
	\benr
	R4&=&\frac{1}{n}\sum_{j=1}^{N+1}\sum_{m_{j-1}<l< h_j}\sum_{\tau_{l-1}<w_i<\tau_{l}}(\al_{(h_j-1)}-\b^0_{(j-1)})^Tx_ix_i^T(\al_{(h_j-1)}-\b^0_{(j-1)})\nn\\
	&\le&  \sum_{j=1}^{N+1}\sum_{m_{j-1}<l< h_j}r_n^2 \le c_uc_m|\cK|r_n^2\nn\\
	|R5|&\le& \frac{2}{n}\sum_{j=1}^{N+1}\sum_{m_{j-1}<l< h_j}\Big|\sum_{\tau_{l-1}<w_i<\tau_{l}}\vep_ix_i^T(\al_{(l-1)}-\b^0_{(j-1)})\Big|\nn\\
	&\le&  2\frac{2}{n}\sum_{j=1}^{N+1}\sum_{m_{j-1}<l< h_j}\sqrt{\frac{\log p}{n}}\sqrt{s}r_n \le c_uc_m|\cK|\sqrt{\frac{s\log p}{n}}r_n\nn\\
	|R8|&\le&\frac{2}{n}\sum_{j=1}^{N+1}\sum_{m_{j-1}<l< h_j}\Big|\sum_{\tau_{l-1}<w_i<\tau_{l}}\vep_ix_i^T(\al_{(h_j-1)}-\b^0_{(j-1)})\Big|\nn\\
	&\le& 2\sum_{j=1}^{N+1}\sum_{m_{j-1}<l< m_j}\sqrt{\frac{\log p}{n}}\sqrt{s}r_n \le c_uc_m|\cK|\sqrt{\frac{s\log p}{n}}r_n\nn
	\eenr
	
	Next, consider the following two subcases. In the first subcase, assume $N=0,$ In this subcase $R1=R3=R6=R7=0.$ Thus, combining the bounds for $(R1)-(R8),$ we obtain for this subcase,
	\benr
	\inf_{\tau\in\cG}\cU^*(\check N,\h\al,\tau)\ge -c_uc_m|\cK|r_n^2- c_uc_m|\cK|\sqrt{\frac{s\log p}{n}}r_n\nn
	\eenr
	In the second subcase, where $N\ge 1,$ we have,
	\benr
	R1&=&\frac{1}{n}\sum_{j=1}^{N}\sum_{i\in n^w(\tau^0_j,\tau_{h_j})}(\h\al_{(h_j-1)}-\b^0_{(j)})^Tx_ix_i^T(\h\al_{(h_j-1)}-\b^0_{(j)})\nn\\
	&\ge&  c_uc_m\xi_{\min}^2v_{n}-c_uc_m\frac{\xi_{\max}^2s\log p}{n}\nn\\
	R3&=&\frac{1}{n}\sum_{j=1}^{N}\sum_{i\in n^w(\tau^0_j,\tau_{h_j})}(\h\al_{(h_{j+1}-1)}-\b^0_{(j)})^Tx_ix_i^T(\h\al_{(h_{j+1}-1)}-\b^0_{(j)})\nn\\
	&\le&  c_uc_mr_n^2\sum_{j=1}^{N}\max\Big\{\frac{s\log p}{n},u_{nj}\Big\}\nn
	\eenr
	
	\benr
	|R6|&\le&\frac{2}{n}\sum_{j=1}^{N}\Big|\sum_{i\in n^w(\tau^0_j,\tau_{h_j})}\vep_ix_i^T(\h\al_{(h_j-1)}-\b^0_{(j)})\Big|\nn\\
	&\le& c_uc_m\xi_{\max}\sqrt{\frac{\log p}{n}}\sum_{j=1}^{N}\max\Big\{\sqrt{\frac{\log p}{n}},\sqrt{u_{nj}}\Big\}\nn\\
	|R7|&\le&\frac{2}{n}\sum_{j=1}^{N}\Big|\sum_{i\in n^w(\tau^0_j,\tau_{h_j})}\vep_ix_i^T(\h\al_{(h_{j+1}-1)}-\b^0_{(j)})\Big|\nn\\
	&\le&c_uc_m r_n\sqrt{\frac{\log p}{n}}\sum_{j=1}^{N}\max\Big\{\sqrt{\frac{\log p}{n}},\sqrt{u_{nj}}\Big\}.\nn
	\eenr
	Combining the bounds for the terms $(R1)-(R8),$ we obtain,
	\benr\label{eq:finalb}
	\inf_{\tau\in\cG}\cU^*(\check N,\h\al,\tau)\ge c_uc_m\xi_{\min}^2v_{n}-c_uc_m\frac{\xi_{\max}^2s\log p}{n}-c_uc_mr_n^2\sum_{j=1}^{N}\max\Big\{\frac{s\log p}{n},u_{nj}\Big\}\nn\\
	-c_uc_m\xi_{\max}\sqrt{\frac{\log p}{n}}\sum_{j=1}^{N}\max\Big\{\sqrt{\frac{\log p}{n}},\sqrt{u_{nj}}\Big\} -c_uc_m|\cK|r_n^2\nn\\
	- c_uc_m|\cK|\sqrt{\frac{s\log p}{n}}r_n\hspace{2.05in}
	\eenr
	To complete the proof, recall the definition of $\cU(\check N,\h\al,\tau)$ from (\ref{def:calliustar}), and observe from (\ref{eq:l0decomp}), that for $\tau\in\cG$
	\benr
	\cU(\check N,\h\al,\tau)&=&\cU^*(\check N,\h\al,\tau)+\mu|\cK|\nn\\
	&&+ \mu\sum_{j=1}^N\Big(\|d(\tau_{h_j-1},\tau_{h_j})\|_0-\|d(\tau_{h_j-1}^*,\tau_{h_j}^*)\|_0\Big)\nn
	\eenr
	Now, if $u_n$ is such that it converges to zero faster than $l_{\min},$ i.e. $u_n/l_{\min}\to 0,$ then clearly the sign of $d(\tau_{h_j-1},\tau_{h_j})$ will be the same as that of 	$d(\tau_{h_j-1}^*,\tau_{h_j}^*),$ for each $j\in\{1,...,N\},$ and $n$ sufficiently large. The statement of this lemma now follows by combining the above expression with (\ref{eq:finalb}) and using the assumption $\xi_{\min}>c_u.$
\end{proof}

%%%%%%%%%%%%%%%%%%%%%%%%%%%%%%%%%%%%%%%%%%%%%%%%%%%%%%%%%%%%%%%%%%%%%%%%%%%%%%%%%%%%%%%%%%%%%%%%%%%%%%%%%%%%%%%%%%%%%%%%%%%%%%%%%%%%%%%%%%%%%%%%%%%%%%%%%%%%%%%%%%%%%%%%%%%%

\vspace{2mm}
%%%%%%%%%%%%%%%%%%%%%%%%PROOF OF MAIN THEOREM%%%%%%%%%%%%%%%%%%%%%%%%%%%%%%%%%%%%%%%%%%%%%%%%%%%%%%%%%%%%%%%%%%%%%%%%%%%%%%%%%%%%%%%%%%%%%%%%%%%%%%%%%%%%%%%%%%%%%%%%%%%%%%%%
\begin{proof}[Proof of Theorem \ref{thm:mainresult}] To begin with, note that by Condition B(iii) we have that  $r_n^2/\xi_{\min}^2=o(s\log p/n)^{1/k}.$ We begin by proving Part (i) of this theorem. For this purpose, first consider the case when $N=0.$ In this in this case $\{h_1,....,h_N\}=\emptyset,$ and thus by construction, the sequences $u_n,$ $v_n$ play no role in the set $\cG(\cK):=\cG(u_n,v_n,\cK).$ Now, applying Part (i) of Lemma \ref{lem:ulowerbound}, we obtain,
	\benr
	\inf_{\tau\in\cG}\cU(\check N,\h\al,\tau)\ge \mu|\cK|-c_uc_m|\cK|r_n^2- c_uc_m|\cK|\sqrt{\frac{s\log p}{n}}r_n,\nn
	\eenr
	Now, let if possible $\cK$ be non-empty. Then by the choice of $\mu=c_uc_m\rho(s\log p/n)^{1/k^*},$ where $k^*=\max\{k,2\},$ and $n$ sufficiently large, we have that, $\inf_{\tau\in\cG}\cU(\check N,\h\al,\tau)>0.$ This implies that the optimizer $\h\tau\in \bar\R^{\check N}$ of {\bf Step 1} of Algorithm 1, cannot lie in the set $\cG(\cK),$ for any non-empty set $\cK,$ with probability at least $1-c_1N\exp(-c_2\log p).$ Thus the only remaining possibility is that $\h\tau\in\bar\R^{\check N}$ is such that $\tau_{j-1}=\tau_j,$ for all $l=1,...,\check N,$ i.e., $\h\tau_j=-\iny,$ $j=1,...,N.$ This directly implies that $\h\cT(\h\tau)=\emptyset,$ and consequently $\tilde N=0,$ with probability at least $1-c_1N\exp(-c_2\log p).$ Thus proving the theorem for this case.
	
	Next consider the case $N\ge 1.$ Since the optimization of {\bf Step 1} of Algorithm 1 is over a subset of $\tau\in\bar\R^{\check N},$ therefore any such $\tau$ must satisfy $0\le \sum_{j=1}^{N}d(\tau_{h_j},\tau_j^0)\le u_{n}=N,$ consequently, $\tau\in\cG:=\cG(N,0,\cK),$ for some $\cK\subseteq \cT^c.$ Let $v_{n}\ge Ns\log p/n$ be any positive sequence, then applying Part (ii) of Lemma \ref{lem:ulowerbound} over the collection $\cG(N,v_n,\cK),$ yields the bound,
	\benr
	\inf_{\tau\in\cG}\cU(\check N,\h\al,\tau)&\ge& c_uc_mv_{n}+\mu|\cK|-c_uc_mN\frac{\rho^2s\log p}{n}-\frac{c_uc_m}{(1\vee\xi_{\min}^2)}|\cK|r_n^2\nn\\
	&&-N\frac{c_uc_m}{(1\vee\xi_{\min}^2)}r_n^2- N\frac{c_uc_m\rho}{(1\vee\xi_{\min})}\sqrt{\frac{s\log p}{n}}\nn\\
	&&- \frac{c_uc_m}{(1\vee\xi_{\min}^2)}|\cK|\sqrt{\frac{s\log p}{n}}r_n-  \frac{N\mu}{(1\vee\xi_{\min}^2)},\nn
	\eenr
	with probability at least $1-c_1N\exp(-c_2\log p).$ Now, if we choose $v_n:=v_n^*= c_uc_mN\rho(s\log p/n)^{1/k^*}.$ Then for $n$ sufficiently large we have, $\inf_{\tau\in\cG}\cU^*(\check N,\h\al,\tau)>0.$ This implies that the optimizer $\h\tau$ cannot lie in the set $\cG(N,v_n^*,\cK),$ and thus $\h\tau\in\cG(v_n^*,0,\cK),$ for some $\cK.$ This statement together with Condition C(ii) also implies that all $\h\tau_{h_j}$'s are finite and distinct, thereby implying that $\tilde N\ge N.$ Now, for any non empty $\cK,$ reset $u_n=v_n^*$ and apply Part (ii) of Lemma \ref{lem:ulowerbound} over the collection $\cG(u_n,0,\cK).$ Noting that in this case $F(u_n)=0,$ we obtain,
	\benr	
	\inf_{\tau\in\cG}\cU(\check N,\h\al,\tau)&\ge& \mu|\cK|-c_uc_mN\frac{\rho^2s\log p}{n}-\frac{c_uc_m}{(1\vee\xi_{\min}^2)}|\cK|r_n^2\nn\\
	&&-\frac{c_uc_m}{(1\vee\xi_{\min}^2)}r_n^2u_{n}- \frac{c_uc_m\rho}{(1\vee\xi_{\min})}\sqrt{\frac{s\log p}{n}}\sqrt{Nu_n}\nn\\
	&&- \frac{c_uc_m}{(1\vee\xi_{\min}^2)}|\cK|\sqrt{\frac{s\log p}{n}}r_n,\nn
	\eenr
	Under the choice $\mu=c_uc_m\rho(s\log p/n)^{1/k^*},$ we obtain that $\inf_{\tau\in\cG}\cU^*(\check N,\h\al,\tau)>0,$ for any non-empty set $\cK.$ Consequently implying that $\h\tau\in \cG(u_n,0,\emptyset).$ In other words, there are no finite and distinct interruptions between $\h\tau_{h_j}$'s,  consequently $\tilde N= N,$ with probability at least  $1-c_1N\exp(-c_2\log p).$ This proves Part (i) of this theorem.
	
	The proof of part (ii) relies on applying the above argument to recursively tighten the bound for $\h\tau.$ We have already shown that $\h\tau\in\cG(u_n,0,\emptyset).$ Applying the same lower bound over the collection $\cG(u_n,v_n,\emptyset)$ we obtain,
	\benr
	\inf_{\tau\in\cG}\cU(\check N,\h\al,\tau)&\ge& c_uc_mv_{n}-c_uc_mN\frac{\rho^2s\log p}{n}-\frac{c_uc_m}{(1\vee\xi_{\min}^2)}r_n^2u_{n}\nn\\
	&&- \frac{c_uc_m\rho}{(1\vee\xi_{\min})}\sqrt{\frac{s\log p}{n}}\sqrt{Nu_n}\nn
	\eenr
	Now, upon choosing,
	\benr
	v_n\ge v_n^*:=c_uc_mN\rho^{1+\frac{1}{2}}\Big(\frac{s\log p}{n}\Big)^{a_2},\,\,\,{\rm with}\,\,a_2=\min\big\{\frac{1}{2}+\frac{1}{2k^*},\frac{1}{k^*}+\frac{1}{k^*}\big\}\nn
	\eenr
	with, we obtain that for $n$ large, $\inf_{\tau\in\cG}\cU^*(\check N,\h\al,\tau)>0.$ Thus implying that $\h\tau\in \cG(v_n^*,0,\emptyset),$ i.e., $\sum_{j=1}^Nd(\h\tau_{h_j},\tau^0_j)\le v_n^*,$ with probability at least $1-c_1N\exp(-c_2\log p).$ Note that, by using the above recursive argument we have tightened the desired rate at each step. Continuing these recursions, by resetting $u_n$ to the bound of the previous recursion, and applying Part (ii) of Lemma \ref{lem:ulowerbound} over the collection $\cG(u_n,v_n,\emptyset),$ we can obtain for the $m^{th}$ recursion that,
	\benr
	\sum_{j=1}^Nd(\h\tau_{h_j},\tau^0_j)\le c_uc_m \rho^{b_m}\Big(\frac{s\log p}{n}\Big)^{a_m},\quad{\rm where,}\nn\\
	a_m=\min\Big\{\frac{1}{2}+\frac{a_{m-1}}{2},\,\frac{1}{k^*}+a_{m-1}\Big\},\,\,{\rm and}\,\, b_m=1+\frac{b_{m-1}}{2}\hspace{-1cm},\nn
	\eenr
	additionally $a_1=1/k^*$ and $b_1=1.$ To finish the proof, note that if we continue the above recursions an infinite number of times, we obtain $a_{\iny}=\sum_{m=1}^{\iny}1/2^m=1$ and $b_{\iny}=1+\sum_{m=1}^{\iny}1/2^m=2.$ 	Note that, despite the recursions in the above argument, the probability of the bound obtained after every recursion is maintained to be at least $1-c_1N\exp(-c_2\log p),$ this follows from Remark \r{contain}.  This completes the proof of this theorem. 	
\end{proof}
%%%%%%%%%%%%%%%%%%%%%%%%%%%%%%%%%%%%%%%%%%%%%%%%%%%%%%%%%%%%%%%%%%%%%%%%%%%%%%%%%%%%%%%%%%%%%%%%%%%%%%%%%%%%%%%%%%%%%%%%%%%%%%%%%%%%%%%%%%%%%%%%%%%%%%%%%%%%%%%%%%%%%%%%%%%%%	

\vspace{2mm}
%%%%%%%%%%%%%%%%%%%%%%%%%%%%%%%%%%%%%%%%%%%%%%%%%%%%%%%%%%%%%%%%%%%%%%%%%%%%%%%%%%%%%%%%%%%%%%%%%%%%%%%%%%%%%%%%%%%%%%%%%%%%%%%%%%%%%%%%%%%%%%%%%%%%%%%%%%%%%%%%%%%%%%%%%%%%
\begin{rem}\lel{contain} (Observation utilized in the proof of Theorem \ref{thm:mainresult}): The proof of Theorem \ref{thm:mainresult} relies on recursive application of Lemma \ref{lem:ulowerbound}. This in turn requires recursive application of the bounds of Lemma \ref{lem:ulowersupportargument}, the probability of all bounds holding simultaneously at each recursion being at least $1-c_1N\exp(-c_2\log p).$ Despite these recursions (potentially infinite) the result from the final recursion continues to hold with probability at least $1-c_1N\exp(-c_2\log p).$ To see this, let $u_n\to 0$ be any positive sequence and let $\{a_j\}\to a_{\iny},$ $j\to\iny,$ $0<a_j\le 1,$ be any strictly increasing sequence over $j=1,2,....$ . Then define sequences $u^j_n=u_n^{a_j},$ $j=1,2...$ . Here note that  $u_n^{j+1}=o(u_n^j),$ $j=1,...,$ i.e., each sequence converges to zero faster than the preceding one. Let $\cE_{u^1},\cE_{u^2}...$ be events, each with probability $1-c_1N\exp(-c_2\log p),$ on which the upper bounds of Lemma \ref{lem:ulowerbound} hold for each $u^1_n,u^2_n,...$ respectively. Clearly, on the intersection of events $\cE_{u^1}\cap\cE_{u^2}\cap....,$ all upper bounds of Lemma \ref{lem:ulowersupportargument} hold simultaneously over any sequence $u_n^j,$ $j=1,...,\iny$ Now, note that by the construction of these sequences, and that these are all upper bounds, the following containment holds $\cE_{u^1}\supseteq\cE_{u^2}\supseteq...\supseteq\cE_{u^\iny}.$ This implies that on the event $\cE_{u^\iny}$ all bounds of Lemma \ref{lem:ulowersupportargument} hold simultaneously for any sequence $\{u_n^{j}\},$ $j=1,...,\iny.$ Here $\cE_{u^\iny}$ represents the set corresponding to the sequence $u_{n}^{\iny}=u_n^{a_{\iny}}.$ Also, by a single application of Lemma \ref{lem:ulowersupportargument}, $P(\cE_{u^\iny})\ge 1-c_1\exp(-c_2\log p).$ The same argument can be made for the lower bound of Lemma \ref{lem:ulowersupportargument}, with the direction of the containment switched.
\end{rem}
%%%%%%%%%%%%%%%%%%%%%%%%%%%%%%%%%%%%%%%%%%%%%%%%%%%%%%%%%%%%%%%%%%%%%%%%%%%%%%%%%%%%%%%%%%%%%%%%%%%%%%%%%%%%%%%%%%%%%%%%%%%%%%%%%%%%%%%%%%%%%%%%%%%%%%%%%%%%%%%%%%%%%%%%%%%%

\vspace{2mm}
%%%%%%%%%%%%%%%%%%%%%%%%%%%%%%%%%%%%%%%%%%%%%%%%%%%%%%%%%%%%%%%%%%%%%%%%%%%%%%%%%%%%%%%%%%%%%%%%%%%%%%%%%%%%
\begin{proof}[Proof of Corollary \ref{cor:tildealpha}] First, note that from the result of Theorem \ref{thm:mainresult}, we have that $\tilde N=N$ and $\sum_{j=1}^{N}d(\tilde\tau_j,\tau^0_j)\le N\rho^2s\log p\big/n,$ with probability at least $1-c_1N\exp(-c_2\log p),$ for $n$ sufficiently large. All arguments to follow are restricted to the event where these two results hold. Now by construction of Algorithm 2, the regression estimates $\tilde\al_{(j)},$ $j=0,...,N$ are computed based on the partition yielded by the change point estimate $\tilde \tau.$ Let $u_{nj}:=|\tilde\tau_{j}-\tau^0_j|\vee|\tilde\tau_{j+1}-\tau^0_{j+1}|$ Then,  choosing $\la_{1j}=c_uc_m\max\{\sqrt{\log p/n},\,\xi_{\max}u_{nj}\},$ and applying Theorem \ref{thm:initialrate}, we obtain for each $j=0,...,N,$ that,
	\benr\label{eq:finreg}
	\|\tilde\al_{(j)}-\b_{(j)}^0\|\le c_uc_ms^{\frac{1}{q}}\max\Big\{\sqrt{\frac{\log p}{n}},\,\xi_{\max}u_{nj}\Big\}\Big/l_{\min},
	\eenr
	with probability at least $1-c_1\exp(-c_2\log p).$  Again, by Theorem \ref{thm:mainresult} we have that $\sum_{j=1}^nu_{nj}\le c_uc_m N\rho^2 s\log p/n,$ with probability at least $1-c_1(1\vee N)\exp(-c_2\log p).$ Thus, summing up the bounds in (\ref{eq:finreg}) over $j=0,...,N$ we obtain the statement of the Corollary. 	
\end{proof}	
%%%%%%%%%%%%%%%%%%%%%%%%%%%%%%%%%%%%%%%%%%%%%%%%%%%%%%%%%%%%%%%%%%%%%%%%%%%%%%%%%%%%%%%%%%%%%%%%%%%%%%%%%%%%	

%%%%%%%%%%%%%%%%%%%%%%%%%%%%%%%%%%%%%%%%%%%%%%%%%%%%%%%%%%%%%%%%%%%%%%%%%%%%%%%%%%%%%%%%%%%%%%%%%%%%%%%%%%%%%%%%%%%%%%%%%%%%%%%%%%%%%%%%%%%%%%%%%%%%%%%%%%%%%%%%%%%%%%%%%%%%

\section{Auxiliary results}

%\begin{center}
%	{\sc Supplementary Material B for ``Arbitrary segmentation: A method for detection and estimation of parameters in high dimensional change point regression models via $\ell_0$ regularization"}
%\end{center}

\vspace{2mm}
%%%%%%%%%%%%%%%%%%%%%%%%%%%%%%%%%%%%%%%%%%%%%%%%%%%%%%%%%%%%%%%%%%%%%%%%%%%%%%%%%%%%%%%%%%%%%%%%%%%%%%%%%%%%%%%%%%%%%%%%%%%%%%%%%%%%%%%%%%%%%%%%%%%%%%%%%%%%%%%%%%%%%%%%%%%%
\begin{lem}\label{lem:crossbounds} Suppose Condition D and let $u_n$ be any non-negative sequence satisfying $\log (u_n^{-1})=O(\log p).$ Then we have for any fixed $\delta\in\R^p$ that,
	\benr
	&(i)& \sup_{\substack{\tau_a,\tau_b\in\bar\R;\\d(\tau_a,\tau_b)\le u_n}}\Big\|\frac{1}{n}\sum_{i\in n^w(\tau_a,\tau_b)} \delta^T x_ix_i^T \Big\|_{\iny} \le c_u c_{m}\|\delta\|_2\max\Big\{\frac{\log p}{n}, u_n\Big\},\nn\\
	&(ii)& \sup_{\substack{\tau_a,\tau_b\in\bar\R;\\d(\tau_a,\tau_b)\le u_n}}\frac{1}{n}\sum_{i\in n^w(\tau_a,\tau_b)} \delta^T x_ix_i^T \delta \le c_u c_{m} \|\delta\|_2^2\max\Big\{\frac{\log p}{n}, u_n\Big\},\nn\\
	&(iii)& \sup_{\substack{\tau_a,\tau_b\in\bar\R;\\d(\tau_a,\tau_b)\le u_n}}\frac{1}{n}\big\|\sum_{i\in n^w(\tau_a,\tau_b)} \vep_ix_i^T \big\|_{\iny} \le c_u c_{m}\sqrt{\frac{\log p}{n}}\max\Big\{\sqrt{\frac{\log p}{n}}, \sqrt{u_n}\Big\},\nn
	\eenr
	with probability at least $1- c_1\exp(-c_2\log p).$
\end{lem}
%%%%%%%%%%%%%%%%%%%%%%%%%%%%%%%%%%%%%%%%%%%%%%%%%%%%%%%%%%%%%%%%%%%%%%%%%%%%%%%%%%%%%%%%%%%%%%%%%%%%%%%%%%%%%%%%%%%%%%%%%%%%%%%%%%%%%%%%%%%%%%%%%%%%%%%%%%%%%%%%%%%%%%%%%%%%

%%%%%%%%%%%%%%%%%%%%%%%%%%%%%%%%%%%%%%%%%%%%%%%%%%%%%%%%%%%%%%%%%%%%%%%%%%%%%%%%%%%%%%%%%%%%%%%%%%%%%%%%%%%%%%%%%%%%%%%%%%%%%%%%%%%%%%%%%%%%%%%%%%%%%%%%%%%%%%%%%%%%%%%%%%%%
\vspace{2mm}
\begin{proof}[Proof of Lemma \ref{lem:crossbounds}]
	We begin with the proof of Part (i). Note that the RHS of the inequality in Part (i) is normalized by the $\ell_2$ norm of $\delta.$ Hence, without loss of generality we can assume $\|\delta\|_2=1.$  In following denote $n^w=n^w(\tau_a,\tau_b).$ Note that if $|n^w|=0$ then Lemma \ref{lem:crossbounds} holds trivially with probability $1,$ thus without loss of generality we shall assume that $|n_w|>0.$ Now, for any fixed $\tau_a,\tau_b\in\bar\R,$ we have
	\benr\label{eq:equality3}
	\Big\|\frac{1}{n}\sum_{i\in n_w} \delta^T x_ix_i^T \Big\|_{\iny} \le \frac{|n_w|}{n}\Big\|\frac{1}{|n_w|}\sum_{i\in n_w} \delta^T x_ix_i^T \Big\|_{\iny}
	\eenr
	Under Condition D(iv) and by properties of conditional expectations (see e.g. Lemma \ref{lem:durrett}), the conditional probability $P_w(\cdotp)= P(\cdotp| w)$ can be bounded by treating $w$ as a constant. Thus,
	\benr
	P_w\Big(\Big\|\frac{\sum_{i\in n_w} \delta^Tx_ix_i^T}{|n_w|}- \delta^T\Sigma\Big\|_{\iny} > t\Big)\le 6p \exp(-c_u|n_w|\min\big\{\frac{t^2}{\si_x^4}, \frac{t}{\si_x^2}\big\})\nn
	\eenr
	where the above probability bound is obtained by an application of Part (ii) of Lemma 14 of Loh and Wainwright (2012): supplementary materials. This lemma is reproduced as Lemma \ref{lem:lwstochb} in this section. Now choosing $t=c_u\max\Big\{\si_x^2\sqrt{\frac{\log p}{|n_w|}}, \si_x\frac{\log p}{|n_w|}\Big\}$ we obtain,
	\benr\label{eq:cond1}
	P_w\left(\Big\|\frac{\sum_{i\in n_w} \delta^Tx_ix_i^T}{|n_w|}\Big\|_{\iny}\le \|\delta^T\Sigma\|_{\iny} + c_u\max\Big\{\si_x^2\sqrt{\frac{\log p}{|n_w|}}, \si_x\frac{\log p}{|n_w|}\Big\}\right)\nn\\
	\ge 1-c_1\exp(-c_2\log p).
	\eenr
	The result in (\ref{eq:cond1}) together with (\ref{eq:equality3}) yields,
	\benr\label{eq:cond2}
	P_w\left(\Big\|\frac{1}{n}\sum_{i\in n_w} \delta^T x_ix_i^T \Big\|_{\iny} \le \frac{|n_w|}{n}\|\delta^T\Si\|_{\iny}+\frac{|n_w|}{n}c_u\max\Big\{\si_x^2\sqrt{\frac{\log p}{|n_w|}}, \si_x\frac{\log p}{|n_w|}\Big\}\right)\nn\\
	\ge 1-c_1\exp(-c_2\log p).\nn
	\eenr
	Taking expectations on both sides and observing that the RHS of the above conditional probability is free of $w,$ we obtain,
	\benr\label{eq:cond3}
	P\left(\Big\|\frac{1}{n}\sum_{ i\in n_w} \delta^T x_ix_i^T \Big\|_{\iny} \le \frac{|n_w|}{n}\|\delta^T\Si\|_{\iny}+\frac{|n_w|}{n}c_u\max\Big\{\si_x^2\sqrt{\frac{\log p}{|n_w|}}, \si_x\frac{\log p}{|n_w|}\Big\}\right)\hspace{-1cm}\nn\\
	\ge 1-c_1\exp(-c_2\log p)
	\eenr
	On the other hand, we have by Part (i) of Lemma \ref{lem:indicatorbound}, that with probability at least $1-c_1\exp(-c_2\log p)$ that $\sup_{d(\tau_a,\tau_b)\le u_n}|n_w|/n\le c_u\max\{\log p/n , u_n\}.$ Also, it is straightforward to see that $\|\delta^T\Si\|_{\iny}\le c_u \phi,$ for some constant $c_u>0.$ Thus with the same probability we have the bound,
	\benr\lel{eq:l2e2}
	\sup_{\tau\in\cT(\tau_{0n},u_n)}\frac{|n_w|}{n}\|\delta^T\Si\|_{\iny}\le c_u\phi\max\Big\{c_a\frac{\log p}{n}, u_n\Big\}.
	\eenr
	Again applying Part (i) of Lemma \ref{lem:indicatorbound} we also have the following bound with probability at least $1-c_1\exp(-c_2\log p),$
	\benr\lel{eq:l2e3}
	\sup_{\substack{\tau_a,\tau_b\in\bar\R;\\d(\tau_a,\tau_b)\le u_n}}\frac{|n_w|}{n}\sqrt{\frac{\log p}{|n_w|}}&\le& c_u\sqrt{\frac{\log p}{n}} \max\Big\{\sqrt{\frac{\log p}{n}}, \sqrt{u_n}\Big\}\nn\\
	&\le& c_u\max\Big\{\frac{\log p}{n}, u_n\Big\}.
	\eenr
	The final inequality follows upon noting that if $\sqrt{\log p/n}\sqrt{u_n}\ge u_n $ then $u_n \le \log p/n.$ Finally also note that $\sup_{d(\tau_a,\tau_b)\le u_n} (|n_w|/n) (\log p/|n_w|)\le \log p/n.$ Part (i) of the lemma follows by combining these results together with the bounds (\ref{eq:l2e2}) and (\ref{eq:l2e3}) in (\ref{eq:cond3}). The proofs of Part (ii) and Part (iii) are similar and are thus omitted.
\end{proof}
%%%%%%%%%%%%%%%%%%%%%%%%%%%%%%%%%%%%%%%%%%%%%%%%%%%%%%%%%%%%%%%%%%%%%%%%%%%%%%%%%%%%%%%%%%%%%%%%%%%%%%%%%%%%%%%%%%%%%%%%%%%%%%%%%%%%%%%%%%%%%%%%%%%%%%%%%%%%%%%%%%%%%%%%%%%%

%%%%%%%%%%%%%%%%%%%%%%%%%%%%%%%%%%%%%%%%%%%%%%%%%%%%%%%%%%%%%%%%%%%%%%%%%%%%%%%%%%%%%%%%%%%%%%%%%%%%%%%%%%%%%%%%%%%%%%%%%%%%%%%%%%%%%%%%%%%%%%%%%%%%%%%%%%%%%%%%%%%%%%%%%%%%
\begin{lem} \label{lem:ulowersupportargument}
	(Bounds used in the proof of Lemma \ref{lem:ulowerbound}): Let $\h\al_{(j)},$ $j=0,...,\check N$ be the regression estimates obtained from {\bf Step 1} of Algorithm 1, $\cT$ and $\cT^*$ be as defined in Condition A and (\ref{def:settstar}) respectively and let $\cG:=\cG(u_n,v_n,\cK)$ be as defined in (\ref{def:rtset}). Then assuming the conditions of Lemma \ref{lem:ulowerbound}, the following bounds hold with probability at least $1-c_1(1\vee N)\exp(-c_2\log p),$ for $n$ sufficiently large.
	\benr
	&(i)&\inf_{\tau\in\cG}\frac{1}{n\xi_{\min}^2}\sum_{j=1}^N\sum_{i\in n^w(\tau_{h_j},\tau_j^0)}(\h\al_{(h_j-1)}-\b^0_{(j)})^Tx_ix_i^T(\h \al_{(h_j-1)}-\b^0_{(j)})\ge\nn\\ &&\hspace{3.2in}c_uc_mv_{n}-c_uc_mN\frac{\rho^2s\log p}{n},\nn\\
	&(ii)&\inf_{\tau\in\cG}\frac{1}{n}\sum_{j=1}^{N+1}\sum_{m_{j-1}<l<h_j}\sum_{\tau_{l-1}<w_i<\tau_l}(\h\al_{(l-1)}-\b^0_{(j-1)})^Tx_ix_i^T(\h\al_{(l-1)}-\b^0_{(j-1)}) \ge 0,\nn\\
	&(iii)& \sup_{\tau\in\cG}\frac{1}{n}\sum_{j=1}^N\sum_{i\in n^w(\tau_j^0,\tau_{h_j})}(\h\al_{(h_{j+1}-1)}-\b^0_{(j)})^Tx_ix_i^T(\h\al_{(h_{j+1}-1)}-\b^0_{(j)}) \le \nn\\
	&&\hspace{3.2in}c_uc_mr_n^2 \max\Big\{\frac{Ns\log p}{n}, u_{n}\Big\},\nn\\
	&(iv)&  \sup_{\tau\in\cG}\frac{1}{n}\sum_{j=1}^{N+1}\sum_{m_{j-1}<l<\tau_l}\sum_{\tau_{l-1}<w_i<\tau_l}(\h\al_{(h_{j}-1)}-\b^0_{(j-1)})^Tx_ix_i^T(\h\al_{(h_{j}-1)}-\b^0_{(j-1)}) \le c_uc_m|\cK|r_n^2,\nn\\
	&(v)& \sup_{\tau\in\cG}\frac{2}{n}\sum_{j=1}^{N+1}\sum_{m_{j-1}<l<\tau_l}\sum_{\tau_{l-1}<w_i<\tau_l}\vep_ix_i^T(\h\al_{(l-1)}-\b^0_{(j-1)}) \le c_uc_m |\cK|\sqrt{\frac{s\log p}{n}}r_n,\nn\\
	&(vi)& \sup_{\tau\in\cG}\frac{2}{n}\sum_{j=1}^N\sum_{i\in n^w(\tau_j^0,\tau_{h_j})}\vep_ix_i^T(\h\al_{(h_{j}-1)}-\b^0_{(j)}) \le c_uc_m \xi_{\max}\sqrt{\frac{\log p}{n}}\max\Big\{N\sqrt{\frac{\log p}{n}}, \sqrt{Nu_{n}}\Big\},\nn\\
	&(vii)& \sup_{\tau\in\cG}\frac{2}{n}\sum_{j=1}^N\sum_{i\in n^w(\tau_j^0,\tau_{h_j})}\vep_ix_i^T(\h\al_{(h_{j+1}-1)}-\b^0_{(j)})  \le c_uc_m r_n\sqrt{\frac{\log p}{n}}\max\Big\{N\sqrt{\frac{\log p}{n}}, \sqrt{Nu_{n}}\Big\},\nn\\
	&(viii)& \sup_{\tau\in\cG}\frac{2}{n}\sum_{j=1}^{N+1}\sum_{m_{j-1}<l<\tau_l}\sum_{\tau_{l-1}<w_i<\tau_l}\vep_ix_i^T(\h\al_{(h_j-1)}-\b^0_{(j-1)}) \le c_uc_m|\cK|\sqrt{\frac{s\log p}{n}}r_n. \nn
	\eenr
\end{lem}
%%%%%%%%%%%%%%%%%%%%%%%%%%%%%%%%%%%%%%%%%%%%%%%%%%%%%%%%%%%%%%%%%%%%%%%%%%%%%%%%%%%%%%%%%%%%%%%%%%%%%%%%%%%%%%%%%%%%%%%%%%%%%%%%%%%%%%%%%%%%%%%%%%%%%%%%%%%%%%%%%%%%%%%%%%%%

%%%%%%%%%%%%%%%%%%%%%%%%%%%%%%%%%%%%%%%%%%%%%%%%%%%%%%%%%%%%%%%%%%%%%%%%%%%%%%%%%%%%%%%%%%%%%%%%%%%%%%%%%%%%%%%%%%%%%%%%%%%%%%%%%%%%%%%%%%%%%%%%%%%%%%%%%%%%%%%%%%%%%%%%%%%%
\vspace{2mm}
\begin{proof}[Proof of Lemma \ref{lem:ulowersupportargument}] To prove part (i), let $v_{nj}\ge s\log p/n,$ $j=1,...,N$ and $v_n=\sum_{j=1}^N v_{nj}\ge Ns\log p/n,$ then by Part (ii) of \ref{lem:restrictedeigen} we have with probability at least $1-c_1\exp(-c_2\log p)$ that
	\benr
	\inf_{\tau_j;\, d(\tau_j,\tau_j^0)\ge v_{nj}}\inf_{\delta_{(j)}\in\cA_2}\frac{1}{n}\sum_{i\in n^w(\tau_j,\tau_j^0)}\delta_{(j)}^Tx_ix_i^T\delta_{(j)}\ge c_uc_mv_{nj}\|\delta_{(j)}\|_2^2-c_uc_m\frac{\xi_{\max}^2s\log p}{n}.\nn
	\eenr
	Applying this bound for each $j=1,...,N,$ and summing them up, we obtain with probability at least $1-c_1N\exp(-c_2\log p),$
	\benr\label{eq:fullunifb}
	\frac{1}{n}\sum_{j=1}^N\inf_{\tau_j;\, d(\tau_j,\tau_j^0)\ge v_{nj}}\inf_{\delta_{(j)}\in\cA_2}\sum_{i\in n^w(\tau_j,\tau_j^0)}\delta_{(j)}^Tx_ix_i^T\delta_{(j)}\ge\hspace{1.5cm}\nn\\ c_uc_mv_{n}\min_j\|\delta_{(j)}\|_2^2-c_uc_mN\frac{\xi_{\max}^2s\log p}{n}.\hspace{-1.5cm}
	\eenr
	Now let $\delta_{(j)}=(\h \al_{(h_j-1)}-\b^0_{(j)}).$ By the construction of the indices $h_j$ of the index set $\cT^*=\{h_0,h_1,...,h_{N+1}\},$ in (\ref{def:taustar}) we have that $m_{j-1}< h_j\le m_j.$ Consequently, from the proof of Theorem \ref{thm:initialrate} we have that $(\h\al_{(h_j-1)}-\b^0_{(j-1)})\in\A,$ $j=1,...,N$ with probability at least $1-c_1N\exp(-c_2\log p).$ This in turn implies that $(\h\al_{(h_j-1)}-\b^0_{(j)})\in\A_2,$ $j=1,...,N$ with the same probability. Additionally, we have for any $j=1,...,N,$
	\benr\label{eq:cauchyexp}
	\|\delta_{(j)}\|_2^2&=&\|(\b^0_{(j)}-\b^0_{(j-1)})+\h\al_{(h_j-1)}-\b^0_{(j-1)}\|_2^2= \|(\b^0_{(j)}-\b^0_{(j-1)})\|_2^2\nn\\
	&&+\|\h\al_{(h_j-1)}-\b^0_{(j-1)}\|_2^2+\|\b^0_{(j)}-\b^0_{(j-1)}\|_2\|\h\al_{(h_j-1)}-\b^0_{(j-1)}\|_2\nn\\
	&\ge& \xi_{\min}^2- r_n^2-2\xi_{\max}r_n
	\eenr
	with probability at least $1-c_1N\exp(-c_2\log p).$ Applying Condition B(iii) we obtain with the same probability that $\min_{j}\|\delta_{(j)}\|_2^2\big/\xi_{\min}^2\ge 1,$ for $n$ sufficiently large. Substituting these results back in \ref{eq:fullunifb} we obtain,
	\benr\label{eq:unifbtau}
	\frac{1}{n\xi_{\min}^2}\sum_{j=1}^N\inf_{\tau_j;\, d(\tau_j,\tau_j^0)\ge v_{nj}}\sum_{i\in n^w(\tau_j,\tau_j^0)}(\h \al_{(h_j-1)}-\b^0_{(j)})^Tx_ix_i^T(\h \al_{(h_j-1)}-\b^0_{(j)})\ge\hspace{0.25cm}\nn\\ c_uc_mv_{n}-c_uc_mN\frac{\rho^2s\log p}{n}.\hspace{-0.2cm}
	\eenr
	with probability at least $1-c_1N\exp(-c_2\log p),$ and for $n$ sufficiently large. Now, recall the collection $\cG(u_n,v_n,\cK),$ defined for any $v_n\ge Ns\log p/n.$ Note that the sequence $u_n$ and the set $\cK$ are irrelevant for this bound, and by definition of this set we have that $\sum_{j=1}^{N}d(\tau_{h_j},\tau_j^0)\ge v_n\ge c_uNs\log p/n.$ In the case where $v_{nj}\ge c_us\log p/n,$ for each $j=1,...,N,$ clearly, the infimum on the LHS of (\ref{eq:unifbtau}) can be directly replaced with an infimum over the collection $\cG(u_n,v_n,\cK),$ with the corresponding expressions evaluated at $\tau_{h_j}'$s in place of $\tau_j$'s. This follows since the replacement infimum is over a subset of that in (\ref{eq:unifbtau}). In the case where $v_{nj}=o(s\log p/n)$ for one or more $j$'s (W.L.O.G. assume $j=1$). Since this component is of smaller order than $v_n,$ consequently we shall still have that $v_{n,-j}:=\sum_{j\ne1} v_{nj}\ge c_uNs\log p/n,$ for $n$ large, i.e., the ratio of $v_n/v_{n,-j}=O(1).$ Thus by applying all above arguments to only the components $j=1,...,N,$ where $v_{nj}\ge c_us\log p/n,$ we obtain,
	\benr
	\frac{1}{n\xi_{\min}^2}\inf_{\tau\in\cG(u_n,v_n,\cK)}\sum_{j=1}^N\sum_{i\in n^w(\tau_{h_j},\tau_j^0)}(\h\al_{(h_j-1)}-\b^0_{(j)})^Tx_ix_i^T(\h \al_{(h_j-1)}-\b^0_{(j)})\ge\hspace{0.25cm}\nn\\ c_uc_mv_{n}-c_uc_mN\frac{\rho^2s\log p}{n}.\hspace{-0.2cm}\nn
	\eenr
	with probability $1-c_1(1\vee N)\exp(-c_2\log p),$ and for $n$ sufficiently large. This completes the proof of Part (i) of this lemma. The bound $R2\ge 0$ is trivial since it is a quadratic term. The bounds for $R3$ and $R4$ follow directly by an application of Part (iii) of (\ref{lem:restrictedeigen}). The bound for $R6,$ $R7$ and $R8$ can be obtained by an application of Lemma \ref{lem:crossbounds}. This completes the proof of the Lemma.
\end{proof}
%%%%%%%%%%%%%%%%%%%%%%%%%%%%%%%%%%%%%%%%%%%%%%%%%%%%%%%%%%%%%%%%%%%%%%%%%%%%%%%%%%%%%%%%%%%%%%%%%%%%%%%%%%%%%%%%%%%%%%%%%%%%%%%%%%%%%%%%%%%%%%%%%%%%%%%%%%%%%%%%%%%%%%%%%%%%

\vspace{2mm}
%%%%%%%%%%%%%%%%%%%%%%%%%%%%%%%%%%%%%%%%%%%%%%%%%%%%%%%%%%%%%%%%%%%%%%%%%%%%%%%%%%%%%%%%%%%%%%%%%%%%%%%%%%%%%%%%%%%%%%%%%%%%%%%%%%%%%%%%%%%%%%%%%%%%%%%%%%%%%%%%%%%%%%%%%%%%	
\begin{lem}\label{lem:maurer} Let the $\{X_i\}_{i=1}^m$ be independent random variables, $EX_i^2<\iny,$ $X_i\ge 0.$ Set $S=\sti X_i$ and let $t>0.$ Then
	\benr
	P\Big(ES-S\ge t\Big)\le \exp\Big(\frac{-t^2}{2\sti EX_i^2}\Big)\nn
	\eenr
\end{lem}
This result is as stated in Theorem 1 of Maurer (2003), it provides a lower bound on a sum of positive independent r.v.'s.
%%%%%%%%%%%%%%%%%%%%%%%%%%%%%%%%%%%%%%%%%%%%%%%%%%%%%%%%%%%%%%%%%%%%%%%%%%%%%%%%%%%%%%%%%%%%%%%%%%%%%%%%%%%%%%%%%%%%%%%%%%%%%%%%%%%%%%%%%%%%%%%%%%%%%%%%%%%%%%%%%%%%%%%%%%%%	

\vspace{2mm}
%%%%%%%%%%%%%%%%%%%%%%%%%%%%%%%%%%%%%%%%%%%%%%%%%%%%%%%%%%%%%%%%%%%%%%%%%%%%%%%%%%%%%%%%%%%%%%%%%%%%%%%%%%%%%%%%%%%%%%%%%%%%%%%%%%%%%%%%%%%%%%%%%%%%%%%%%%%%%%%%%%%%%%%%%%%%	
\begin{lem}\label{lem:lw} Let $z_i\in\R^p,$ $i=1,...,n$ be i.i.d subgaussian random vectors with variance parameter $\si_z^2$ and covariance $\Si_z=Ez_iz_i^T.$ Also, let $\la_{\min}(\Si_z)$ and $\la_{\max}(\Si_z)$ be the minimum and maximum eigenvalues of the covariance matrix respectively. Then,
	\benr
	&(i)&\frac{1}{n} \sum_{i=1}^n \delta^Tz_iz_i^T\delta \ge \frac{\la_{\min}(\Si_z)}{2}\|\delta\|_2^2- c_u\la_{\min}(\Si_z)\max\Big\{\frac{\si_z^4}{\la_{\min}^2(\Si_z)},1\Big\} \frac{\log p}{n} \|\delta\|_1^2,\quad \forall \delta\in\R^p, \nn\\
	&(ii)& \frac{1}{n} \sum_{i=1}^n \delta^Tz_iz_i^T\delta \le \frac{3\la_{\max}(\Si_z)}{2}\|\delta\|_2^2+ c_u\la_{\min}(\Si_z)\max\Big\{\frac{\si_z^4}{\la_{\min}^2(\Si_z)},1\Big\} \frac{\log p}{n} \|\delta\|_1^2,\quad \forall \delta\in\R^p, \nn
	\eenr
	with probability at least $1-c_1\exp(-c_2\log p).$
\end{lem}
%%%%%%%%%%%%%%%%%%%%%%%%%%%%%%%%%%%%%%%%%%%%%%%%%%%%%%%%%%%%%%%%%%%%%%%%%%%%%%%%%%%%%%%%%%%%%%%%%%%%%%%%%%%%%%%%%%%%%%%%%%%%%%%%%%%%%%%%%%%%%%%%%%%%%%%%%%%%%%%%%%%%%%%%%%%%	

\vspace{2mm}
%%%%%%%%%%%%%%%%%%%%%%%%%%%%%%%%%%%%%%%%%%%%%%%%%%%%%%%%%%%%%%%%%%%%%%%%%%%%%%%%%%%%%%%%%%%%%%%%%%%%%%%%%%%%%%%%%%%%%%%%%%%%%%%%%%%%%%%%%%%%%%%%%%%%%%%%%%%%%%%%%%%%%%%%%%%%	
\begin{lem}\label{lem:durrett} Suppose $X$ and $Y$ are independent random variables. Let $\phi$ be a function with $E|\phi(X,Y)|<\iny$ and let $g(x)=E\phi(x,Y),$ then
	\benr
	E\big(\phi(X,Y)|X\big)=g(X)\nn
	\eenr
\end{lem}
This is an elementary result on conditional expectations and is stated for the reader's convenience. A straightforward proof can be found in Example 1.5. page 222, \cite{durrett2010probability}.
%%%%%%%%%%%%%%%%%%%%%%%%%%%%%%%%%%%%%%%%%%%%%%%%%%%%%%%%%%%%%%%%%%%%%%%%%%%%%%%%%%%%%%%%%%%%%%%%%%%%%%%%%%%%%%%%%%%%%%%%%%%%%%%%%%%%%%%%%%%%%%%%%%%%%%%%%%%%%%%%%%%%%%%%%%%%	

\vspace{2mm}
%%%%%%%%%%%%%%%%%%%%%%%%%%%%%%%%%%%%%%%%%%%%%%%%%%%%%%%%%%%%%%%%%%%%%%%%%%%%%%%%%%%%%%%%%%%%%%%%%%%%%%%%%%%%%%%%%%%%%%%%%%%%%%%%%%%%%%%%%%%%%%%%%%%%%%%%%%%%%%%%%%%%%%%%%%%%	
\begin{lem}\label{lem:lwstochb} If $X\in \R^{n\times p_1}$ is a zero mean subgaussian matrix with parameters $(\Si_x,\si_x^2)$, then for any fixed (unit) vector in $v\in\R^{p_1},$ we have
	\benr
	(i)\,\, P\Big(\Big|\|Xv\|_2^2-E\|Xv\|_2^2\Big|\ge nt\Big)\le \exp\Big(-cn\min\Big\{\frac{t^2}{\si_x^4},\frac{t}{\si_x^2}\Big\}\Big)\hspace{22mm}\nn
	\eenr
	Moreover, if $Y\in \R^{n\times p_2}$ is a zero mean subgaussian matrix with parameters $(\Si_y,\si_y^2),$ then
	\benr
	(ii)\,\,P\Big(\|\frac{Y^TX}{n}-{\rm cov}(y_i,x_i)\|_{\iny}\ge t\Big) \le 6p_1p_2\exp\Big(-cn \min\Big\{\frac{t^2}{\si_x^2\si_y^2},\frac{t}{\si_x\si_y}\Big\}\Big)\nn
	\eenr
	where $x_i,y_i$ are the $i^{th}$ rows of $X$ and $Y$ respectively. In particular, if $n\ge c\log p,$ then
	\benr
	(iii)\,\,P\Big(\|\frac{Y^TX}{n}-{\rm cov}(y_i,x_i)\|_{\iny}\ge c\si_x\si_y \sqrt{\frac{\log p}{n}}\Big) \le c_1\exp(-c_2\log p).\hspace{14mm} \nn
	\eenr
\end{lem}
%%%%%%%%%%%%%%%%%%%%%%%%%%%%%%%%%%%%%%%%%%%%%%%%%%%%%%%%%%%%%%%%%%%%%%%%%%%%%%%%%%%%%%%%%%%%%%%%%%%%%%%%%%%%%%%%%%%%%%%%%%%%%%%%%%%%%%%%%%%%%%%%%%%%%%%%%%%%%%%%%%%%%%%%%%%%	

This lemma provides tail bounds on subexponential r.v.'s and is as stated in Lemma 14 of \cite{loh2012}: supplementary materials. The first part of this lemma is a restatement of Proposition 5.16 of \cite{vershynin2010introduction} and the other two part are derived via algebraic manipulations of the product under consideration.

\end{document}